{\makeatletter\AtBeginDocument{\def\UTFviii@undefined@err#1{}}}

\documentclass[10pt]{article}
\usepackage{amsmath, amsfonts, amsthm, amssymb, amscd, mathrsfs, a4wide}
\usepackage{cases}
\usepackage{dsfont, pxfonts}
\usepackage{color}  
\usepackage[makeroom]{cancel}
\usepackage[mathcal]{euscript} 
\usepackage[perpage]{footmisc}

\newcommand \mt {\widetilde m}
\newcommand \obig {o}
\newcommand \cbf {\mathbf c} 
\newcommand \Mbf {\mathbf M}
 
\usepackage{graphicx} 

\usepackage{float}

\usepackage[hidelinks]{hyperref} 

 \newcommand \Acal {\mathcal A}

 \newcommand \smallcal {\mathcal N}
 \newcommand \vs {\vskip.15cm}
\newcommand \ocal {{o}}

\newcommand \bse {\begin{subequations}}
\newcommand \ese {\end{subequations}}

\newtheorem{theorem}{Theorem}[section]
\newtheorem{proposition}[theorem]{Proposition}

\newtheorem{lemma}[theorem]{Lemma}
\newtheorem{corollary}[theorem]{Corollary}
\newtheorem{definition}[theorem]{Definition}
\newtheorem{claim}[theorem]{Claim}
\newtheorem{remark}[theorem]{Remark}
\numberwithin{equation}{section}
\newcommand \beginmyremark {\begin{remark} \rm}
\newcommand \beginmylemma {\vskip.16cm \begin{lemma}} 
\newcommand \beginmyproposition {\vskip.16cm \begin{proposition}} 

\newcommand\ee         {\end{equation}}
\newcommand \bei {\begin{itemize}}
\newcommand \eei {\end{itemize}}
\newcommand \be  {\begin{equation}}
\newcommand\bel {\be\label} 
 
\newcommand \gloc {g^\text{loc}}
\newcommand \hloc {h^\text{loc}}
\newcommand \ghat {\widehat g}

\newcommand \pt {\widetilde p} 
\newcommand \Deltas {\cancel \Delta}

\newcommand \del \partial 
\newcommand \pstar {p_*}
\newcommand \qstar {q_*} 
\newcommand \chichart {\chi_{\text{chart}}}
\newcommand \oldp p
\newcommand \oldq q 
\newcommand \eeG {\eps_G}
\newcommand \eeM {\eps_M}
\newcommand \ppG {p_G}
\newcommand \qqG {q_G}
\newcommand \ppM {p_M}
\newcommand \qqM {q_M}
\newcommand \Obig {\mathcal O}
\newcommand \Hstar {H_*} 
\newcommand \Mstar {M_*} 
\newcommand \Gstar {G_*} 
\newcommand \ADM {{\text{ADM}}}
\newcommand \mADM {\mathbf m^\ADM}
\newcommand \PADM {\mathbf P^\ADM}

\newcommand \Mbb {\mathbb M} 
\newcommand \Pss {\mathcal{P}_{\text{seed}}^{\text{sol}}}  
\newcommand \Dcal {\mathcal D}  
\newcommand \Cscr {\mathscr C}
\newcommand \Tscr {\mathscr T}
\newcommand \sgn {\text{sgn}} 
\newcommand \loc {\text{loc}} 
\newcommand \Hess {\mathbf{Hess}}  
\newcommand \Tr {\mathbf{Tr} \hskip.03cm}
\newcommand \Ocal {\mathcal O} 
\newcommand \Jbf {\mathbf J} 
\newcommand \eps \epsilon 

\newcommand{\bs}[1]{\boldsymbol{#1}}

\newcommand \Lcal {\bs{\mathcal{L}}} 
\newcommand \Mcal {\mathcal M}
\newcommand \ut {\widetilde u}
\newcommand \gt {\widetilde g}
\newcommand \xt {\widetilde x}
\newcommand \rt {\widetilde r}
\newcommand \wt {\widetilde w}
\newcommand \htt {\widetilde h}
\newcommand \Zt {\widetilde Z}

\newcommand \Wt {\widetilde W}
\newcommand \Sch {\text{Sch}}
\newcommand \Eucl {\text{Eucl}}
\newcommand \Et {\widetilde E}
\newcommand \Ric {\textbf{Ric}}
\newcommand \Div {\mathbf{Div}}
 
\newcommand \Gcal   {\mathcal G} 
\newcommand \LMcal {\mathcal B}
\newcommand \RR         {\mathbb R}
\newcommand \Pt {\widetilde P}
\newcommand \Ecal   {\mathcal E}
\newcommand \Fcal   {\mathcal F}
\newcommand \Hcal   {\mathcal H}
\newcommand \Tcal   {\mathcal T}
\newcommand \Qcal {\mathcal Q}
\begin{document}

\title{The seed-to-solution method for  
 the Einstein constraints 
\\
 and the asymptotic localization problem\footnotetext{$^1$ 
Laboratoire Jacques-Louis Lions \& Centre National de la Recherche Scientifique,
Sorbonne Universit\'e, 4 Place Jussieu, 75252 Paris, France.
Email: {\sl contact@philippelefloch.org}
\newline
$^2$ D\'epartement de Math\'ematiques et Informatique, Universit\'e Paris-Nanterre, 
200 Avenue de la R\'epublique, 92001 Nanterre, France. 
Email: {\sl alpthecang@gmail.com}
\newline 
\textit{Key Words and Phrases.} Einstein constraints; asymptotically Euclidean; harmonic decay; asymptotic localization.
\newline Published in: Journal of Functional Analysis (2023). 
}  
}
\author{Philippe G. LeFloch$^1$ and The-Cang Nguyen$^2$}

\date{}
               
\maketitle

\begin{abstract}  
We establish the existence of a broad class of asymptotically Euclidean solutions to Einstein's constraint equations, whose asymptotic behavior at infinity is arbitrarily prescribed. 
The proposed {\sl seed-to-solution method}  
encompasses vacuum  
as well as matter spaces, 
and relies on iterations based on the linearized  Einstein operator and its dual.
It generates a  Riemannian manifold (with finitely many asymptotically Euclidean ends)
from any {\sl seed data set} consisting of 
(1): a Riemannian metric and a symmetric two-tensor,
 and (2): a (density) field and a (momentum) vector field representing the matter content. 
We distinguish between {\sl tame} or {\sl strongly tame}
seed data sets,
depending whether the data provides a rough or an accurate asymptotic Ansatz at infinity. 
We encompass classes of metrics and matter fields with low decay (with infinite ADM mass) or strong decay (with Schwarzschild behavior). 
Our analysis is motivated by Carlotto and Schoen's pioneering work on the localization problem for Einstein's vacuum equations. Dealing with metrics with very low decay and, simultaneously, establishing estimates that include 
(and go beyond) harmonic decay require significantly new arguments which are developed in the present paper. 
We work in a weighted Lebesgue-H\"older framework adapted to the given seed data,
and we analyze the nonlinear coupling between the Hamiltonian and momentum constraints.
By establishing elliptic regularity estimates for the linearized Einstein operator and its dual, 
we uncover the novel notion of {\sl mass-momentum correctors} which is related to the ADM mass of the manifold. We derive precise estimates for the difference between the seed data and the actual Einstein solution, a result that should be of interest for future numerical investigation. 
Furthermore, we introduce here and study the {\sl asymptotic localization problem} 
in which we replace Carlotto-Schoen's exact localization requirement by an asymptotic condition
at a super-harmonic rate. By applying our seed-to-solution method with a suitably constructed, parametrized family of seed data, we solve this problem by exhibiting mass-momentum correctors with harmonic decay. 
\end{abstract} 
 
\newpage 

{\small 

\setcounter{tocdepth}{1}  
\tableofcontents   
 
}

\section{Introduction} 
\label{section--1}

\subsection{Einstein's constraint equations}

\paragraph{Curvature operators of interest.}

Any $3$-dimensional spacelike (i.e.~Riemannian) hypersurface embedded in a $4$-dimensional Lorentzian spacetime 
fullfiling {Einstein's field equations} must 
 satisfy constraint equations of Gauss-Codazzi type, which are  expressed in terms of the induced geometry (first and second fundamental forms) and the projection of the matter tensor (density and momentum) on this hypersurface. 
We thus consider such manifolds $(\Mbf, g, k)$ (with finitely many asymptotic ends) endowed with a Riemannian metric $g$ and a symmetric $2$-tensor field $k$ ---the latter representing the second fundamental form in the dynamical picture. 
Given a scalar field $\Hstar: \Mbf \to \RR$ and a vector field $\Mstar$ defined on $\Mbf$, we study here 
the equations\footnote{See the textbook~\cite{Choquet-book}. Some conditions such as $\Hstar \geq 0$ are natural in the context of general relativity,
but will not be used in this paper.} 
\bel{eq:ee11}
\aligned
R_g + (\Tr_g k)^2 -  | k |_g^2 & = \Hstar,   
\qquad
\Div_g \big ( k - (\Tr_g k) g \big) = \Mstar,   
\endaligned
\ee
which are referred to as the {\bf Hamiltonian and momentum constraints,} respectively.  
Here, $R_g$ denotes the scalar curvature of $g$, 
while $\Tr_g k$  and $| k |_g$ denote the trace  and the norm of $k$. 
Recall that in an arbitrary chart of local coordinates $(x^j)$ (in a subset of $\RR^3$) and with the standard notation for lowering or raising indices with the metric $g = g_{ij} dx^i dx^j$, one defines the trace $\Tr_g k = k_j^j = g_{ij} k^{ij}$, the (squared) norm $| k |_g^2 = k_{ij} k^{ij}$, and 
the divergence operator $(\Div_g k)_j = \nabla_i k^i_j$. 
Here, $\nabla$ denotes the Levi-Civita connection of $g$, and the range of Latin indices is $i,j, \ldots = 1,2,3$.

With the notation
$h \coloneqq k - \Tr_g(k) g$, 
\begin{subequations}
we write the Hamiltonian and momentum operators as
\be
\aligned
\Hcal(g,h) 
& \coloneqq R_g + {1  \over 2} \big( \Tr_g h \big)^2 - | h |^2_g,
\qquad\quad
\Mcal(g,h) \coloneqq \Div_g h, 
\endaligned
\ee
and by setting   
\bel{eq:EOp0}
\Gcal(g,h) \coloneqq \big( \Hcal(g,h), \Mcal(g,h) \big),
\qquad 
\Gstar \coloneqq (\Hstar, \Mstar), 
\ee 
\end{subequations}
Einstein's constraint equations \eqref{eq:ee11} read  
\bel{eq:Einstein00}
\Gcal(g,h) = \Gstar \quad \text{on the manifold } \Mbf. 
\ee  
Throughout, we use the notation $(g, h)$ rather than $(g,k)$.

\paragraph{Classes of solutions to Einstein's constraint equations.}

The Einstein constraints form a system of nonlinear partial differential equations 
of elliptic type,
 which is highly under-determined. The standard technique of existence of solutions goes back to pioneering work by Lichnerowicz and followers who developed the so-called conformal method. For a detailed bibliography\footnote{We also refer the reader to the following comprehensive review: A. Carlotto, The general relativistic constraint equations, Living Reviews in Relativity (2021), 24:2 (which appeared after the completion of the present paper).}
 we refer the reader to~\cite{CaciottaN}--\cite{DiltMaxwell},~\cite{GMS},~\cite{Maxwell}, and~\cite{Maxwell2014}, 
and the references therein. The coupling with a scalar field and the stability of solutions was investigated in the work~\cite{DruetHebey,DruetPremoselli}. 
In the present paper, we build upon a recent work by Carlotto and Schoen~\cite{CarlottoSchoen} about the localization problem (see Section~\ref{sec-1313} below), which stemed from pioneering work by 
Corvino~\cite{Corvino-2000}
and
Corvino and Schoen~\cite{CorvinoSchoen}. 
We also refer to 
Chru\'sciel and Delay~\cite{ChruDelay} and Chru\'sciel, Corvino, and Isenberg~\cite{ChrusCorvinoIsenberg} 
and the references therein, as well as~\cite{Allen,ChruscielB,ChD,CorvinoHuang}. 
Most of the existing literature (but not all) is focused on solutions that enjoy the standard Schwarzschild decay at infinity. However, as in~\cite{CarlottoSchoen} one may also investigate a more general behavior possibly arising with gravitational systems of matter, such as stars, galaxies, etc. It is our objective here to develop a theory that encompasses a broad class of  behaviors at infinity.

\subsection{Solutions generated from seed data sets: selected results}
\label{sec--1212}

\paragraph{Notion of seed data set.} 

Focusing on the set of solutions that are asymptotically Euclidean at infinity, in this paper we prove that the asymptotic behavior of solutions to  
Einstein's constraint equations may be essentially {\sl arbitrarily prescribed at infinity.} We proceed by choosing a ``seed data set'', consisting of a Riemannian manifold $(\Mbf, g_1)$ with finitely many asymptotically Euclidean ends, together with a symmetric two-tensor field $h_1$ on $\Mbf$, as well as a scalar field $\Hstar$ and a vector field $\Mstar$. In our theory, these data may be chosen to have very low decay at infinity or, to the contrary, very fast decay. 
Considering the vacuum equations or, more generally, the matter Einstein constraints, we prove that from any seed data we can generate a solution $(g,h)$ on $\Mbf$ that (essentially) enjoys the same asymptotic behavior as the one of the seed data. 
As will show in this paper, the precise description of the asymptotic behavior, in fact, is {\sl more involved} and one of our objectives is to derive conditions on the data $g_1, h_1, \Hstar, \Mstar$ (especially on their decay at infinity) and precisely control the asymptotic behavior of the corresponding Einstein solution.

\paragraph{A broad class of vacuum Einstein spaces.}

From our general theory in Section~\ref{section--2} below, we extract the following result. For the sake of simplicity, we tacitly assume sufficient regularity on the data and restrict attention to vanishing matter fields, while referring the reader to the next section for much more general statements. The exponents $p_G$ and $p_M$ below determine the decay of the metric (and extrinsic curvature) and the decay of the matter (or Einstein operator), respectively. 

\vskip.16cm 

\noindent{\bf Main Theorem 1} (The seed-to-solution method). 
{\sl 
{
Consider Einstein's constraint equations in the vacuum (that is, \eqref{eq:Einstein00} with $\Gstar = 0$) posed on a manifold with a single asymptotically Euclidean end. 
Consider any seed data set $(g_1, h_1)$ consisting of a Riemannian metric $g_1$ and a symmetric two-tensor $h_1$ satisfying suitable smallness conditions together with the following decay conditions (in a coordinate chart defined at infinity, $g_\Eucl$ being the Euclidean metric, and $r$ being the radius in the coordinates at infinity): }
\be
\aligned 
& g_1 = g_\Eucl + \Obig(r^{-p_G}),
\qquad \qquad 
&& h_1 =  \Obig(r^{-p_G-1}),  
\\
& 
\Hcal(g_1,h_1) =  \Obig(r^{-p_M-2}),
\qquad 
&& \Mcal(g_1,h_1) =  \Obig(r^{-p_M-2}),
\endaligned
\ee
for some\footnote{Throughout, for clarity in the presentation we keep the exponent $p_G$ to be less or equal to one: see the paragraph on ``follow-up works'' at the end of this section.} 
$p_G \in (1/2, 1]$
and $p_M \in (1/2, + \infty)$ 
(with $p_G \leq p_M$). 
Then, 
there exists a solution $(g,h)$ to Einstein's vacuum constraint equations 
$\Gcal(g,h) =0$ that enjoys the following decay properties. 
\bse
\bei 

\item {\bf Sub-harmonic constraint decay.} If $p_M < 1$ then 
\be
g = g_1 + \Obig(r^{-p_M}),
\qquad \quad
h =  h_1 + \Obig(r^{-p_M-1}). 
\ee

\item   
{\bf Harmonic constraint decay.} When $p_M = 1$ and $\Hcal(g_1,h_1)$ and $\Mcal(g_1,h_1)$ are integrable, then 
\be
\aligned
& g = g_1 + \mt \, r^2 \, \Hess_{g_\Eucl}(1/r) + \obig(r^{-1}), 
\qquad \quad
h =  h_1 + \Obig(r^{-2}), 
\endaligned
\ee 
where $\mt$ is a constant determined from the data.

\item {\bf Super-harmonic constraint decay.} When $p_M >1$, one has a stronger statement with $p = \min(p_G+1, p_M, 2)$: 
\be
\aligned
&  
g = g_1 +  \mt \, r^2 \, \Hess_{e}(1/r) + \Obig(r^{-p}),
\qquad \quad
h =  h_1 + \Obig(r^{-2}). 
\endaligned
\ee 
\eei
\noindent 
\ese
 
 }

Moreover, as we prove below, the ``mass corrector'' $\mt = \mt(g_1, h_1, \Mstar, \Hstar)$ in the last two cases 
satisfies\footnote{The meaning $\Ocal $ is specified in terms of weighted functional norms below.}
\bel{equa--39}
\mt = -{1 \over 16 \pi} \int_\Mbf \Hcal( g_1, h_1) \,d V_{g_1}
+ \Ocal( \Gcal( g_1, h_1)^2),
\ee
(which obviously vanishes if $( g_1, h_1)$ is an actual solution).  
  
\vskip.16cm


\paragraph{Geometry vs. matter decay rates.}

Let us consider particular choices for our exponents. 

\bei

\item {\bf Regime $p_G = p_M \to 1/2$.} This is the weakest possible decay on the metric as well as on the Einstein operator. Interestingly, our method does cover the limiting case $p_G = 1/2$ (not included in the previous statement). 

\item {\bf Regime $p_G \to 1/2$ with $p_M \to + \infty$.} This is the weakest possible decay on the metric and the strongest possible decay of the Einstein operator. This regime is covered in our statement above and \eqref{equa--39} tells us that the constant $\mt$ turns out to approach zero when $p_M \to + \infty$. Obviously the limit $p_M \to + \infty$ is realized when, for instance,
 the seed data is a solution\footnote{but the converse does not hold, as allowing for $p_M \to + \infty$ does not imply that the seed data is a solution}. 

\item {\bf Regime $p_G =1$ with arbitrary $p_M \in [1, +\infty)$.} This is the strongest decay on the metric while the Einstein operator can have a broad range of behaviors. When $p_M$ is close to $1$, the seed data is  a ``rough'' approximation, while our seed data is more accurate as $p_M$ becomes larger. 

\eei 
 

\subsection{The Asymptotic Localization Problem}
\label{sec-1313}

\paragraph{Optimal localization.}

The seed-to-solution method provides us with a strategy in order to tackle a question raised by Carlotto and Schoen in~\cite{CarlottoSchoen}. Solutions to 
Einstein's constraint equations
  were constructed which are prescribed within two asymptotic angular regions joined by a ``small' angular region, denoted as $\Cscr_a \cup \Cscr^c_{a+\eps}\cup \Tscr_{a,\eps}$ in our statement below. Within the asymptotic regions $\Cscr_a$ and $\Cscr^c_{a+\eps}$ Carlotto and Schoen impose that the solution coincides {\sl exactly} with (for instance) the Euclidean and Schwarzschild metrics, respectively. They prove that such solutions exist by solving the  
    vacuum Einstein constraints 
     in the transition region $\Tscr_a^\eps$ in suitably weighted function spaces. However, the method in~\cite{CarlottoSchoen} provides only a {\sl sub-harmonic} control on the decay of the solutions within $\Tscr_{a,\eps}$, that is, $r^{-p}$ with $p \in (1/2, 1)$. 

\paragraph{Asymptotic localization.}

In the present paper, we do construct solutions which are controlled at the level of harmonic decay (and, actually, beyond the harmonic decay)
 at the expense of slightly {\sl relaxing} the requirement that the solution coincides with prescribed metrics in angular regions. We study the {\bf asymptotic localization problem}, as we call it, consisting of seeking solutions which enjoy estimates with \underline{super-harmonic rate} and approach prescribed Minkowski and Schwarzschild behaviors (for instance) except outside a small transition region. From a physical standpoint, this new problem appears to be as natural as the optimal localization one.  Furthermore, with additional analysis our technique could be combined with the method in~\cite{CarlottoSchoen} and be applied within a cone-like domain in order to achieve an exact localization, but this is not our main objective in the present paper.


\paragraph{A new class of solutions.}

We solve the proposed problem and construct solutions that are harmonic in the transition region. We state only a typical result and refer to Section~\ref{sectio6} for a more general result (cf.~Theorems~\ref{theorem:trois}  and~\ref{corollary-optimalCS}). 

\vskip.16cm 

{

\noindent{\bf Main Theorem 2} (Resolution of the asymptotic localization problem.) 
{\sl  
Consider  
Einstein's vacuum constraint equations
 \eqref{eq:Einstein00} on a manifold $\Mbf$ with a single asymptotic end. 
Decompose asymptotic infinity into three asymptotic angular regions in $\RR^3$, say $\Cscr_a  \cup \Cscr^c_{a+\eps} \cup \Tscr_{a,\eps}$, where $\Cscr_a$ is a cone\footnote{defined as
$\Cscr_a := \big\{ x \cdot \nu \geq |x| \, \cos a \text{ and } |x| \geq 1 \big\}$ in coordinates at infinity for some unit vector $\nu$}
with a (possibly arbitrarily) small angle $a \in (0, 2\pi)$, 
$\Cscr^c_{a+\eps}$ is the complement of the same cone but with a (slightly) larger angle $a+\eps$, 
while $\Tcal_a^\eps$ is the remaining transition region.  
Then, by considering the Euclidean metric $g_\Eucl$ and the Schwarzschild metric $g_\Sch$ (with mass denoted by $m_\Sch>0$), 
there exists a solution to 
Einstein's vacuum constraint equations
 $\Gcal(g,h) =0$,  
whose metric has the $1/r$ behavior in each asymptotic direction while being asymptotic to the Euclidean and Schwarzschild metrics in the chosen cones, that is, in suitable coordinates at infinity and for any a priori fixed $q \in (1,2)$, 
\be 
\aligned
g & = g_\Eucl + \Obig(r^{-1}) \, \text{in } \, \Tscr_{a,\eps}, 
\qquad
g  = g_\Eucl  + \Obig(r^{-q}) \, \text{in } \, \Cscr^c_{a+\eps}, 
\qquad
g = g_\Sch + \Obig(r^{-q}) \, \text{in } \, \Cscr_a. 
\endaligned
\ee

}

}

\vskip.16cm 
 
The proof of this theorem will be based on a set of seed data that depend on parameters (determined implicitly within our proof), together with a further refinement of the seed-to-solution method described below.


\subsection{The Seed-to-Solution Method}

\paragraph{The seed data sets.} 

On a given $3$-manifold with finitely many asymptotic ends, we are thus given a seed data set consisting of a Riemannian metric $g_1$ and a symmetric two-tensor $h_1$, as well as a (matter density) scalar field $\Hstar$ and a (matter momentum) vector field $\Mstar$. 
To any such data (satisfying suitable smallness and decay conditions), 
we are able to associate
an asymptotically Euclidean solution to Einstein's constraint equations.   
We distinguish between several classes of seed data referred to as {\sl tame} or {\sl strongly tame}. 
This allows us to encompass metrics with the weakest possible decay and even solutions with infinite ADM mass, as well as
the strongest possible (that is, Schwarzschild) decay at infinity. 
We also distinguish whether the seed data provide a rough or accurate asymptotic Ansatz at infinity.    
 One can think of the seed data as being an approximate solution to the Einstein constraints. It can be generated by taking a formal solution in the vicinity of infinity, which is determined by plugging an Ansatz in the equations and can made to be of arbitrary accuracy at infinity. Such an Ansatz can be merged an with arbitrary data picked up 
in a bounded domain. In particular, the data is the bounded domain can be chosen to be an exact solution. Other strategies for  constructing seed data can be also considered.

\paragraph{Definition and continuity of the Seed-to-Solution map.}

We are going to linearize around the (rough or accurate) approximate solution $(g_1, h_1, \Mstar, \Hstar)$ and, under our tame decay conditions, analyze the structure of the (linearized, adjoint) 
Einstein constraints. 
 We proceed by assuming a rather minimal decay required on the seed data set in order to solve the equations \eqref{eq:Einstein00}. 
We prove the convergence of an iteration scheme whose principal part is a linearization of    Einstein's constraint equations 
around the given seed data, while nonlinearities are treated as a ``source''. The convergence to an actual solution is established in suitably weighted Lebesgue-H\"older spaces. Interestingly, at this stage of our construction, the solution {\sl need not have the expected decay} prescribed via the seed data or does so only at a {\sl sub-harmonic} rate of decay.
In the course of our analysis, we also study whether our solutions depend continuously upon the seed data and, specifically, we establish a Lipschitz continuity estimate in weighted spaces. We observe that the weighted Lebesgue regularity arises 
naturally in the analysis of the linearized Einstein operator in a variational form, while 
the H\"older regularity arises in the application of Douglis--Nirenberg's elliptic theory~\cite{DouglisNirenberg}. 
We refer the reader to Sections \ref{sec-LEO} and \ref{section Douglis-Nirenberg}, respectively. The existence of the 
 seed-to-solution map in weighted Lebesgue-H\"older spaces is established in Theorem~\ref{theo-main3}.

\paragraph{Encompassing slow or fast decay.}

In our construction, the seed data set may have {\sl much slower decay} at infinity in comparison to the Schwarzschild decay 
or, alternatively, may have precisely the Schwarzschild decay. In both cases, what is relevant is whether the seed data is
(for the Einstein-matter system) a rough or an accurate asymptotic Ansatz at infinity. This leads us to define two classes of data, referred to as tame and strongly tame, respectively. 
For the first class of data, we prove that the Einstein solutions have the prescribed decay (i.e.~the behavior of the seed data) up to a sub-harmonic rate, only. For the second class of data, our result is stronger and we prove that the constructed solutions have the prescribed decay up to ---and including--- the harmonic decay rate. 


\paragraph{Nonlinear geometry-matter coupling.}

In a second stage of our analysis, we investigate the asymptotic properties of the solutions we have constructed, and we relate the asymptotic behavior of the data and the asymptotic behavior of the actual solution. We must cope with several difficulties which are not dealt with in Carlotto and Schoen's original method~\cite{CarlottoSchoen}. We work with different decay exponents for the Hamiltonian and momentum components, and, most importantly, we treat geometry-matter terms at the harmonic level of decay. 
The main technical part is thus an investigation, based on several successive improvements of basic estimates, of the nonlinear coupling between the Hamiltonian and momentum operators. Under sufficiently strong decay conditions on the seed data,  we also relate the ADM mass and momentum of the solution with the one of the data, together with suitable mass and momentum ``correctors'', as we call them.
 

\paragraph{Follow-up works.}

In the present paper, we focus on the problem of ``reaching {\it and} breaking the $1/r$ barrier" , as  it could be called, by considering a problem first identified in~\cite{CarlottoSchoen}. 
From the standpoint of mathematical analysis, the main novelty of the present paper is the use of 
Green functions for the linearized Einstein constraints and their dual.  
After this paper was distributed (arXiv:1903.00243, March 2019), two related and very interesting results were announced, namely two methods of gluing with optimal decay by 
S. Aretakis, S. Czimek, and I. Rodnianski (arXiv:2107.02441, July 2021) 
and Y. Mao and Z. Tao (arXiv:2210.09437, October 2022). 
While in the present paper we focus on the approximate localization problem with harmonic estimates, we point out that our method can be extended in several directions. In a recent work, B. Le Floch and P.G. LeFloch~\cite{Le-FlochLeFloch} solve the exact localization problem with harmonic estimates, in the original form posed by Carlotto and Schoen~\cite{CarlottoSchoen}.


\paragraph{Outline of this paper.} 

In Section~\ref{section--2}, we introduce our tame decay conditions together with the relevant function spaces. We state the main results established in this paper in Theorem~\ref{theo-prec} (existence of the seed-to-solution map) and in Theorems~\ref{theo:deuxieme} and \ref{theo:deuxieme-2} (asymptotic properties of the solutions). In Section~\ref{section--3}, we investigate the dual version of the linearized Einstein operators in suitably weighted Sobolev spaces and, next,
 in suitably weighted H\"older spaces. We then solve
 Einstein's (nonlinear) constraint equations  and establish the existence of the seed-to-solution map, completing therefore the proof of Theorem \ref{theo-prec}. 
Next, in Section~\ref{section--6}, we study the linearized Hamiltonian and momentum operators and derive refined estimates at the (super-)harmonic level of decay. 
In Section~\ref{section--7}, the asymptotic properties of the seed-to-solution map are established and we provide a proof of Theorems~\ref{theo:deuxieme} and \ref{theo:deuxieme-2}.  
Finally, the asymptotic localization problem is solved in Section~\ref{sectio6}.  


\section{The seed-to-solution map: definition, existence, and asymptotic properties}  
\label{section--2}

\subsection{The functional setup} 

\paragraph{Background geometry.}

Throughout, $\Mbf$ denotes a topological $3$-manifold with finitely many Euclidean ends ---a notion we define below after introducing now a background manifold. 

\begin{definition} 
\label{def-backgr}
Given some $\eps_* \in (0,1]$, a {\bf background manifold $(\Mbf,e,r)$} is a smooth Riemannian manifold that admits 
finitely many ends, which are denoted by $N_1, N_2, \ldots, N_n$ and are pairwise disjoint and diffeomorphic to the exterior of a closed ball in $\RR^3$, 
and moreover is endowed with a {\bf radius function} $r: M \to [1, + \infty)$, such that:

\bei 

\item In each end, the metric $e = e_{ij} dx^i dx^j = \delta_{ij} dx^i dx^j$ is the Euclidean metric in a suitably chosen chart\footnote{As mentioned earlier, all Latin indices $j,k, \ldots$ range in $1,2,3$.}
 $(x^j)$ and the radius function coincides with $r^2 = \sum_j (x^j)^2$.  

\item The manifold-with-boundary $M_0\coloneqq M \setminus (N_1 \cup N_2 \ldots \cup N_n)$ is covered by a finite collection of local coordinate charts $(x^j)$ in which the metric $e = e_{ij} dx^i dx^j$ is close to the Euclidean metric in the sense that the functions $e_{ij} - \delta_{ij}$ admit continuous derivatives up to fourth order which are less than $\eps_*$ in the sup-norm. 

\eei 
\end{definition}
 
A sufficiently small parameter $\eps_* \in (0,1]$ is fixed throughout this paper and we often denote such a fixed background by $\Mbf = (\Mbf,e,r)$. No topological restriction is required on the manifold under consideration. 
It is useful to fix (once and for all) a {\bf partition of unity} $\chi_a$ adapted to the family of asymptotic ends $N_a$ and the compact set $M_0$, that is,  
\bel{equa=chif}
\chi_a \geq 0, 
\qquad
\sum_{0 \leq a \leq n} \chi_a \equiv 1, 
\qquad 
\chi_a|_{N_a} \equiv 1 \qquad (a=1, \ldots, n). 
\ee
\beginmyremark
\label{rem22}
  A more standard presentation would consist of introducing only the Euclidian metric in the asymptotic ends; 
  our  
   introduction of a background metric $e$ is convenient in order to state quantitative estimates involving the actual 
   solution and the seed data.
For instance if we are interested in the topology $M \simeq\RR^3$, then we simply choose $\eps_*= 0$ and standard Cartesian coordinates $(x^j)$ defined globally on $\RR^3$  
with $e_{ij} = \delta_{ij}$ in $\Mbf$ together with the radius function given by 
$r(x) = \sqrt{|x|^2 + e^{-1/(1-|x|^2)}}$ for $|x| \leq 1$, while 
$r(x) = |x|$ for all $|x| \geq 1$. 
\end{remark}


\paragraph{H\"older and Lebesgue spaces of interest.}

On a background manifold $\Mbf = (\Mbf,e,r)$, we now define functional spaces based on the volume form $dV_e$ determined by the metric $e$. It is convenient to use the same notation for spaces of scalar, vector, and tensor fields (except when some emphasis is useful) and, for simplicity, we state our definitions below for functions. For tensor fields, all the norms below should be defined with respect to the specific atlas of coordinate charts implied by Definition~\ref{def-backgr}. More precisely, the following weighted spaces will be needed. 

\bei

\item {\bf H\"older spaces.} 
Given $\alpha \in (0,1)$ and $\theta >0$
and  given a 
 non-negative integer $l$, 
   we define the {\bf weighted H\"older space} $C_\theta^{l,\alpha}(\Mbf,e,r)$ as the space of functions $f: \Mbf \to \RR$ with H\" older regularity of order $l + \alpha$ and with finite weighted norm 
\begin{subequations}
\be
\| f \|_{C_\theta^{l, \alpha}(\Mbf,e,r)}
\coloneqq 
\sum_{\text{charts}}
\sum_{|L| \leq l}   \sup_{\Mbf} \Big( r^{|L| + \theta} \, |\del^L f|  \chichart \Big) 
 +  
\sum_{\text{charts}}
\sum_{|L| = l}   \sup_{\Mbf} \Big( r^{|L| + \theta} \,  [\del^L f]_\alpha \chichart \Big),
\ee
where $\chichart$ denotes a partition of unity associated with the partition of unity $\chi_a$ (see above) completed with the chosen family of charts covering $M_0$.  
Here, $\del^L f$ denotes the partial derivatives (in any given chart) with respect to the multi-index $L$
and, furthermore with $x,y$ denoting local coordinates,  
\be
[ f]_\alpha (x) \coloneqq r(x)^\alpha  \sup_{0 < |y - x| \leq r(x)} {|f(y) - f(x)| \over |y - x|^\alpha}. 
\ee 
Moreover, when $l= 0$ we simply write $C_\theta^\alpha(\Mbf,e,r)$ instead of $C_\theta^{0,\alpha}(\Mbf,e,r)$. 


\item {\bf Lebesgue spaces.} Given any real $\theta >0$,   
 we define the {\bf weighted Lebesgue space} $L_\theta^2(\Mbf,e,r)$ by completion of the set of smooth functions $f: M \to \RR$ with bounded support and finite norm
\be
\|f\|^2_{L_\theta^2(\Mbf,e,r)}
\coloneqq  \int_\Mbf  | f |^2 \, r^{-3  + 2 \theta} \, dV_e. 
\ee
\end{subequations}
For instance, when $M= \RR^3$, we recover the standard $L^2$ space provided the exponent $\theta$ is chosen to be $3/2$. We will also use the notation $L^k(\Mbf,e)$ and the standard $L^1k$ norm of a function $f: \Mbf \to \RR$ defined as 
$\|f\|_{L^k(\Mbf,e,r)} \coloneqq  \big( \int_\Mbf  | f|^k \, dV_e\big)^{1/k}$ (with $k=1,2$ for our purpose).  


\item {\bf Lebesgue-H\"older spaces.} Combining the previous two definitions, we refer to 
$L^2C_\theta^{l,\alpha}(\Mbf,e,r) \coloneqq L_\theta^2(\Mbf,e,r) \cap C_\theta^{l,\alpha}(\Mbf,e,r)$ 
as the {\bf weighted Lebesgue-H\"older space} with decay exponent $\theta$ and regularity exponents $l, \alpha$.
The squared norm of a function $f$ in this space is defined as 
$\| f \|_{L^2C_\theta^{l,\alpha}(\Mbf,e,r)}^2
\coloneqq \| f \|_{L^2_\theta(\Mbf,e,r)}^2 + \| f \|_{C_\theta^{l,\alpha}(\Mbf,e,r)}^2$. 
In our notation, {\sl a larger $\theta$ means a stronger decay in space. } We find this notation to be natural in the present context, especially since we are interested in harmonic decay issues. {\sl Roughly speaking,} a function in a space with a subscript $\theta$ decays slightly faster than $1/r^\theta$ at infinity. 

\eei 

Often, we will not specify the background manifold $(\Mbf,e,r)$ and we will write $L^2C_\theta^{l,\alpha}(\Mbf)$, $L^1(\Mbf)$, etc. In our proofs, we sometimes omit the manifold $\Mbf$ from the notation, unless some confusion may arise.


\subsection{The notion of tame seed data sets} 

\paragraph{Basic definitions.}

We begin with several definitions. We emphasize that our decay conditions below are much weaker than the standard ones (and we refer the reader to specific comments given after each definition and statement below). The notation $(a,b) < (c,d)$ is used when $a < c$ and $b < d$ both hold (with obvious generalizations). Throughout, $\eps_*, \eeM, \eeG \, \in (0,1)$ are assumed to be sufficiently small. 

\begin{definition}
\label{def:basics}
A pair $(p, q)$ satisfying\footnote{$p, q$ can be taken to be arbitrarily large in some of the following statements.} 
$1/2 \leq p \leq 2 ( q -1)$ is called a pair of {\bf admissible decay exponents.} Such a pair is said to be {\bf sub-harmonic} if, moreover, one has $(p,q) < (1,2)$, and to be {\bf harmonic} if $p=1$ or $q=2$ (or both). In particular, the {\bf balanced regime} by definition corresponds to the choice of admissible exponents 
$$
p \geq 1/2, \qquad q=p+1, 
$$
which is sub-harmonic or harmonic if $p <1$ or $p=1$, respectively. 
\end{definition}

\begin{remark}
The term "balanced" refers to the fact that under the condition 
$q=p+1$ the terms $\del g$ and $h$ 
 then have a decay of the same order.  For instance, this will be our choice in the case~\eqref{equa-typical-choice}. 
In a first reading, all of our exponents could be taken in the balanced regime. Yet, encompassing exponents $q \neq p+1$ is  
of interest in order to clearly distinguish between the contributions of the metric and extrinsic curvature.
\end{remark}

\begin{definition} 
\label{def-aset}  
Given a H\"older exponent $\alpha \in (0,1]$ and admissible exponents $(\ppG, \qqG)$ and $(\ppM, \qqM)$, a {\bf tame seed data set } $(g_1, h_1, \Hstar, \Mstar)$ over the background manifold $\Mbf = (\Mbf,e,r)$ consists of four tensor fields defined on $\Mbf$ and satisfying the following regularity and decay conditions for  
$\eeG, \eeM \in (0,1)$. 
\bei 

\item {\bf Asymptotically Euclidean data:} 
$g_1$ is a Riemannian metric satisfying\footnote{Our notation $A \lesssim B$ stands for $A \leq c B$ where $c>0$ is a fixed constant.
  In our presentation we find it convenient to think of $ \eeG$ and $\eeM$ as real numbers in the interval $(0,1)$ 
  and, consequently,  
the implied constants such as the ones in Definition~\ref{def-aset} 
are {\sl not} dimension-less. These 
numerical constants depend on the choice of our data as well as $\ppG, \qqG, \ppM,\qqM, \alpha$; they are irrelevant for the mathematical analysis, but would play a role in numerical computations, say.
}  
$\| g_1 - e \|_{C_{\ppG}^{4, \alpha}(\Mbf)} \lesssim \eeG$, 
while 
$h_1$ is a symmetric $(0,2)$-tensor satisfying  
$\| h_1 \|_{C_{\qqG}^{3,\alpha}(\Mbf)} \lesssim \eeG$. 

\item {\bf Asymptotically Einstein data:} $\Hstar$ is a scalar field satisfying\footnote{This inequality and the next one restrict, both, the Einstein operator applied to $(g_1, h_1)$ and the matter content of the space.} 
$\| \Hcal(g_1,h_1) - \Hstar \|_{L^2C_{\ppM + 2}^{\alpha}(\Mbf)} \lesssim \eeM$, 
while 
$\Mstar$ is a vector field satisfying  
$\| \Mcal(g_1, h_1) - \Mstar \|_{L^2C_{\qqM + 1}^{1, \alpha}(\Mbf)} \lesssim \eeM$.

\eei 
\end{definition}


\paragraph{Decay of the seed data.}

The above definition provides us with a (quantitative) formulation of the asymptotic flatness conditions at each end, while also constraining the manifold to be ``almost Euclidean'' in a whole. 
For instance, on $M= \RR^3$ endowed with the standard metric, we can easily construct seed data satisfying our conditions by deriving a formal expansion near infinity and using a standard cut-off technique in order to glue this expansion with the Euclidean metric in the interior. 
As will become clear in the following, our admissibility condition $1/2 \leq p \leq 2 ( q -1)$ is natural and, for instance, 
  we observe that  the Hamiltonian equation 
  is
{\sl schematically} 
a Laplace operator (in the metric $e$, say) with a quadratic right-hand side 
  in $h$ (plus other terms with similar or higher decay), hence the inequality $p-2 \leq 2q$ is assumed. 
Our conditions so far are very mild, but will be sufficient for our existence result in Theorem~\ref{theo-prec}. In order to further motivate our definition above, we record here some observations\footnote{which, in a first reading, can be skipped by the reader}.
 
\bei 

\item {\bf Rate of decay of the seed metric $g_1$.} We assume that $\ppG \geq 1/2$, while the condition $\ppG > 1/2$ is, {\em in principle}, the standard assumption for the ADM mass to be well-defined (together with the integrability of the scalar curvature). 
  Yet, even under the condition $\ppG > 1/2$, our tensor $h_1$ may have a rather slow decay
  and, after solving the coupled Einstein constraints,
may contribute to generate an 
actual solution with infinite ADM mass. 

\item {\bf Rate of decay of the seed tensor $h_1$.}
Our condition $\qqG \geq 1 + \ppG/2$ does not imply that the ADM momentum is well-defined, since the standard condition is $\qqG > 3/2$. This condition $\qqG \geq 1 + \ppG/2$ is motivated  from a different perspective: in the Hamiltonian constraint this is the condition required 
  for the source $(1/2) \Tr_e(h_1)^2 - |h_1|_e^2$  
  to be compatible with the assumed decay on the metric given by $\ppG \geq1/2$. 

\item {\bf Mass and matter content.} 
In the important special case $\qqG \geq \ppG + 1 > 3/2$, both the ADM mass and the ADM momentum are well-defined, and this case will be treated in Theorem~\ref{theo:deuxieme} below. 
Our general existence theory, stated in Theorem~\ref{theo-prec} below, provides us with solutions to  
 the Einstein constraints 
 that may have infinite mass ---namely in the regime $\qqG \in (5/4, 3/2)$. This is the first result of this kind in the mathematical literature, which is of interest at least from the standpoint of understanding the structure and the coupling properties of  Einstein's constraint equations. 

\eei


Furthermore, we do not impose any {\sl lower bound} on the decay rate of $g_1, h_1$ and, on the other hand, we do not assume any specific sign on the scalar curvature $R_{g_1}$. Of course, these two issues are related in view of the positive mass theorem~\cite{SchoenYau}. 

\bei 

\item When $R_{g_1} \geq 0$ one can not allow data with a too strong decay since, otherwise, by the positive mass theorem the assumption $\ppG >1$ would then imply that $g_1=e = \delta_{ij}$ can only be the Euclidean metric on $M \simeq \RR^3$.

\item However, since $R_{g_1}$ need not be non-negative in Theorem~\ref{theo-prec} below, we can thus allow for $\ppG >1$ in our existence theory ---although this may not be the most interesting regime. 
 
\eei

\subsection{  
Solutions to Einstein constraints generated from general tame data}

\paragraph{Effective exponents.} 

Based on a suitable iteration scheme, we are going to define a map $(g_1, h_1) \mapsto (g, h)$ that associates an actual solution $(g, h)$  to any seed data. We can think of $(g_1, h_1)$ as an ``approximate solution''. 
In Definition~\ref{def-aset}, we specified its decay at infinity as well as the decay enjoyed by the matter terms $\Hcal(g_1,h_1) - \Hstar$ and $\Mcal(g_1,h_1) - \Mstar$. We anticipate that the decay behavior of $(g, h)$ could be different from the one of $(g_1, h_1)$.  
Our analysis will proceed in two stages. We begin by establishing a rather general existence theory for tame data
 and by controlling the asymptotic behavior at any arbitrary sub-harmonic level.

\begin{definition} 
\label{def:basics2}
Given any admissible exponents $(\ppG, \qqG,)$ and $(\ppM, \qqM)$, a pair of admissible exponents $(p,q)$ is called {\bf effective} if 
\bel{eq--2-12}
\aligned
& (p, q) <
(1,2), 
\qquad 
 (p, q) \leq (\ppM, \qqM), 
\qquad
 | q - p -  1 | \leq \qqG -1.
\endaligned
\ee 
In the balanced regime, this amounts to assume that 
\be
p <
1, \qquad 
p \leq \ppM, \qquad q=p+1.
\ee 
\end{definition}

Heuristically, a natural choice for $(p,q)$ is to try to saturate the inequality $(p, q) \leq (\ppM, \qqM)$ and choose them to coincide with $(p_M, q_M)$. We are also constrained by the requirement that $(p,q)$ remains sub-harmonic, while  
the exponents $(\ppM, \qqM)$ may be arbitrary large.  


\paragraph{The existence theory.}  

Our first result concerns the existence of solutions with prescribed asymptotics beyond the harmonic decay. In short, we prove now that the equations can be solved from a prescribed seed data up to the effective rate of decay.

\begin{theorem}[Existence of solutions to Einstein constraints
 for tame data sets] 
\label{theo-prec} 
Consider an arbitrary tame seed data set $(g_1, h_1, \Hstar, \Mstar)$ defined on a background manifold $\Mbf = (\Mbf,e,r)$, associated with admissible decay exponents $(\ppG, \qqG,\ppM, \qqM)$ and H\"older exponent $\alpha \in (0,1)$, and 
fix any pair $(p,q)$ of effective exponents. 
For instance, in the balanced regime this amounts to assume that 
\bel{equa-typical-choice} 
1/2 \leq \ppG, 
\qquad 
p < 1, 
\qquad
1/2 \leq p \leq \ppM,
\qquad
 (\qqG, \qqM, q) = (\ppG+1, \ppM+1, p+1).
\ee
Under these conditions and provided 
$\eeG, \eeM \in (0,1)$ are sufficiently small, 
  there exists a pair
  of $(0,2)$-tensors $(g,h)$ defined on $\Mbf$ such that: 
\bei 

\item {\bf Solution property.} The pair $(g,h)$ is a $C^{2, \alpha}$--H\" older continuous solution to 
 the Einstein constraints 
 \eqref{eq:Einstein00} associated with the matter fields $\Hstar$ and $\Mstar$. 

\item {\bf Asymptotic property.} This solution is asymptotic to $(g_1, h_1)$ in the sense that 
$\| g - g_1 \|_{L^2C_{p}^{2, \alpha}(\Mbf)} + \| h - h_1 \|_{L^2C_{q}^{2, \alpha}(\Mbf)}
\lesssim \eeG$. 

\item {\bf Stability property.} More precisely, this solution depends continuously upon the values of the Einstein operator in the sense that 
\bel{equa-213-2}
\aligned
& \| g - g_1 \|_{L^2C_{p}^{2, \alpha}(\Mbf)} 
 \lesssim
\| \Hcal(g_1,h_1) - \Hstar \|_{L^2C_{p + 2}^{\alpha}(\Mbf)}
 + \eeG \, \| \Mcal(g_1,h_1) - \Mstar \|_{L^2C_{q+1}^{1, \alpha}(\Mbf)}, 
\\
& \| h - h_1 \|_{L^2C_{q}^{2, \alpha}(\Mbf)}
 \lesssim
\eeG \, \| \Hcal(g_1,h_1) - \Hstar \|_{L^2C_{p + 2}^{\alpha}(\Mbf)}
 +  
 \| \Mcal(g_1,h_1) - \Mstar \|_{L^2C_{q+1}^{1, \alpha}(\Mbf)}, 
\endaligned 
\ee
in which the implied constants depend on the choice of decay and regularity exponents.

\item  {\bf Structure of the solutions.} Furthermore, $(g,h)$ 
belongs to the image of the dual operator 
$d\Gcal_{(g_1,h_1)}^*$ associated with 
 the linearized Einstein constraints\footnote{See the discussion around \eqref{eq2:hua}, below.}
 at $(g_1,h_1)$ and, specifically, for some $(u, Z)$ 
\bel{equa--condition}
g-g_1 = r^{3 - 2\oldp} \, d\Hcal_{(g_1,h_1)}^* [u,Z], 
\qquad 
h-h_1 =  r^{3 - 2\oldq}  \, d\Mcal_{(g_1,h_1)}^*[u,Z], 
\ee
with 
\bel{equa-213-2-u-Z} 
\aligned
& \| u \|_{C^{4, \alpha}_{1-p}}  
 \lesssim
\| \Hcal(g_1,h_1) - \Hstar \|_{L^2C_{p + 2}^{\alpha}(\Mbf)}
 + \eeG \, \| \Mcal(g_1,h_1) - \Mstar \|_{L^2C_{q+1}^{1, \alpha}(\Mbf)}, 
\\
&  \| Z \|_{C^{3, \alpha}_{2-q}} 
 \lesssim
\eeG \, \| \Hcal(g_1,h_1) - \Hstar \|_{L^2C_{p + 2}^{\alpha}(\Mbf)}
 +  
 \| \Mcal(g_1,h_1) - \Mstar \|_{L^2C_{q+1}^{1, \alpha}(\Mbf)}.
\endaligned 
\ee

\eei 
\end{theorem}


In particular, if $(g_1, h_1)$ is known to be an actual solution, then $(g, h)$ does coincide with that solution. 
In fact, if $(g_1, h_1)$ is chosen to be a sequence converging to an actual solution, then the distance (in appropriate weighted spaces specified in our theorem) between this solution and the solution $(g,h)$ approaches zero.  
It is worth mentioning also that, by construction, the solutions $(g,h)$ in Theorem~\ref{theo-prec} are small perturbations of the background data $(e, 0)$ (with $e$ itself being sufficiently close to the Euclidean metric), since 
$ 
\| g - e \|_{L^2C_{p}^{2,\alpha}(\Mbf)} +  \| h \|_{L^2C_{q}^{2,\alpha}(\Mbf)}
\lesssim \eeM + \eeG. 
$
From now on, we assume that $\eeG, \eeM \in (0,1)$ are sufficiently small, so that all of the conclusions in Theorem~\ref{theo-prec} hold.

\begin{remark}
\label{rem-new1}
A version of Theorem~\ref{theo-prec} (in suitably weighted Sobolev spaces) also follows
by combining the surjectivity property of the linearized Einstein map established in \cite{CorvinoSchoen} with general implicit function arguments in Banach spaces~\cite{AMR83} (cf.~Theorem 2.5.9 therein). 
 On the other hand, our presentation is self-contained and provides properties of the solutions constructed by the seed-to-solution method, and 
 the conditions \eqref{equa-213-2}, \eqref{equa--condition}, and  \eqref{equa-213-2-u-Z}
 will play an essential role in the rest of this paper. 
\end{remark}


\paragraph{Heuristic observations.} 

We can motivate our conditions \eqref{eq--2-12} as follows.  

\bei 

\item  The sub-harmonic assumption in \eqref{eq--2-12} is imposed since  
  from the schematic form of the Hamiltonian operator, namely $\Delta_e g_2$, for the perturbation $g_2 = g - g_1$
 one expects harmonic terms $1/r$ which are not included in Theorem~\ref{theo-prec} (but will be studied in Theorem~\ref{theo:deuxieme}). 
Similarly, from the schematic form of the Hamiltonian operator $\Div_e h_2$ for the perturbation $h_2 = h - h_1$, 
 one expects terms in $1/r^2$ which are not included in Theorem~\ref{theo-prec} (but will be studied in Theorem~\ref{theo:deuxieme-2}). 

\item Our second condition $(p,q) \leq (p_M, q_M)$ in  \eqref{eq--2-12} relates the effective exponents to the matter exponents, while the geometry exponents satisfy $(p_G, q_G) \leq (p_M, q_M)$.

\item Our third condition in \eqref{eq--2-12} is relevant in view of the  
  linearization of the constraints
   about the data $(g_1, h_1)$. 
{\sl Schematically,} the Hamiltonian equation is a Laplace equation on $g_2$ (in the metric $e$, say)
with 
right-hand side containing products of $h_1$ and $h_2$, as well as the term $(\Hcal(g_1,h_1) - H_*)$
(plus other terms with similar or higher decay). 
From which we argue
 that the conditions $p+2 \leq q_G + q$ and $p+2 \leq p_M+2$
are necessary for the nonlinearities to decay at least as fast as the principal term. 

Similarly, an analysis of the linearization of the momentum constraint about the data $(g_1, h_1)$ leads us to consider the operator
  $\Div_e h_2$ 
  and source-terms that are products of $h_1$ and $\del g_2$ as well as the term
 $(\Mcal(g_1, h_1) - \Mstar)$ (plus other terms with similar or higher decay), 
  from which we argue that the conditions  $q+1 \leq q_G + p + 1$ and $q+1 \leq q_M +1$ are required. 

\eei


\subsection{  
Solutions to Einstein constraints generated from strongly tame data} 

\paragraph{The seed-to-solution map.}

Under the conditions in Theorem~\ref{theo-prec}, by choosing $p=1/2$ and 
  $q=3/2$ therein we can define a map 
\bel{def.seed-map}
\Pss : (g_1, h_1) \mapsto (g, h),  
\ee
which we refer to as the {\bf seed-to-solution map} associated with 
Einstein's constraint equations. 
It is defined on a suitable subset of $C_{\ppG}^{4, \alpha}(\Mbf) \times C_{\qqG}^{3, \alpha}(\Mbf)$, and its image is such that 
$g- g_1 \in L^2C_{1/2}^{2, \alpha}(\Mbf)$ and $h-h_1 \in L^2C_{3/2}^{2, \alpha}(\Mbf)$. 
Of course, this map also depends upon the choice of the matter terms $\Hstar, \Mstar$, which we regard as fixed data and, for instance, can be chosen to vanish identically if one is solely interested in vacuum solutions. 

\paragraph{Basic definitions.}

Our next result establishes that, under additional conditions, the 
prescribed seed data behavior is ``realized '' by the actual solution at (and beyond) the {\sl standard $1/r$ Schwarzschild rate.} 
This naturally requires us to assume that our seed data is a suitably ``accurate'' asymptotic solution. 

\begin{definition}
\label{def-sb}  
Under the condition in Definition~\ref{def-aset}, a tame seed data set $(g_1, h_1, \Hstar, \Mstar)$ is said to be {\bf strongly tame} for, respectively, the Hamiltonian operator $\Hcal(g_1,h_1) - \Hstar$
and for the momentum operators
 $\Mcal(g_1,h_1) - \Mstar$ if
the following decay conditions\footnote{The finite integrability of the Hamiltonian and momentum operators is relevant only for the harmonic values $\ppM = 1$ and $\qqM=2$. Recall that throughout we write $L^1(\Mbf)$ instead of $L^1(\Mbf,e,r)$, etc.}
hold, respectively: 
\bel{eq:strongtameH}
\ppM \geq 1, \qquad  \qqM  > \max(3 - \qqG, 3/2), 
\qquad  
\Hcal(g_1,h_1) - \Hstar \in L^1(\Mbf), 
\ee  
\bel{eq:strongtameM}
\qqM \geq 2, 
\qquad \hskip4.2cm 
\Mcal(g_1,h_1) - \Mstar \in L^1(\Mbf). 
\ee    
\end{definition}
 
We emphasize that our conditions concern the matter (or Einstein) content of the seed data and do not restrict the behavior of the metric itself, which may have very low or very strong decay (cf.~the examples in Section~\ref{sec:ADM-exam}). 

\begin{definition} 
\label{definition-pstar}
Given any admissible exponents $(\ppG, \qqG,\ppM, \qqM)$ associated with a strongly tame data set, some exponents $(\pstar, \qstar)$ are called {\bf effective for strongly tame data} if\footnote{More precisely, $\pstar=1$ is the harmonic case while $\pstar>1$ is the super-harmonic regime.} 
\bel{eq:strongexpo} 
\aligned
&  (1,2) \leq (\pstar, \qstar) \leq (\ppM, \qqM), 
\qquad
\\
& \pstar \leq \min (\ppG + 1, \qqG + \qqM - 2),
\qquad  
&& \pstar < \min (2, \qqG), 
\\
& \qstar \leq \min (\ppG + 2, \qqG +1), 
\qquad 
&&
\qstar < 3. 
&
\endaligned
\ee   
In the balanced regime $(\qqG, \qqM, \qstar) = (\ppG+1, \ppM+1, \pstar+1)$ these conditions amount to assuming
\bel{equa-210-balanced} 
1/2 \leq \ppG, 
\qquad
1 \leq \pstar \leq \min(\ppG+1,\ppM), 
\qquad 
\pstar < \min(2,\ppG+1). 
\ee
\end{definition} 


\paragraph{Analysis of the seed-to-solution map.}
 
 We now establish that the Einstein solution $(g,h)$ in Theorem~\ref{theo-prec} is asymptotically close to $(g_1, h_1)$ at the {\sl harmonic rate} of decay, in the sense that if the seed data is strongly tame, then the metric $g$ enjoys the same asymptotic behavior at infinity as the seed metric up to the harmonic rate, with 
$ \| g - g_1 \|_{C_{1}^{2,\alpha}(\Mbf)} +  \| h - h_1 \|_{C_{2}^{2,\alpha}(\Mbf)} \lesssim \eeM$. 
In fact, our result is even stronger since we also describe the harmonic term in the metric, as now stated. 

\begin{theorem}[Asymptotic behavior at the (super-)harmonic level -- The metric component] 
\label{theo:deuxieme} 
Assume the strongly tame conditions \eqref{eq:strongtameH} for some seed data set $(g_1, h_1, \Hstar, \Mstar)$ on a background manifold $\Mbf = (\Mbf,e,r)$ with decay exponents $(\ppG, \qqG,\ppM, \qqM)$ and H\"older exponent $\alpha \in (0,1)$, 
and consider the solution $(g,h)$ to the Einstein-matter equations given by Theorem~\ref{theo-prec}.
(For instance, the statement applies with decay exponents satisfying \eqref{equa-typical-choice} and \eqref{eq:strongtameH}.)
Then, in suitable coordinates at infinity the metric is asymptotic to\footnote{The functions $\chi_a$ were introduced in \eqref{equa=chif}.}
\bel{equa-2p8} 
\gt_1  := g_1 + \sum_{1 \leq a \leq n} \chi_a \mt_a r^2 \, \Hess_{e} (1/r), 
\ee
referred to as the {\bf effective seed metric}, 
for all (super-)harmonic exponents $\pstar$ satisfying \eqref{eq:strongexpo}: 
\begin{subequations}
\bel{p*<2}
\aligned
\| g - \gt_1\|_{C_{\pstar}^{2,\alpha}(\Mbf)} 
& \lesssim \Ecal(g_1, h_1, \Hstar, \Mstar) 
+ \max_{1 \leq a \leq n}  
\big| m_a^* (g_1, h_1, \Hstar) \big|. 
\endaligned 
\ee
Here, $\Ecal = \Ecal(g_1, h_1, \Hstar, \Mstar)$ measures the ``accuracy'' of the seed data (with $\qstar' := \min(\pstar + 1, \qqM)$): 
\be
\aligned
\Ecal 
:= 
\, 
& 
\| \Hcal(g_1,h_1) - \Hstar \|_{C^{\alpha}_{\pstar + 2}(\Mbf)}  
 + ( \eeG + \eeM ) \, \| \Mcal(g_1,h_1) - \Mstar \|_{C^{1,\alpha}_{\qstar'+1}(\Mbf)}, 
\endaligned
\ee
and, for each asymptotic end $N_a$, $m^*_a = m_a^* (g_1, h_1, \Hstar)$ is the following average of the Hamiltonian data 
\be
m^*_a := {1 \over 16 \pi} \int_\Mbf  \chi_a  \big(\Hstar -  \Hcal( g_1, h_1) \big) \, d V_{g_1}, 
\ee
while the {\bf mass correctors} $\mt_a$ satisfy the estimate 
\bel{massg2b}
\aligned
\max_{1 \leq a \leq n}
\big| \mt_a - m_a^* \big| 
& \lesssim 
\eeG \, \| \Hcal(g_1,h_1) - \Hstar \|_{L^2C_{5/2}^{\alpha}(\Mbf)}
   + \eeG \, \| \Mcal(g_1,h_1) - \Mstar \|_{L^2C_{5/2}^{1, \alpha}(\Mbf)}.   
\endaligned
\ee
In particular, as well in the harmonic case $p_* = 1$, 
one has $\lim_{r \to +\infty} \big( r \, | g-\gt_1 | + r^2 \, |\del (g - \gt_1) | \big) = 0$. 
\end{subequations}
\end{theorem}

It is convenient to introduce now the notation $V_{ij} := \Big( x_i x_j + 3 r^2 \delta_{ij} \Big)/ r^3$, which is defined in each of the charts at infinity and can be smoothly extended to the whole of the manifold.  
The notation $\Lcal$ below stands for the Lie derivative operator.

\begin{theorem}[Asymptotic behavior at the (super-)harmonic level --- The extrinsic curvature]
\label{theo:deuxieme-2} 
Under the conditions in Theorem~\ref{theo:deuxieme}, assume now that the strongly tame conditions \eqref{eq:strongtameH}
and \eqref{eq:strongtameM} hold.  
Then, the extrinsic curvature $h$ 
is asymptotic to 
\bel{eq:217}
\htt_1 := h_1 -{1 \over 2}  \sum_{1 \leq a \leq n} \chi_a  \Lcal_{\Pt_a \cdot V} g_1,  
\ee
referred to as the {\bf effective seed extrinsic curvature}, 
for all (super-)harmonic exponents $\pstar, \qstar$ satisfying \eqref{eq:strongexpo}): 
\begin{subequations} 
\bel{q*<3} 
\| h - \htt \|_{C_{\qstar}^{2,\alpha}(\Mbf)} \lesssim \Fcal(g_1, h_1, \Hstar, \Mstar)
+  \max_{1 \leq a \leq n} \big( \eeG \, | m_a^* |  + |P_a^* | \big). 
\ee
Here, $\Fcal = \Fcal(g_1, h_1, \Hstar, \Mstar)$ measures the ``accuracy'' of the seed data: 
\be
\aligned
\Fcal
:= \, &  (\eeG + \eeM) \, \| \Hcal(g_1,h_1) - \Hstar \|_{C^{\alpha}_{\pstar + 2}(\Mbf)}  
 + \, \| \Mcal(g_1,h_1) - \Mstar \|_{C^{1,\alpha}_{\qstar + 1}(\Mbf)} 
\endaligned 
\ee 
and, for each asymptotic end $N_a$, the
{\bf average of the momentum data}  $P^*_a = P^*_a(g_1, h_1, \Mstar)$ is defined as 
\be
P^*_a  := {1 \over 8 \pi} \int_\Mbf  \chi_a \big( \Mstar - \Mcal( g_1, h_1) \big) \, d V_{g_1}, 
\ee
while the {\bf momentum correctors} $\Pt_a$ satisfy the estimate 
\bel{q*<3b} 
\aligned
\max_{1 \leq a \leq n}
\big|  \Pt_a - P_a^* \big| 
& \lesssim 
 \eeG \, \| \Hcal(g_1,h_1) - \Hstar \|_{L^2C_{5/2}^{\alpha}(\Mbf)}
   + \eeG \, \| \Mcal(g_1,h_1) - \Mstar \|_{L^2C_{5/2}^{1, \alpha}(\Mbf)}.   
\endaligned
\ee
In particular, as well as in the harmonic case $q_* = 2$, one has $\lim_{r \to +\infty} \big( r^2 \, | h - \htt_1 | + r^3 \, |\del (h-\htt_1)| \big) = 0$.
\end{subequations} 
\end{theorem}


\paragraph{Heuristic observations.}

The strongly tame decay conditions can be motivated from the schematic expressions (for the perturbations $g_2 = g- g_1$ and $h_2=h- h_1$)
$$
\aligned
\Delta g_2 & = \big( \Hcal(g_1, h_1) - \Hstar \big) + h * h_2 + \del g * \del g_2, 
\qquad \qquad 
\Div (h_2) = \big( \Mcal(g_1, h_1) - \Mstar \big) + \del h * g_2. 
\endaligned
$$
Differentiating the momentum equation, we consider the second-order formulation obtained by replacing it by
$
\Delta h_2 + \Div(\nabla h_2) = \del \Big( \Mcal(g_1, h_1) - \Mstar  + g_2 * \del h  \Big).
$
If $(p, q)$ denotes the pair of decay exponents associated with the solution perturbation $(g_2, h_2)$, then the study of the second-order elliptic system above will lead us to the conditions 
$p \leq \min(p_M, 2q - 2, q_G + q - 2, 2p, p_G + p)$ and 
$q \leq \min(q_M, p + q_G, p + q).$
The conditions \eqref{eq:strongtameH} and \eqref{eq:strongtameM} 
 arise in order to reach the harmonic decay exponents $p = 1$ and $q = 2$, respectively. The definition of $(\pstar, q_*)$ in \eqref{eq:strongexpo} is similarly motivated.


\subsection{ADM mass, ADM momentum, and examples}
\label{sec:ADM-exam}
 
\paragraph{Mass and momentum.}  

Whenever the expressions below are well-defined, the {ADM mass} $\mADM(\Mbf,g)$ and {ADM momentum} $\PADM(\Mbf,h)$ are defined as the following scalar and vector in $\RR^3$, respectively, by  
\be
\aligned  
\mADM(\Mbf,g) & \coloneqq {1 \over 16 \pi} \lim_{r \to +\infty} \int_{S_r} 
\sum_{i,j = 1}^3 \big( g_{ij,i} - g_{ii,j} \big) {x_j \over r} \, d\omega, 
\qquad
\PADM_i(\Mbf,h) 
\coloneqq {1 \over 8 \pi} \lim_{r \to +\infty} \int_{S_r} \sum_{1 \leq j \leq 3} h_{ij} {x_j \over r} \, d\omega, 
\endaligned
\ee
where $d\omega$ denotes the standard measure on the unit sphere $\RR^2$ and $S_r$ is the sphere with radius $r$ (in the coordinate chart given in the asymptotic end). These formulas are valid when the manifold admits a single asymptotic end, and is easily extended to manifolds with several ends by considering the contributions at each end, which we denote by 
  $\mADM_a(\Mbf,g)$ and  $\PADM_a(\Mbf,g)$ for $a=1,\ldots, n$.
Our asymptotic expansion provides us with some expressions of the mass and momentum of the actual solutions. 

\begin{corollary}[Mass and momentum properties of the seed-to-solution map] 
\label{corol:admmm} 
The ADM mass and ADM momentum of the Einstein solution $(g,h)$ in Theorems~\ref{theo:deuxieme} and \ref{theo:deuxieme-2}
enjoy the following decomposition at each end. 
Provided $\ppG > 1/2$ and $\qqG > 3/2$ as well as\footnote{Since we already assumed $\Hcal(g_1,h_1) - \Hstar \in L^1(\Mbf)$, it follows that $\Hstar \in L^1(\Mbf)$  since $h_1$ has sufficient decay.}
   $R_{g_1} \in L^1(\Mbf)$ 
(so that the mass and momentum are finite), one has
$$
\aligned
\mADM _a(\Mbf, g) & = \mADM_a (\Mbf, g_1) +
\mt_a, 
\qquad\qquad
\PADM_a(\Mbf, h) = \PADM _a(\Mbf, h_1) + 
\Pt_a.
\endaligned
$$ 
\end{corollary}

Recall also that,  
by the positive mass theorem~\cite{BC96,SchoenYau,Wit81},
 if the matter data $(\Hstar, \Mstar)$ satisfies the 
dominant energy condition\footnote{specifically, $\Hstar \geq |\Mstar|_g$}
{\sl with respect to} the constructed solution $(g,h)$,  
we find 
$\mADM (\Mbf, g) \geq |\PADM (\Mbf, h)|_g$, 
which is a strict inequality 
unless $(\Mbf, g , h)$ is a Cauchy initial data for the Minkowski spacetime, i.e.~$(\Mbf,g)$ can be isometrically embedded in Minkowski spacetime with second fundamental form $k$.

\bse
\bei

\item Hence, the positive mass theorem implies that  
\be
\mADM (\Mbf, g_1) + \sum_{1 \leq a \leq n} \mt_a
\geq
\Big|
\PADM (\Mbf, h_1) +  \sum_{1 \leq a \leq n} \Pt_a \Big|_g, 
\ee
which is a {\sl lower bound} on the mass correctors $\sum_{1 \leq a \leq n} \mt_a$. 

\item If the matter data also satisfies the positive energy condition {\sl with respect to the seed data}, 
then  
  $\mADM (\Mbf, g_1) \geq |\PADM (\Mbf, h_1)|_{g_1}$
   or 
\be
\mADM (\Mbf, g) - \sum_{1 \leq a \leq n} \mt_a
\geq 
\Big|
\PADM (\Mbf, h) -  \sum_{1 \leq a \leq n} \Pt_a \Big|_{g_1}, 
\ee 
which is an {\sl upper bound} on the mass correctors $\sum_{1 \leq a \leq n} \mt_a$. 

\eei 

\ese 


\paragraph{Several regimes of interest.}
 
It is interesting to consider special values of the exponents arising in Theorems~\ref{theo-prec} and \ref{theo:deuxieme}. For simplicity, in our examples we assume that $\Hstar$ and $\Mstar$ vanish identically. 
\begin{subequations} 
\bei

\item {\bf Seed metric with slow decay.} The slowest possible decay allowed in our theorems is as follows. 

\bei

\item {\bf Tame data.} In Theorem~\ref{theo-prec}, for any sufficiently small $\delta \geq 0$ we can choose
\be  
(\ppG, \qqG) = (1/2, 5/4), 
\qquad 
(\ppM, \qqM) = (p,q) = (1/2 + \delta, 5/4+\delta).
\ee
Hence, Theorem~\ref{theo-prec} shows that a prescribed data set that decays much slower than Schwarz\-schild and in fact with infinite ADM mass, 
generates a solution $(g,h)$ such that $g - g_1$ decays slower than Schwarzschild, but slightly faster than the seed data $g_1$ if one has $\delta >0$.  Theorem~\ref{theo-prec} applies even when $\ppG = \ppM=p=1/2$ so that the seed data, the matter, and the solution enjoy the same decay in this regime ---however, it must be observed that a slightly stronger norm is used to control the matter based on a weighted $L^2$ norm. 

\item {\bf Strongly tame data.} In Theorem~\ref{theo:deuxieme}, for any sufficiently small $\delta \geq 0$ 
we can choose 
\be  
(\ppG, \qqG) = (1/2, 5/4), 
\qquad 
(\ppM, \qqM) = (p,q) = (1 + \delta, 2+\delta).
\ee 
In this regime, the remainder $g - g_1$ enjoys the harmonic decay at infinity when $\delta \geq 0$, and the remainder $g - \gt_1$ has even  super-harmonic decay when $\delta >0$. 
\eei


\item{\bf Seed metric with fast decay.} For any sufficiently small ${\delta'} \geq \delta \geq 0$ we can choose our exponents to be
\be  
(\ppG, \qqG) = (1 + \delta, 2+\delta), 
\qquad 
(\ppM, \qqM) = (\pstar,\qstar) = (1 + {\delta'}, 2+{\delta'}).
\ee
Hence, our method allows us to generate a solution perturbation enjoying a (super-)harmonic behavior at infinity depending whether we choose $\delta'=0$ or $\delta'>0$. 
 
\eei  

\end{subequations}

\paragraph{A class of strongly tame seed data sets.}
 
It is not difficult to check that our collection of seed data sets is non-empty. We take again $M= \RR^3$, $\Hstar=0$, and $\Mstar = 0$, and we choose any two-tensor $h_1 \in L^2C^{2,\alpha}_3(\RR^3)$ together with an arbitrary decay exponent $\ppG \in (1/2, 1]$. For any given $\eps \in (0,1)$, we can find a Riemannian metric $g_1$ on $\RR^3$ such that
  $\|g_1 - \delta\|_{L^2C^{4,\alpha}_{\ppG}(\RR^3)} \leq \eps$ as well as
$\|R_{g_1} \|_{C^{2,\alpha}_{2\ppG + 2}(\RR^3)} \leq \eps^2$. 
For instance, we can choose the class of metrics 
$g_1 = 2 v \, dx^1dx^2 + (1 + u)  \sum_{i=1}^3(dx^i)^2$, 
in which the metric coefficients $u, v \in C^{2,\alpha}_{\ppG}(\RR^3)$ are chosen to satisfy 
$\Delta u = \del_1\del_2  v$. 
Namely, such functions exist since, for instance, we can pick up an arbitrary function $v \in C^{2, \alpha}_{\ppG}(\RR^3)$, which, therefore, 
satisfies 
$\del_1\del_2 v \in C^{0,\alpha}_{\ppG+2}(\RR^3)$,
   and the existence of a function $u$ is guaranteed by Proposition~\ref{propo.classcialregular}, below.
Moreover, one has 
$R_{g_1} =  - \Delta g_{1ii} +  \del_{ij} g_{1ij} + \del g_1 * \del g_1
  = 2\del_1\del_2 v - 2\Delta u + \del g_1 * \del g_1
  = \del g_1 * \del g_1$, 
which enjoys the desired decay. 
Our theory in Theorems \ref{theo:deuxieme} and \ref{theo:deuxieme-2} applies to this seed data set (which has $\ppM = 2\ppG > 1$).

\paragraph{Schwarzschild metric.}  

Consider the effect of a change of coordinates at infinity on the Hessian term arising in the metric. 
Let us assume that, in some given coordinates 
at infinity, an Einstein solution $(g , h)$ has the following expansion near infinity (for some constants $m_1 , m_2$): 
\bel{equa=gm}
g = \big(1 + 2m_1 /r \big) \, e + m_2 r^2 \, \Hess_e (1/r) + \obig( 1/r). 
\ee 
A natural question to ask is whether $g$ is in fact asymptotic to the Schwarzschild metric itself and, indeed, we claim that
there exists an asymptotic coordinate system $(\xt^j) = (\xt^j(x))$ such that 
\bel{gSchcoor}
\gt_{ij} = \big(1 + 2 m/ \rt \big) \, \delta_{ij} + \ocal(1/\rt)
\ee
 for some mass coefficient $m$ which we can determine from $m_1, m_2$, as follows. 
Here, $(\rt)^2 = \sum_{i = 1}^3 (\xt^i)^2 \geq 1$ and $\rt = r + o(1)$. 
  (We use a  
standard notation for functions $o(1)$ that tend to zero at infinity.) 
We search for such a coordinate transformation $x \mapsto \xt$ in the form 
$
\xt = C x + \varphi,
$
where $C = (C_i^j)$ is a  
$3\times 3$ matrix 
and the vector-valued function $\varphi =(\varphi^j(x))$ (in the weighted space $C_0^{2,\alpha}$, say) is bounded at infinity. In view of the identity $g_{ij} = \gt_{kl} \, \del_i \xt^k \del_j \xt^l$, \eqref{equa=gm} and \eqref{gSchcoor} lead us to the condition 
$$
\aligned
\big(1 + 2m_1/r \big) \, \delta_{ij} + m_2 \Big( - \delta_{ij} /r + 3 x_i x_j / r^3 \Big) 
&
= (C_i^k + \del_i \varphi^k) (C_j^l + \del_j \varphi^l) \big(1 + 2m/r \big)  \delta_{kl} 
  + \ocal( 1/r)
\\
&= \big(1 + 2m/r \big) (C_{ki} C_j^k + C_{kj} \del_i \varphi^k + C_{ki} \del_j \varphi^k) 
   + \ocal( 1/r). 
\endaligned
$$
Therefore,  we must impose $C_{ki} C_j^k = \delta_{ij}$ and 
$ 
C_{ki} \del_j \varphi^k + C_{kj} \del_i \varphi^k 
= (- 2m + 2m_1 - m_2) \delta_{ij}/r + 3m_2 x_i x_j/r^3. 
$
In order to handle the corrector term arising in the expansion of the metric, the simplest choice is obtained by taking $C_i^j = \delta_i^j$ and we can choose a solution to the non-homogeneous 
Killing equation $\del_j \varphi_i + \del_i \varphi_j 
= (- 2m + 2m_1 - m_2) \delta_{ij}/r + 3m_2 x_i x_j/r^3$, for instance $\varphi(x) = a x/r$ with $a=-(3/2) m_2$ and 
therefore $m=m_1 + m_2$.


\section{Existence of the seed-to-solution map for general tame data sets}  
\label{section--3}

\subsection{The linearized Einstein operator and its adjoint}

\paragraph{Weighted Sobolev spaces.}

We consider first the adjoint of the linearization $d\Gcal$  of the Einstein operator around a given seed data set $(g_1, h_1, \Hstar, \Mstar)$, and we solve a linearized version of our problem in which the unknowns are associated with the adjoint operator $d\Gcal_{(g_1, h_1)}^*$. These adjoint variables consist of a scalar field $u$ and a vector field $Z$ sought in (suitably weighted) Sobolev spaces. 
The system under consideration couples together second- and first-order elliptic equations. 
For the convenience of the reader, the most technical proofs in this section are postponed to Appendix~\ref{appendix-AAA}. 

Recall that a background manifold $\Mbf = (\Mbf,e,r)$ is fixed throughout and all of our statements assume that $\eps_*$ is sufficiently small and fixed. For any real $\theta>0$  
we define the weighted Sobolev space $H^k_\theta(\Mbf)$ by completion from the set of all smooth and compactly supported functions $f: M \to \RR$ with finite weighted Sobolev norm 
\bel{equa-scale31}
\|f\|^2_{H_\theta^k(\Mbf)}
\coloneqq \sum_{\text{charts}} \sum_{|K| \leq k} \int | \del^K f|^2 r^{-3 + 2 |K| + 2 \theta} \, \chichart \, dV_e, 
\ee
where $\chi_{\text{chart}}$ is a given partition of unity compatible with the partition \eqref{equa=chif}. The integration is performed over our (finite) covering of $\Mbf$ while $dV_e$ denotes the volume form of $(\Mbf,e)$. As already pointed out we use the same notation for tensor fields, by considering components in the prescribed coordinate charts of $(\Mbf,e,r)$.

\paragraph{Linearization around a seed data set.}

Given any seed data $(g_1,h_1)$, from \eqref{eq:Einstein00} we can compute the {linearized Einstein operator} 
$d\Gcal_{(g_1,h_1)}: (g_2, h_2) \mapsto d\Gcal_{(g_1,h_1)}[g_2, h_2]$, 
together with the {adjoint Einstein operator} 
$d\Gcal^*_{(g_1,h_1)}: (u,Z) \mapsto d\Gcal^*_{(g_1,h_1)}[u,Z]$.
Recall that the formal adjoint  in the $L^2$ sense is defined by
\be
\int_\Mbf  (g_2, h_2) \cdot d\Gcal_{(g_1,h_1)}^*[u, Z] \, dV_{g_1}
\coloneqq 
\int_\Mbf  d\Gcal_{(g_1,h_1)}[g_2, h_2] \cdot (u,Z) \, dV_{g_1}, 
\ee
where the dot notation is defined with the metric $g_1$. 
The expressions were derived in~\cite{FischerMarsden-1975a,FischerMarsden-1975b}.

Namely, the {\bf linearized Hamiltonian and momentum operators} read
\begin{subequations}
\bel{eq-linearHM}
\aligned
d\Hcal_{(g_1,h_1)}[g_2, h_2]
= 
& - \Delta_{g_1} \big( \Tr_{g_1}(g_2) \big) + \Div_{g_1} \big( \Div_{g_1} g_2  \big) 
   - g_1\big(g_2, \Ric_{g_1} \big) + h_1 * h_1 * g_2 + h_1*h_2,
\\
d\Mcal_{(g_1,h_1)}[g_2, h_2] 
= \, & \Div_{g_1}(h_2) + h_1*\nabla_{g_1}g_2, 
\endaligned
\ee
which are second-order (scalar-valued) and first-order (vector-valued) operators, respectively. 
Here, $\nabla_{g_1}$, $\Div_{g_1}$, $\Delta_{g_1}$, and $\Ric_{g_1}$ denote the Levi-Civita connection, divergence operator, Laplace operator, and Ricci curvature associated with the seed metric $g_1$, respectively. 
The {\bf adjoint Hamiltonian and momentum operators} are more involved and, in a schematic form,  read
\bel{eq-linearHMstar}
\aligned
d\Hcal_{(g_1,h_1)}^* [u,Z]
= & - \big(\Delta_{g_1} u\big)g_1 + \Hess_{g_1}(u) - u \, \Ric_{g_1} + h_1*h_1*u + \nabla_{g_1}(h_1*Z),
\\
d\Mcal_{(g_1,h_1)}^*[u,Z] 
= &-{1 \over 2}  \Lcal_Z g_1 + h_1*u, 
\endaligned
\ee
which are second-order and first-order operators, respectively. (Recall that $\Lcal$ stands for the Lie derivative operator.) 
\end{subequations}


\paragraph{Full expression of the adjoint operator.}

More precisely, the adjoint Hamiltonian operator at $(g_1,h_1)$ reads 
\begin{subequations}
\bel{eq=formuleexactes}
\aligned
d\Hcal^*_{(g_1,h_1)}[u,Z]
& = - (\Delta_{g_1}u)g_1 + \Hess_{g_1}(u) 
      + \Big( - \Ric_{g_1} + (\Tr_{g_1} h_1) h_1 - 2 h_1 \times h_1 \Big) \, u - {1 \over 2} \Lcal_Z h_1  
\\
& \quad   + {1 \over 2} \Div(Z)h_1
 - {1\over 2} \Big(
Z\otimes\Div_{g_1} h_1 + \Div_{g_1} h_1\otimes Z
\Big)^\flat
 +  {1\over 4}g_1\big(\Lcal_Z g_1,h_1\big) - {1\over 2}g_1(Z,\Div_{g_1} h_1),
\endaligned
\ee
where the flat symbol ${}^\flat$ is the contravariant vs.~covariant transformation of a tensor via the metric duality based on $g_1$, and the notation $h_ 1 \times h_1 \coloneqq h_1^{ik} h_{1k}^j$ is used. 
On the other hand, the full expression of the adjoint momentum operator at $(g_1,h_1)$ is 
\bel{eq=formuleexactes2}
\aligned
d\Mcal^*_{(g_1,h_1)}[u,Z] 
& \coloneqq -{1 \over 2} \Lcal_Z g_1 + \Big((\Tr_{g_1} h_1)g_1^{-1} - 2 h_1\Big) \, u. 
\endaligned
\ee
\end{subequations}
A special case of interest is obtained by taking the tensor $h_1$ and the vector field $Z$ to vanish identically, and the adjoint Hamiltonian operator then reduces to the {adjoint scalar curvature operator}
\bel{equa-Letoile}
L^*_{g_1}(u)
 \coloneqq 
d\Hcal_{(g_1,0)}^* [u,0]
= 
- (\Delta_{g_1} u) g_1 + \Hess_{g_1}(u) - u \, \Ric_{g_1}. 
\ee 
We will also use the adjoint momentum operator obtained by taking the tensor $h_1$ to vanish identically, which is essentially the {Lie derivative operator} 
\bel{equa-LieD}
M^*_{g_1}(Z) :=  d\Mcal_{(g_1,0)}^*[0,Z] = - {1 \over 2}  \Lcal_Z g_1. 
\ee 


\paragraph{Einstein operators in weighted Sobolev spaces.} 

The scale of weighted Sobolev spaces defined in \eqref{equa-scale31} is now used.  
 
\beginmyproposition
Given any tame seed data set $(g_1, h_1, \Hstar, \Mstar)$ on a background manifold $\Mbf = (\Mbf,e,r)$ with (admissible) decay exponents $(\ppG, \qqG)$ and $(\ppM, \qqM)$,  
the images of the linearized Einstein operator and its adjoint are such that, for any effective  exponents $(p,q)$ in the sense 
of \eqref{eq--2-12}, 
the operators 
$d\Gcal_{(g_1, h_1)} : H^2_{p}(\Mbf) \times H_{\oldq}^1(\Mbf) \to L^2_{p+2}(\Mbf) \times L^2_{\oldq+1}(\Mbf)$ and 
$d\Gcal_{(g_1, h_1)}^*: H^{2}_{1-p}(\Mbf) \times H^1_{2-\oldq}(\Mbf) \to L^2_{3-p}(\Mbf) \times L^2_{3-\oldq}(\Mbf)$ 
are bounded operators. 
\end{proposition}

\begin{proof}  
We treat first the linearized Hamiltonian operator $d\Hcal_{(g_1,h_1)}[g_2, h_2]$ and consider each term arising in \eqref{eq-linearHM}. Denoting by $\Gamma_{ij}^k(g_1)$ the Christoffel symbols in the chart prescribed in any of the asymptotic ends, since $(g_1, h_1)$ is tame and $(g_2,h_2) \in H^2_{p}(\Mbf) \times H_{\oldq}^1(\Mbf)$ we obtain $\Gamma_{ij}^k(g_1) \simeq r^{-P_G-1}$ and, with some obvious notation, 
$$
\aligned
\Delta_{g_1} \big( \Tr_{g_1}(g_2) \big) 
& \simeq r^{-p-2} + r^{-p_G-1-p-1} \lesssim r^{-p-2}, 
\qquad
\Div_{g_1} \big( \Div_{g_1} g_2  \big) 
 \simeq r^{-p-2} + r^{-p_G-2-p} + r^{-p_G-1-p-1} \lesssim r^{-p-2}, 
\\
g_1\big(g_2, \Ric_{g_1} \big) 
& \simeq r^{-p - p_G-2} \lesssim r^{-p-2}, 
\qquad
h_1 * h_1 * g_2 + h_1*h_2 
 \simeq r^{-2 q_G-p} + r^{-q_G-q} \lesssim r^{-p-2}. 
\endaligned
$$
  (The notation $\simeq$ is used here between functions enjoying a decay of the same order at infinity, and we write 
  $\lesssim$ when the left-hand decays at least as fast as the right-hand side.) 
For the latter condition, we require that $p + 2 - \qqG \leq \oldq \leq p + \qqG$, as stated in our definition. 
On the other hand, for the linearized momentum operator $d\Mcal_{(g_1,h_1)}[g_2, h_2]$ in \eqref{eq-linearHM}, we obtain 
$\Div_{g_1}(h_2) \lesssim  r^{-q-1}$ 
and  
$h_1*\nabla_{g_1} g_2 \lesssim r^{-q-1}$. 
Finally, dealing with the adjoint (Hamiltonian and momentum) operators is similar. 
\end{proof}

\paragraph{Technical estimates.} 

We state here standard inequalities in weighted functional spaces and, specifically, the statements 
given in the following two lemmas 
go back to 
 the works~\cite{CorvinoSchoen} and \cite{CarlottoSchoen}; 
 for the convenience of the reader, their proofs is also provided in Appendix \ref{appendix-AAA}. 
By definition,  in the Euclidean space $\RR^3$, our radius function $r=r(x)$ is bounded below by a positive constant and coincides with the usual distance function $|x|$ outside a fixed ball.  

\begin{lemma}
\label{lemma WP1}
For any exponent $a>0$ and any function $w: \RR^3 \to \RR$ tending to zero at infinity, one has 
$\| w \|_{L^2_{a}(\RR^3)} \lesssim \| \nabla w \|_{L^2_{a + 1}(\RR^3)}$, 
as well as  $\| w \|_{H^2_{a}(\RR^3)} \lesssim \big\| \Hess_{g_\Eucl}(w) \big\|_{L^2_{a + 2}(\RR^3)}$, 
with implied constants depending upon $a$ only.
\end{lemma}

On the other hand, let us consider the Killing operator in the Euclidean space $\RR^3$. Recall that the Killing operator $\Dcal$ is defined on any Riemannian manifold $(\Mbf,g)$ by 
$\Dcal(Y)(Z,W) \coloneqq g(\nabla_Z Y,W) + g(\nabla_W Y, Z)$, 
where $\nabla$ denotes the Levi-Civita connection associated with the metric under consideration. 

\begin{lemma}
\label{prop. weighted KKK}
Given any exponent $q \in ( - \infty, 2)$, the inequality 
$\|Z\|_{H^1_{2 - q}(\RR^3)} \lesssim \| \Dcal(Z)\|_{L^2_{3 - q}(\RR^3)}$ 
holds  for any vector field $Z$ defined on $\RR^3$ and tending to zero at infinity, 
in which the implied constant depends upon $q$ only. 
\end{lemma}

\paragraph{An invertibility property.}
Relying on weighted functional inequalities, we are in position to establish a fundamental property of the linearization. 
Observe that our functional setup (namely the choice of weights) 
here  rules out the kernel of the adjoint operators. 

\begin{proposition}[Invertibility on its image for the adjoint Einstein operator]
\label{prop.key}
Consider any tame seed data set $(g_1, h_1, \Hstar, \Mstar)$ on a background manifold $\Mbf = (\Mbf,e,r)$ with decay exponents $(\ppG, \qqG, \ppM, \qqM)$. 
Then for any effective exponents $(p,q)$
and for all $u \in H_{1-p}^2(\Mbf)$ and $Z \in H_{2-\oldq}^1(\Mbf)$, one has (with implied constant depending upon the decay and regularity exponents)
$$ 
\aligned
\| u \|_{H^2_{1-p}(\Mbf)} 
& \lesssim 
\big\|d\Hcal^*_{(g_1,h_1)}[u,Z]\big\|_{L^2_{3-p}(\Mbf)}
+ \eeG \, \big\|d\Mcal^*_{(g_1,h_1)}[u,Z]\big\|_{L^2_{3-\oldq}(\Mbf)},
\\
 \| Z \|_{H^1_{2-\oldq}(\Mbf)} 
& \lesssim 
\eeG \, \big\|d\Hcal^*_{(g_1,h_1)}[u,Z]\big\|_{L^2_{3-p}(\Mbf)}
+ 
\big\|d\Mcal^*_{(g_1,h_1)}[u,Z]\big\|_{L^2_{3-\oldq}(\Mbf)}. 
\endaligned
$$
\end{proposition}

We recall that, throughout, the coefficients\footnote{The parameter 
$\eps_*$ plays a minor role and, for instance, can be taken to vanish when the background can be chosen to be $\RR^3$. On the other hand, $\eeG$ arises when 
we compare the seed metric to the background metric. 
}
 $\eps_*$ and $\eeG$ are sufficiently small. If one is not interested in the dependency of the constant in $\eeG$, one can rewrite the estimate in Proposition~\ref{prop.key} in the (slightly less precise) form 
$
\|(u,Z)\|_{H^2_{1-p}(\Mbf) \times H^1_{2-\oldq}(\Mbf)} 
\lesssim 
\big\|d\Gcal^*_{(g_1,h_1)}[u,Z]\big\|_{L^2_{3-p}(\Mbf) \times L^2_{3-\oldq}(\Mbf)}. 
$
It is convenient to decompose the proof into several steps, and 
Proposition~\ref{prop.key} follows immediately once we establish the following two technical lemmas. 
The second lemma below follows from a standard perturbation and its proof is postponed to Appendix~\ref{appendix-AAA}.


\beginmylemma[Adjoint Einstein operator for the background metric]
\label{prop. weighted Poincare-curved} 
1. Given any $p \in ( - \infty, 1)$, one has
$
\| w \|_{H^2_{1-p}(\Mbf,e,r)} 
\lesssim
\| L_e^* (w) \|_{L^2_{3-p}(\Mbf,e,r)}$
for all $w \in H^2_{1-p}(\Mbf,e,r)$, 
in which the implied constant may depend upon $p$ (as well as $(\Mbf,e,r)$). 

2. Given any $q \in ( - \infty, 2)$, for the operator in \eqref{equa-LieD} one has 
$
\| W \|_{H^1_{2 - q}(\Mbf,e,r)} 
\lesssim
\| M^*_e(W) \|_{L^2_{3-q}(\Mbf,e,r)}$ for all $W \in H^1_{2 - q}(\Mbf,e,r)$, 
in which the implied constant may depend upon $q$ (as well as $(\Mbf,e,r)$). 
\end{lemma}

\beginmylemma[Adjoint Einstein operator close to the background metric]
\label{lemma-onemore}
Under the assumptions of Proposition~\ref{prop.key}, one has
$$
\aligned
\big\| d\Hcal^*_{(g_1,h_1)}[u,Z] - d\Hcal^*_{(e,0)}[u,Z] \big\|_{L^2_{3-p}(\Mbf)} 
& \lesssim  
\eeG \, \|(u,Z)\|_{H^2_{1-p}(\Mbf) \times H^1_{2 - \oldq}(\Mbf)},
\\
\big\| d\Mcal^*_{(g_1,h_1)}[u,Z] - d\Mcal^*_{(e,0)}[u,Z] \big\|_{L^2_{3-q}(\Mbf)} 
& \lesssim 
\eeG \, \|(u,Z)\|_{H^2_{1-p}(\Mbf) \times H^1_{2 - \oldq}(\Mbf)}.
\endaligned
$$
\end{lemma}


\begin{proof}[Proof of Lemma \ref{prop. weighted Poincare-curved}]
\bse
{\bf 1. Hamiltonian operator.} 
{\bf 1a.}  
We establish first the following Poincar\'e-type inequality on the background manifold $(\Mbf,e,r)$: 
\bel{iq.2derive}
\| w \|_{H^2_{1-p}(\Mbf,e,r)} \lesssim \| \del^2 w \|_{L^2_{3 - p}(\Mbf,e,r)}. 
\ee
Proceeding by contradiction, if \eqref{iq.2derive} does not hold, we can find a sequence  $w^n: M \to \RR$ such that
\bel{aeq:w}
\| w^n \|_{H^2_{1-p}(\Mbf,e,r)} = 1, 
\qquad 
\lim_{n \to +\infty} \| \del^2 w^n \|_{L^2_{3 - p}(\Mbf,e,r)} = 0.
\ee
Passing to a subsequence if necessary, we can assume that $w^n$ converges weakly to a limit $w \in H^2_{1-p}(\Mbf,e,r)$. Since $w, \del w$ 
 enjoy integrability conditions at infinity, 
we deduce from \eqref{aeq:w} that the limit $w \equiv 0$. Moreover, by the Sobolev theorem, $w^n$ converges strongly to $0$ in $H^1(\Mbf, e)$, hence in $H^2(\Mbf, e)$ by \eqref{aeq:w}. (Recall here that the embedding $H^2_\theta \hookrightarrow H^1_{\theta'}$ is compact provided $\theta > \theta'$; see, for instance,   \cite[Lemma 2.1]{ChoquetC}.)

Now we let $\xi: M \to [0,1]$ be a cut-off function such that $\xi$ is identically $1$ outside  some ball of sufficiently large radius, chosen so that 
its support is included in the union of the asymptotic ends which we have denoted by $N_1, N_2,  \ldots, N_n$. 
On one hand, since $w^n$ converges to $0$ in $H^2(\Mbf, e)$, we see that $(1 - \xi) w^n$ converges to $0$ in $H^2_{1 - p}(\Mbf,e,r)$. In combination with  the property $\| w^n \|_{H^2_{1-p}(\Mbf,e,r)} = 1$, we get
$\lim_{n \to +\infty} \|\xi w^n\|_{H^2_{1-p}(\Mbf,e,r)} = 1$. 
On the other hand, since $\del^2 w^n$ converges to $0$ in $L^2_{3 - p}(\Mbf,e,r)$ and since $w^n$ converges to $0$ in $H^2(\Mbf, e)$, we see that $\del^2(\xi w^n)$ converges to $0$ in $L^2_{3 - p}(\Mbf,e,r)$. Therefore, since $\lim_{n \to +\infty} \|\xi w^n\|_{H^2_{1-p}(\Mbf,e,r)} = 1$ it follows that
$
\lim_{n \to +\infty} 
{\| \del^2  (\xi w^n) \|_{L^2_{3 - p}(\Mbf,e,r)}  / \|\xi w^n\|_{H^2_{1-p}(\Mbf,e,r)}}
 = 0.
$
However, the function $\xi w^n$ may be regarded (for each end) to be defined on $\RR^3$, so this then contradicts the statement in Lemma~\ref{lemma WP1}. Hence \eqref{iq.2derive} holds as claimed.  


\vskip.16cm 

{\bf 1b.} Next, by our definition of $\eps_*$ we have
$
\|  \Ric(e)^k_{ij} \del_k w \|_{L^2_{3 - p}(\Mbf,e,r)} 
\lesssim \eps_*  \| \del w\|_{L^2_{2-p}(\Mbf,e,r)}
$
and, provided $\eps_*$ is sufficiently small,  from \eqref{iq.2derive} we deduce that 
\bel{eq.Hess}
\| w \|_{H^2_{1-p}(\Mbf,e,r)} \lesssim \| \Hess_e w \|_{L^2_{3 - p}(\Mbf,e,r)}. 
\ee
In each asymptotic end and by denoting by $\delta$ the Euclidean metric in the coordinate chart at infinity, with the notation \eqref{equa-Letoile} we can write 
$\Hess_e w = L^*_e w -{1 \over 2}\Tr ( L^*_e w) e + w * \del^2 e * e + \del^2 w * (e - \delta)$. 
Provided $\eps_*$ is sufficiently small, we obtain 
$
\| \Hess_e w \|_{L^2_{3 - p}(\Mbf,e,r)} \lesssim \|L^*_e w\|_{L^2_{3 - p}(\Mbf,e,r)}, 
$ 
which together with \eqref{eq.Hess} completes the argument for the first item of the lemma. 
\ese


\vskip.16cm

{\bf 2. Momentum operator.} We work in the coordinate charts chosen on $(\Mbf,e,r)$ and to any vector field $W$
we associate the  Killing operator $\Dcal(W)$, which for instance for the Euclidian metric
in Cartesian coordinates 
reads 
$(\Dcal(W))_{ij} = ( \del_j W^i + \del_i W^j) \, dx^i \otimes dx^j$ with $W = W^i d x^i$.
Thanks to Lemma~\ref{prop. weighted KKK} and an analysis similar to the derivation of \eqref{iq.2derive} above, we can check that  
$
\| W \|_{H^1_{2 - q}(\Mbf, e, r)} 
\lesssim
\| \Dcal(W) \|_{L^2_{3-q}(\Mbf,e,r)}$
for all $W \in H^1_{2 - q}(\Mbf,e,r)$. 
On the other hand, by the definition of the Lie derivative we have (again in the coordinate charts under consideration) 
$$
(\Lcal_W e - \Dcal(W))_{ij} = \del_k e_{ij} W^k  + (e_{kj} - \delta_{kj}) \del_i W^k +  (e_{ki} - \delta_{kj}) \del_j W^k,   
$$ 
so provided $\eps_*$ is sufficiently small, we obtain
$\| W \|_{H^1_{2 - q}(\Mbf,e,r)} 
\lesssim
\| \Lcal_W e \|_{L^2_{3-q}(\Mbf,e,r)}$. 
\end{proof}

\subsection{Variational framework for the linearized Einstein operator}
\label{sec-LEO}

Following Corvino and Schoen~\cite{CorvinoSchoen} 
and Carlotto and Schoen~\cite{CarlottoSchoen}, 
given any $(f,V)\in L^2_{p+2} \times L^2_{\oldq+1}(\Mbf)$ we consider the real-valued functional 
\begin{subequations}
\label{Gformula}
\be
\Jbf_{(g_1, h_1, f, V)} : H^2_{1-p}(\Mbf) \times H^1_{2-\oldq}(\Mbf) \ni (u,Z) \mapsto  \Jbf_{(g_1, h_1, f, V)}(u,Z)   \in \RR
\ee
defined by 
\be  
\aligned
& \Jbf_{(g_1, h_1, f, V)} (u,Z) 
\coloneqq 
\int_\Mbf  \Bigg(
{1 \over 2}  \, \big|d\Hcal_{(g_1,h_1)}^*[u,Z] \big|^2 r^{3-2p}
 + {1 \over 2} \, \big| d\Mcal_{(g_1,h_1)}^*[u,Z] \big|^2 r^{3-2\oldq}
- f u -  g_1(V,Z) 
\Bigg) \, dV_{g_1}.
\endaligned
\ee 
\end{subequations}
The {\bf Euler-Lagrange equation} for a minimizer $(u,Z)$ of the functional $\Jbf_{(g_1, h_1, f, V)}$ reads
\bel{eq:minmini}
\aligned
& d\Gcal_{(g_1,h_1))}[g_2,h_2] 
=  (f, V), 
\qquad 
 g_2  \coloneqq r^{3-2p} d\Hcal_{(g_1,h_1)}^*[u,Z], 
\qquad
h_2 \coloneqq r^{3-2\oldq} d\Mcal_{(g_1,h_1)}^*[u,Z]. 
\endaligned
\ee

\begin{theorem}[Existence theory for the adjoint   
version of the Einstein constraints]
\label{prop-exislinea}
Consider a tame seed data set $(g_1, h_1, \Hstar, \Mstar)$ on a background manifold $\Mbf = (\Mbf,e,r)$ with decay exponents $(\ppG, \qqG, \ppM, \qqM)$. 
Then for any effective  exponents $(p,q)$ in the sense \eqref{eq--2-12}
and for any
$
(f, V)\in L^2_{p+2}(\Mbf) \times L^2_{\oldq+1}(\Mbf),
$
there exists a unique minimizer $(u, Z) \in H^2_{1 - p}(\Mbf) \times H^1_{2 - \oldq}(\Mbf)$ of the adjoint Einstein functional $\Jbf_{(g_1, h_1, f, V)}$ satisfying \eqref{eq:minmini}.
\end{theorem}

\begin{proof} For simplicity in the proof\footnote{The assumption $(p, q) \leq (p_M, q_M)$ is not used in this proof.}
 we suppress the explicit dependence in $\eeG$. 
From the definition of $\Jbf_{(g_1, h_1, f, V)}$ and for some constants $C_1, C_2 >0$ we have 
$$
\aligned
\Jbf_{(g_1, h_1, f, V)}(u,Z)\geq 
& \, C_1\, \big\|d\Gcal^*_{(g_1,h_1)}[u,Z]\big\|^2_{H^2_{1-p} \times H^1_{2-\oldq}(\Mbf)}
  - C_2 \, \|(f,V)\|_{L^2_{p + 2} \times L^2_{\oldq+1}(\Mbf)} \|(u,Z)\|_{H^2_{1-p} \times H^1_{2-\oldq}(\Mbf)}
\endaligned
$$
and therefore, with  Proposition~\ref{prop.key}, 
$$
\aligned
\Jbf_{(g_1, h_1, f, V)}(u,Z)\geq 
C_1 \, \|(u,Z)\|^2_{H^2_{1-p} \times H^1_{2-\oldq}(\Mbf)}
 -C_2 \, \|(f,V)\|_{L^2_{p + 2} \times L^2_{\oldq+1}(\Mbf)} \|(u,Z)\|_{H^2_{1-p} \times H^1_{2-\oldq}(\Mbf)}.
\endaligned
$$
Therefore, the functional $\Jbf_{(g_1, h_1, f, V)}$ is coercive. It is a standard matter to show that it is lower semi-continuous in the spaces under consideration, and we conclude that $\Jbf_{(g_1, h_1, f, V)}$ admits at least one minimizer. 

Moreover, if both $(u_1,Z_1)$ and $(u_2,Z_2)$ are minimizers, then we find
$$
\aligned
&
\Jbf_{(g_1, h_1, f, V)} \big(\frac{(u_1,Z_1) + (u_2,Z_2)}{2} \big)
\\
& = 
{1 \over 2}  \, \Jbf_{(g_1, h_1, f, V)}(u_1,Z_1) + {1 \over 2}  \, \Jbf_{(g_1, h_1, f, V)}(u_2,Z_2)
-\frac{1}{8} \, \big\|d\Gcal^*_{(g_1,h_1)}[u_2-u_1,Z_2-Z_1]\big\|^2_{H^2_{1-p}(\Mbf) \times H^1_{2-\oldq}(\Mbf)}. 
\endaligned
$$ 
Therefore, we obtain $d\Gcal^*_{(g_1,h_1)}[u_2-u_1,Z_2-Z_1]= 0$, 
which implies $u_1=u_2$ and $Z_1=Z_2$ by Proposition~\ref{prop.key}. 
\end{proof}


\subsection{Uniform ellipticity of the adjoint Einstein operator}
\label{section Douglis-Nirenberg} 

\paragraph{H\"older regularity for elliptic systems.}

We continue our study of the linearized equations and we now investigate the weighted H\"older regularity of our solutions. We begin by recalling a general theory in the Euclidean space, or rather in a bounded domain $\Gamma \subset \RR^3$, as specified below. We present the setting in three-dimensions, since this is our application of interest. We consider first a system of $N$ linear partial differential equations in $\RR^3$ of the general form
\begin{subequations}
\bel{elliptic}
L_i(\cdot, \del)  w = \sum_{j=1}^N L_{ij}(\cdot, \del)  w_j = f_i, \qquad i =1, 2,\ldots, N, 
\ee
where the operator $L_{ij} = L_{ij}(x, \del)$ are polynomials of the partial derivatives $\del=(\del_1,\del_2,\del_3)$. Assume that there exist $2N$ integers, denoted by $s_1,\ldots,s_N$ and $t_1,\ldots,t_N$, such that for all relevant $x$ 
\be
L_{ij}(x, \del)  \text{ has order less or equal than } s_i + t_j. 
\ee
We denote $L^\prime_{ij}$ the sum of those terms in $L_{ij}$ that have exactly the order 
$s_i + t_j$, and, for $\xi \in \RR^3$, we refer to
\bel{eq41CC}
P(x,\xi) \coloneqq \text{det} \big( L^\prime_{ij}(x,\xi) \big)_{1\leq i,j \leq N}
\ee
as the \textbf{characteristic polynomial} associated with the operator \eqref{elliptic}. 
\end{subequations}

We consider the above operator in a bounded domain $\Gamma \subset \RR^3$ with sufficiently regular boundary, and let $d: \Gamma \to \RR$ be the distance function from the boundary $\del\Gamma$. For any integer $k \geq 0$ and reals $\theta \geq 0$ and $\alpha \in [0,1)$, we consider the {\bf weighted H\"older norm}
\be
\| w \|_{C_\theta^{k, \alpha}(\Gamma, d)}  
\coloneqq  \sum_{i= 0}^k \sup_{x \in \Gamma} d(x)^{\theta + i} | \del^i w(x)| + \sup_{x \in \Gamma} d(x)^{\theta+k+\alpha} [\del^k w]_\alpha, 
\ee
and we denote by $C_l^{k,\alpha}(\Gamma,d)$ the Banach space determined by completion (with respect to the above norm) of the set of all smooth functions on $\RR^3$ restricted to $\Gamma$. 

\begin{subequations}
\paragraph{Ellipticity conditions.} 

We decompose the operator in the form
$
L_{ij}(x,\del) = \sum_{| \beta|= 0}^{s_i + t_j}a_{ij,\beta} \, \del^\beta,
$
where the summation is over all multi-indices ordered by their length $|\beta|$. The following conditions are assumed for some $\alpha \in (0,1)$ and $K>0$ and for all indices $i,j = 1,\ldots,N$: 
\bei

\item[(1)] The coefficients $a_{ij,\beta}$ belong to $C^{-s_i, \alpha}_{-s_i - t_j + |\beta|}(\Gamma, d)$ with
$
\max_{i,j,\beta} \|a_{ij,\beta} \|_{C^{-s_i, \alpha}_{-s_i - t_j + | \beta|}(\Gamma, d)}
\leq K.
$

\item[(2)] The right-hand sides $f_i$ in \eqref{elliptic} belong to the space $C_{s_i + t}^{-s_i,\alpha}(\Gamma, d)$ with $t := \max_j t_j$.  

\item[(3)] With $m\coloneqq \sum _{k=1}^N(s_k + t_k)$, the characteristic polynomial satisfies the uniform positivity condition 
$
P(x,\xi) \geq K^{-1} \big(\sum_{i=1}^{n} \xi_i^2\big)^{m/2}$ for $x \in \Gamma, \, \xi \in \RR^N.
$
\eei
\end{subequations}
We now recall Douglis--Nirenberg's elliptic theory \cite[Theorem 1]{DouglisNirenberg}. 

\begin{theorem}[Interior H\"older regularity]
\label{propo.ellipticity}
Consider a solution $w: \Gamma \to \RR^N$ to the system \eqref{elliptic} under the ellipticity conditions (stated above) such that, for any $j=1,\ldots, N$, the component $w_j \in C_{t - t_j}^{(0,0)}(\Gamma,d)$ admits H\"older continuous derivatives up to order $t_j$. Then one has the higher regularity $w_j \in C^{t_j,\alpha}_{t  - t_j}(\Gamma,d)$ and, moreover, 
\be
\sum_{1 \leq j \leq N} 
\| w_j \|_{C^{t_j,\alpha}_{t  - t_j}(\Gamma,d)} 
\lesssim
\sum_{1 \leq j \leq N} 
\Big( 
\| w_j \|_{C^{0,0}_{t  - t_j}(\Gamma,d)} + \| f_j \|_{C^{-s_j,\alpha}_{s_j+ t}(\Gamma,d)}
\Big),
\ee
where the implied constant may depend upon $K, N, \alpha, s_1,\ldots,s_N, t_1, \ldots, t_N$.
\end{theorem}


Recall that a seed data set $(g_1, h_1, \Hstar, \Mstar)$ with exponents $(p_G, q_G, p_M, q_M)$
is prescribed together with effective exponents $(p,q)$. 
Thanks to Theorem~\ref{prop-exislinea}, for any $(f,V)\in L^2_{2 + \oldp}(\Mbf) \times L^2_{\oldq+2}(\Mbf)$ we can associate a unique minimizer of the functional $\Jbf_{(g_1, h_1, f, V)}$. This minimizer is  denoted by $(u, Z)\in H^2_{1-\oldp}(\Mbf)  \times H^1_{2-\oldq}(\Mbf)$ and enjoys the direct equations
\bel{eq1:hu}
\aligned
-\Delta_{g_1}(\Tr_{g_1}(g_2)) + \Div_{g_1}(\Div_{g_1}(g_2)) - g_1(g_2, \Ric_{g_1}) + h_1*h_1*g_2 + h_1*h_2 & = f,
\quad
\Div_{g_1}(h_2) + h_1*\nabla_{g_1}g_2  = V,
\endaligned
\ee
and the adjoint equations 
\bel{eq2:hua}
\aligned
-\big(\Delta_{g_1}u\big)g_1 + \Hess_{g_1}(u) -u\Ric_{g_1} + h_1*h_1*u + \nabla_{g_1}(h_1*Z) 
& = r^{-3 + 2\oldp} g_2, 
\qquad 
-{1 \over 2}  \Lcal_Z g_1 + h_1*u 
 = r^{-3 + 2\oldq} h_2. 
\endaligned
\ee
Taking the trace of the first equation in \eqref{eq2:hua} and the divergence of the second equation in \eqref{eq2:hua} into account, we obtain
$$
\aligned
-\Delta_{g_1}u & = \big(\Tr_{g_1}(g_1) - 1\big)^{-1} 
\Big(
u R_{g_1} - \Tr_{g_1}(h_1*h_1*u) - \Tr_{g_1} \big(\nabla_{g_1}(h_1*Z)\big) 
+ r^{-3 + 2\oldp} \Tr_{g_1}(g_2) 
\Big),
\\
-{1 \over 2}  \Div_{g_1} \big(\Lcal_Zg_1\big) 
& =  r^{-3 + 2\oldq} \Div_{g_1} (h_2) 
+ h_2\cdot\nabla_{g_1} \big( r^{-3 + 2\oldq}\big) - \Div_{g_1}(h_1*u).
\endaligned
$$
In combination with \eqref{eq1:hu}, we arrive at the following {\sl fourth order equation} for the dual Hamiltonian variable $u$: 
\begin{subequations}
\bel{system0}
\aligned
-\Delta_{g_1} \big(\Delta_{g_1} u\big) & = \Delta_{g_1} \Big((\Tr_{g_1}(g_1) - 1)^{-1} \big(u R_{g_1} - \Tr_{g_1}(h_1*h_1*u) - \Tr_{g_1} \big(\nabla_{g_1}(h_1*Z))\big)\Big)
\\
& \quad + {\Tr_{g_1}(g_2) \over r^{3 - 2\oldp}} \Delta_{g_1} \big((\Tr_{g_1}(g_1) - 1)^{-1} \big) 
+ 2g_1\Big(\nabla (\Tr_{g_1}(g_1) - 1)^{-1}, \nabla \big(r^{-3 +2\oldp} \Tr_{g_1}(g_2) \big) \Big)
\\
& \quad + 2(\Tr_{g_1}(g_1) - 1)^{-1}g_1\big(\nabla\big(\Tr_{g_1}(g_2)\big), \nabla r^{2p -3} \big) + (\Tr_{g_1}(g_1) - 1)^{-1} \Tr_{g_1}(g_2)\Delta_{g_1}r^{2p - 3}
\\
& \quad + (\Tr_{g_1}(g_1) - 1)^{-1} \big(f - \Div_{g_1} \big(\Div_{g_1}(g_2)\big) + g(g_2,\Ric) - h_1*h_1*g_2 + h_1*h_2 \big) r^{2p - 3}, 
\endaligned
\ee
together with the following {\sl second-order system of equations} for the dual momentum variable $Z$: 
\be
\aligned 
 -{1 \over 2}  \Div_{g_1} \big(\Lcal_Z g_1\big) 
& = r^{-3 + 2\oldq} \big(V - h_1* \nabla_{g_1}g_2 \Big)
 + h_2. \nabla_{g_1} \big(r^{-3 + 2\oldq} \big) - \Div_{g_1}(h_1*u).
\endaligned
\ee
\end{subequations}


In terms of the {\bf weighted unknowns\footnote{Observe that different exponents are used for the components $u$ and $Z$.}}
$\ut \coloneqq r^{-\oldp}u$ and $\Zt \coloneqq r^{-\oldq}Z$, 
this system may be rewritten as 
\begin{subequations}
\bel{system1}
\aligned
\Delta_{g_1}(\Delta_{g_1} \ut) + \sum_{0\leq | \beta| \leq 3}a_{\beta}^{(1)} \del^\beta \ut
+ A_p *\Big(\sum_{0\leq | \beta| \leq 3}c_{\beta}^{(1)} \del^\beta \Zt \Big) 
& = {\Big(\Tr_{g_1}(g_1) - 1\Big)^{-1}}r^{\oldp - 3}f,
\\
\Div_{g_1} \big(\Lcal_{\Zt} g_1\big) +  \sum_{0\leq | \beta| \leq 1}c_{\beta}^{(2)} \del^\beta \Zt 
+ B_p * \Big(\sum_{0\leq | \beta| \leq 3}c_{\beta}^{(1)} \del^\beta \ut\Big) 
& = - 2r^{\oldq -3}V,
\endaligned
\ee 
in which the coefficients satisfy the following bounds within the manifold $(\Mbf, g_1)$ 
\be
\aligned
&|a_\beta^{(1)} | \lesssim \eeG r^{| \beta|-4}, 
\qquad
&& |c_{\beta}^{(1)} | \lesssim \eeG r^{| \beta| -3},
\qquad |A_p| \lesssim \eeG r^{-\oldp}, 
\\
& |a_\beta^{(2)} | \lesssim \eeG r^{| \beta| - 3}, 
\qquad
&& |c_{\beta}^{(2)} | \lesssim \eeG r^{| \beta| -2}, 
\qquad |B_p| \lesssim \eeG r^{-\oldp}.
\endaligned
\ee
\end{subequations}
From~\cite{CarlottoSchoen} we recall the following ellipticity property of 
 Einstein's constraint equations
  in the sense of Douglis and Nirenberg. Within each asymptotic end of the manifold $(\Mbf, g_1)$ in the coordinate chart at infinity: provided the coefficients $\eps_*, \eeG$ are sufficiently small and by choosing the order parameters 
$s_1 = 0$, $s_2=s_3=s_4 = -1$, $t_1=4$, and $t_2=t_3=t_4 = 3$, 
the system \eqref{system1} in local coordinates takes the form \eqref{elliptic} and is elliptic in the sense of Douglis and Nirenberg.


\subsection{Lebesgue-H\"older regularity theory for the
linearized Einstein constraints}

\begin{subequations}
\label{eq:Ejdnb55}
Recall that
\bel{def: gh}
\aligned
g_2 = r^{3-2\oldp} d\Hcal_{(g_1,h_1)}^*[u,Z], 
\qquad
h_2 = r^{3-2\oldq} d\Mcal_{(g_1,h_1)}^*[u,Z],
\endaligned
\ee
satisfy
\bel{E-L-2}
d\Gcal_{(g_1,h_1)}[g_2,h_2] - (f, V) = 0.
\ee
\end{subequations}
We can now estimate $(g_2, h_2)$ in suitably weighted Lebesgue and H\"older norms.  
Recall once more that $\eps_*$ and $\eeG$ are assumed to be sufficiently small. We point out that the condition $(p, q) \leq (p_M, q_M)$ is actually not needed for the following two properties to hold. 
The proofs are postponed to Appendix~\ref{appendix-AAA}.

\begin{proposition}[Weighted Lebesgue regularity for the linearized Einstein operator]
\label{pro_integral}
Consider any tame seed data set $(g_1, h_1, \Hstar, \Mstar)$ on a background manifold $\Mbf = (\Mbf,e,r)$ with decay exponents $(\ppG, \qqG, \ppM, \qqM)$, and H\"older exponent $\alpha \in (0,1)$. 
Then for any effective exponents $(p,q)$, the solution $(g_2, h_2)$ to the linearized Einstein constraints
 \eqref{eq:Ejdnb55} satisfies 
$$
\aligned
\| g_2 \|_{L^2_{\oldp}(\Mbf)} 
& \lesssim \| f \|_{L^2_{\oldp + 2}(\Mbf)} + \eeG \, \| V \|_{L^2_{\oldq+1}(\Mbf)},
\qquad
\| h_2\|_{L^2_{\oldq}(\Mbf)} 
 \lesssim \eeG \, \| f \|_{L^2_{\oldp + 2}(\Mbf)} + \| V \|_{L^2_{\oldq+1}(\Mbf)}. 
\endaligned
$$
\end{proposition} 


\begin{proposition}[Weighted H\"older regularity for the linearized Einstein operator] 
\label{prop. estimate 4 order}
Consider any tame seed data set $(g_1, h_1, \Hstar, \Mstar)$ on a background manifold $\Mbf = (\Mbf,e,r)$ with decay exponents $(\ppG, \qqG, \ppM, \qqM)$, and H\"older exponent $\alpha \in (0,1)$. 
Then for any effective exponents $(p,q)$, the solution $(g_2, h_2)$ to the linearized equations \eqref{eq:Ejdnb55} satisfies 
$$
\aligned
\|g_2\|_{C^{2,\alpha}_{\oldp}(\Mbf)} 
& 
\lesssim  
\|f\|_{L^2C^{0,\alpha}_{\oldp+2}(\Mbf)} + \eeG \, \| V \|_{L^2C^{1,\alpha}_{\oldq+1}(\Mbf)}, 
\qquad
\|h_2\|_{C^{2,\alpha}_{\oldq}(\Mbf)}
\lesssim \eeG \, \| f \|_{L^2C^{0,\alpha}_{\oldp+2}(\Mbf)} + \| V \|_{L^2C^{1,\alpha}_{\oldq+1}(\Mbf)}.
\endaligned
$$
\end{proposition}


\subsection{A Lipschitz continuity property for  
 the Einstein constraints} 
\label{existence of solution}

The calculations in this section are parallel to those in~\cite{CarlottoSchoen} but in a different setup. Let us summarize our conclusion so far with a slightly different notation. Thanks to Theorem~\ref{prop-exislinea}, for each $(f_2,V_2) \in L^2C^{0,\alpha}_{\oldp + 2}(\Mbf) \times L^2C^{1,\alpha}_{\oldq+1}(\Mbf)$ there exists a unique minimizer 
of the adjoint Einstein functional $\Jbf_{(f_2,V_2)}$, that is, 
$(u_2, Z_2) \in H^2_{1 - \oldp}(\Mbf) \times H^1_{2 - \oldq}(\Mbf)$. 
By setting
$
g_2 \coloneqq r^{3-2\oldp} d\Hcal_{(g_1,h_1)}^*[u,Z], 
$
and $h_2 \coloneqq r^{3-2\oldq} d\Mcal_{(g_1,h_1)}^*[u,Z]$, 
it follows that 
$
d\Gcal_{(g_1,h_1)}[g_2,h_2] - (f_2, V_2) = 0. 
$ 

As established in the previous section, given any $(f_2, V_2) \in L^2C^{0,\alpha}_{\oldp + 2}(\Mbf) \times L^2C^{1,\alpha}_{\oldq+1}(\Mbf)$, from Propositions~\ref{pro_integral} and \ref{prop. estimate 4 order} we deduce  that $(g_2,h_2) \in {L^2C^{2,\alpha}_{\oldp}(\Mbf) \times L^2C^{2,\alpha}_{\oldq}}(\Mbf)$ and, in fact, 
$$
\|(g_2,h_2)\|_{{L^2C^{2,\alpha}_{\oldp}(\Mbf) \times L^2C^{2,\alpha}_{\oldq}}(\Mbf)} 
\lesssim \|(f_2,V_2)\|_{L^2C^{0,\alpha}_{\oldp + 2}(\Mbf) \times L^2C^{2,\alpha}_{\oldq+1}(\Mbf)}.
$$
Therefore, it allows us to define a bounded linear map which we find convenient to denote\footnote{Strictly speaking, this map also depends upon $(\Hstar, \Mstar)$.}
 by $(d\Gcal)^{-1}_{(g_1, h_1)}$ 
$$
\aligned 
(d\Gcal)^{-1}_{(g_1, h_1)}: 
L^2C^{0,\alpha}_{\oldp + 2}(\Mbf) \times L^2C^{1,\alpha}_{\oldq+1}(\Mbf)
\ni
(f_2,V_2) 
& \mapsto (d\Gcal)^{-1}_{(g_1, h_1)}(f_2,V_2)\coloneqq (g_2,h_2)
\in {L^2C^{2,\alpha}_{\oldp}(\Mbf) \times L^2C^{2,\alpha}_{\oldq}}(\Mbf).
\endaligned
$$
For the nonlinear problem to be considered now, the main unknown is $(g_2, h_2)$ and the argument takes place in the Lebesgue-H\"older spaces defined in Section~\ref{section--2}. 

With the notation 
$g = g_1 + g_2$ and $ h = h_1 + h_2$, 
we define the ``quadratic part'' of the Einstein operator to be 
\be
\aligned
& \Qcal_{(g_1,h_1)}[g_2, h_2] 
= \Big( \Qcal\Hcal_{(g_1,h_1)}[g_2, h_2],  \Qcal\Mcal_{(g_1,h_1)}[g_2, h_2] \Big) 
\\
& \coloneqq \Big(\Hcal(g,h) - \Hcal(g_1,h_1) - d\Hcal_{(g_1,h_1)}[g_2, h_2], \Mcal(g,h) - \Mcal(g_1,h_1)  - d\Mcal_{(g_1,h_1)}[g_2, h_2] \Big).  
\endaligned
\ee
The following proposition shows that the nonlinearities of Einstein's constraint equations
 can be controlled at the level of decay and regularity of interest. The proof follows from similar arguments as in~\cite{CarlottoSchoen} and, therefore, is omitted.

\begin{proposition}[Lipschitz continuity for the Einstein operator in weighted Lebesgue-H\"older spaces] 
\label{propositionfixedpoint} 
Consider any tame seed data set $(g_1, h_1, \Hstar, \Mstar)$ on a background manifold $\Mbf = (\Mbf,e,r)$ with decay exponents $(\ppG, \qqG, \ppM, \qqM)$, and H\"older exponent $\alpha \in (0,1)$. 
Then for any effective exponents $(p,q)$ and for any $\lambda>0$, there exists a sufficiently small real $r_0>0$ such that for all $f_2, V_2, f_3, V_3$ satisfying 
$
\| (f_2, V_2) \|_{L^2C^{0,\alpha}_{\oldp + 2}(\Mbf)  \times L^2C^{1,\alpha}_{\oldq+1}(\Mbf) } \leq  r_0,  
$
and 
$
\| (f_3, V_3) \|_{L^2C^{0,\alpha}_{\oldp + 2}(\Mbf) \times L^2C^{1,\alpha}_{\oldq+1}(\Mbf) }  \leq r_0,  
$
 with $(g_2,h_2) =  (d\Gcal)^{-1}_{(g_1, h_1)}(f_2, V_2)$ and $(g_3, h_3) =  (d\Gcal)^{-1}_{(g_1, h_1)}(f_3, V_3)$ one has
$$
\| \Qcal_{(g_1,h_1)}[g_2, h_2] - \Qcal_{(g_1,h_1)}[g_3, h_3] \|_{L^2C^{0,\alpha}_{\oldp + 2}(\Mbf)  \times L^2C^{1,\alpha}_{\oldq+1}(\Mbf)} 
\leq 
\lambda \, \|(g_2, h_2) - (g_3, h_3)\|_{L^2C^{2,\alpha}_{\oldp}(\Mbf)  \times L^2C^{2,\alpha}_{\oldq}(\Mbf)}.
$$
\end{proposition}


\subsection{Existence of the seed-to-solution map (Theorem~\ref{theo-prec})}

Given a seed data set $(g_1, h_1)$, the requirement that $(g,h) = (g_1 + g_2, h_1 + h_2)$ is a solution to 
Einstein's constraint equations 
 is equivalent to saying that 
\bel{Eeqe}
d\Gcal_{(g_1,h_1)}[g_2, h_2] + \Qcal_{(g_1,h_1)}[g_2, h_2] = (\Hstar, \Mstar) - \Gcal(g_1,h_1).
\ee 
In order to apply what we have established so far, we must ensure that 
$(g_2,h_2)$
  remains within the range of the mapping $(d\Gcal)^{-1}_{(g_1, h_1)}$, that is, 
$
(g_2, h_2) \in (d\Gcal)^{-1}_{(g_1, h_1)} \big( L^2C^{0,\alpha}_{\oldp + 2}(\Mbf) \times L^2C^{1,\alpha}_{\oldq+1}(\Mbf) \big).
$
In view of the Lipschitz continuity property enjoyed by the quadratic part which was derived in Proposition~\ref{propositionfixedpoint}, the left-hand-side of \eqref{Eeqe} 
 should also be guaranteed to
belong to $L^2C^{0,\alpha}_{\oldp + 2}(\Mbf) \times L^2C^{1,\alpha}_{\oldq+1}(\Mbf)$. This leads us to the condition 
$
(\Hstar, \Mstar) - \Gcal(g_1,h_1) \in L^2C^{0,\alpha}_{\oldp + 2}(\Mbf)  \times L^2C^{1,\alpha}_{\oldq+1}(\Mbf),
$
which, as stated in \eqref{eq--2-12}, requires us to assume the condition $(\oldp, \oldq) \leq (\ppM, \qqM)$ on the decay exponents for the matter.

\begin{theorem}[Existence of the seed-to-solution map for the Einstein constraints]  
\label{theo-main3}
Consider any tame seed data set $(g_1, h_1, \Hstar, \Mstar)$ on a background manifold $\Mbf = (\Mbf,e,r)$ with decay exponents $(\ppG, \qqG, \ppM, \qqM)$, and H\"older exponent $\alpha \in (0,1)$. 
Then for any effective exponents $(p,q)$, there exists a pair 
$
(g_2, h_2) \in
 (d\Gcal)^{-1}_{(g_1, h_1)} \big( L^2C^{0,\alpha}_{\oldp + 2}(\Mbf) \times L^2C^{1,\alpha}_{\oldq+1}(\Mbf) \big) 
$
such that 
\bel{eq:520}
d\Gcal_{(g_1,h_1)}[g_2, h_2] + \Qcal_{(g_1,h_1)}[g_2, h_2] 
= (\Hstar, \Mstar) - \Gcal(g_1,h_1).  
\ee
Moreover, the solution $(g,h) = (g_1+g_2, h_1+ h_2)$ satisfies \eqref{equa-213-2} and such a solution is unique.   
\end{theorem}

\begin{proof}
 1. Letting $g_3 = 0,\,h_3 = 0$ and $f_3 = 0,\,V_3= 0$, we inductively construct a sequence 
 $(f_i,V_i)$ and a sequence $(g_i,h_i)$ 
  as follows for all $i \geq 4$:
$$
\aligned
(f_i,V_i) & \coloneqq -\Qcal_{(g_1,h_1)}[g_{i-1}, h_{i-1}] + (\Hstar, \Mstar) - \Gcal(g_1,h_1),
\qquad 
(g_i,h_i) \coloneqq  (d\Gcal)^{-1}_{(g_1, h_1)} (f_i,V_i).
\endaligned
$$
Let $\lambda>0$ be sufficiently small and $r_0>0$ be as in Proposition~\ref{propositionfixedpoint}. We first prove by induction that 
\bel{induction:52}
\| (f_i, V_i) \|_{L^2C^{0,\alpha}_{\oldp + 2} \times L^2C^{1,\alpha}_{\oldq+1}} \leq r_0.
\ee 
In fact, assume that $\|(f_i,V_i)\|_{L^2C^{0,\alpha}_{\oldp + 2} \times L^2C^{1,\alpha}_{\oldq+1}}  \leq r_0$ for $i=3, \ldots,k$. Thanks to Proposition~\ref{propositionfixedpoint} and the fact that the operator $(d\Gcal)^{-1}_{(g_1, h_1)}$ is bounded (with respect to the specified spaces), we have 
$$
\aligned
& \|(f_{k + 1},V_{k + 1}) - (f_k,V_k)\|_{L^2C^{0,\alpha}_{\oldp + 2} \times L^2C^{1,\alpha}_{\oldq+1}}  
= \| \Qcal_{(g_1,h_1)}[g_k,h_k] - \Qcal_{(g_1,h_1)}[g_{k-1},h_{k-1}]\|_{L^2C^{0,\alpha}_{\oldp + 2} \times L^2C^{1,\alpha}_{\oldq+1}} 
\\
& \leq \lambda \|(g_k,h_k) - (g_{k-1},h_{k-1})\|_{L^2C^{2,\alpha}_{\oldp} \times L^2C^{2,\alpha}_{\oldq}}
\lesssim
\lambda\|(f_k,V_k) - (f_{k-1},V_{k-1})\|_{L^2C^{0,\alpha}_{\oldp + 2} \times L^2C^{1,\alpha}_{\oldq+1}}.
\endaligned
$$
Therefore, provided $\lambda$ is sufficiently small we have 
$$
\aligned
& \|(f_{k + 1},V_{k + 1}) - (f_k,V_k)\|_{L^2C^{0,\alpha}_{\oldp + 2} \times L^2C^{1,\alpha}_{\oldq+1}}  
 \leq 2^{-k} \, \|(f_4, V_4) - (f_3,V_3)\|_{L^2C^{0,\alpha}_{\oldp + 2} \times L^2C^{1,\alpha}_{\oldq+1}} 
 = 2^{-k} \, \|(\Hstar, \Mstar) - \Gcal(g_1,h_1)\|_{L^2C^{0,\alpha}_{\oldp + 2} \times L^2C^{1,\alpha}_{\oldq+1}},
\endaligned
$$
hence
$$
\aligned
& \|(f_{k + 1},V_{k + 1}) - (\Hstar, \Mstar) + \Gcal(g_1,h_1)\|_{L^2C^{0,\alpha}_{\oldp + 2} \times L^2C^{1,\alpha}_{\oldq+1}}  
\\
&
\leq \sum_{i=1}^{k}2^{-i} \|(\Hstar, \Mstar) - \Gcal(g_1,h_1)\|_{L^2C^{0,\alpha}_{\oldp + 2} \times L^2C^{1,\alpha}_{\oldq+1}} 
\leq 
2 \, \|(\Hstar, \Mstar) - \Gcal(g_1,h_1)\|_{L^2C^{0,\alpha}_{\oldp + 2} \times L^2C^{1,\alpha}_{\oldq+1}}.
\endaligned
$$
It follows that
$ 
\|(f_{k + 1},V_{k + 1})\|_{L^2C^{0,\alpha}_{\oldp + 2} \times L^2C^{1,\alpha}_{\oldq+1}}  
\leq 3 \|(\Hstar, \Mstar) - \Gcal(g_1,h_1)\|_{L^2C^{0,\alpha}_{\oldp + 2} \times L^2C^{1,\alpha}_{\oldq+1}}
  \leq 3 \eeM, 
$
and, therefore, provided that $ \eeM \leq r_0/3$, we obtain 
$
\| (f_{k + 1},V_{k + 1}) \|_{L^2C^{0,\alpha}_{\oldp + 2} \times L^2C^{1,\alpha}_{\oldq+1}} \leq r_0
$   
and \eqref{induction:52} holds, as claimed.

\vskip.16cm

2. Next, we observe that for all integers $m,n>0$
$$
\|(f_{m + n},V_{m + n}) - (f_{n},V_{n})\|_{L^2C^{0,\alpha}_{\oldp + 2} \times L^2C^{1,\alpha}_{\oldq+1}}  \leq 2^{-n} \|(f_{m},V_{m})\|_{L^2C^{0,\alpha}_{\oldp + 2} \times L^2C^{1,\alpha}_{\oldq+1}}  \leq 2^{-n}r_0. 
$$
This means that the sequence $(f_i,V_i)$, hence $(g_i,h_i)$, is Cauchy. As a consequence, $\{(g_i,h_i)\} $ converges in $L^2C^{2,\alpha}_{\oldp} \times L^2C^{2,\alpha}_{\oldq}$ to a limit which we denote by $(g_2, h_2)$ and satisfies the equation
$
d\Gcal_{(g_1,h_1)}[g_2, h_2] + \Qcal_{(g_1,h_1)}[g_2, h_2] =  (\Hstar, \Mstar) - \Gcal(g_1,h_1).
$
Moreover, each pair $(g_i,h_i)$ is of the form
$
d\Gcal^*_{(g_1,h_1)}[u_i,Z_i] = (r^{-3 + 2\oldp} g_i, r^{-3 + 2\oldq} h_i)
$
 for some minimizer $(u_i,Z_i) \in H^2_{1-\oldp} \times H^1_{2-\oldq}$.
We deduce from Proposition~\ref{prop.key} that the sequence $(u_i,Z_i)$ is also Cauchy, hence $(u_i,Z_i)$ converges in $H^2_{1-\oldp} \times H^1_{2-\oldq}$ to the minimizer $(u_2,Z_2)$ which satisfies 
$$
\aligned 
(r^{-3 + 2\oldp} g_2, r^{-3 + 2\oldq} h_2) 
& = d\Gcal^*_{(g_1,h_1)}[u_2,Z_2],
\qquad
(u_2,Z_2) 
 \in H^2_{1-\oldp}(\Mbf)  \times H^1_{2-\oldq}(\Mbf). 
\endaligned
$$
This means that $(g_2, h_2) =  (d\Gcal)^{-1}_{(g_1, h_1)}\big( d\Gcal_{(g_1,h_1)}[g_2, h_2] \big)$.

\vskip.16cm

3. Finally, let $(g_2, h_2)$ and $(g_2', h_2')$ be such solutions. It follows by \eqref{eq:520} that
\bel{unique}
d\Gcal_{(g_1,h_1)}[g_2 - g_2', h_2 - h_2'] + \Qcal_{(g_1,h_1)}[g_2, h_2] - \Qcal_{(g_1,h_1)}[g_2', h_2']
= 0.
\ee
On one hand, since $S$ is a bounded linear map we have
\bel{unique2}
 \|(g_2,h_2) - (g_2',h_2')\|_{L^2C^{2,\alpha}_{\oldp} \times L^2C^{2,\alpha}_{\oldq}}
\lesssim
\|d\Gcal_{(g_1,h_1)}[g_2 - g_2', h_2 - h_2'] \|_{L^2C^{0,\alpha}_{\oldp + 2} \times L^2C^{1,\alpha}_{\oldq+1}}.
\ee
On the other hand, the Einstein solutions, say $(g, h) = (g_1 + g_2, h_1 + h_2), (g',h') = (g_1 + h_2', h_1 + h_2')$, satisfy \eqref{equa-213-2}, and 
$d\Gcal_{(g_1,h_1)}[g_2, h_2]$ and $d\Gcal_{(g_1,h_1)}[g_2', h_2']$ are sufficiently small in $L^2C^{0,\alpha}_{\oldp + 2}(\Mbf) \times L^2C^{1,\alpha}_{\oldq+1}(\Mbf)$. Therefore, thanks to Proposition~\ref{propositionfixedpoint} we obtain 
$$
\|\Qcal_{(g_1,h_1)}[g_2, h_2] - \Qcal_{(g_1,h_1)}[g_2', h_2'] \|_{L^2C^{0,\alpha}_{\oldp + 2} \times L^2C^{1,\alpha}_{\oldq+1}}
\lesssim
\lambda \|(g_2,h_2) - (g_2',h_2')\|_{L^2C^{2,\alpha}_{\oldp} \times L^2C^{2,\alpha}_{\oldq}},
$$
for some sufficiently small $\lambda>0$. Combining with \eqref{unique}-\eqref{unique2}, we get
$$
\|(g_2,h_2) - (g_2',h_2')\|_{L^2C^{2,\alpha}_{\oldp} \times L^2C^{2,\alpha}_{\oldq}} \lesssim \lambda \|(g_2,h_2) - (g_2',h_2')\|_{L^2C^{2,\alpha}_{\oldp} \times L^2C^{2,\alpha}_{\oldq}},
$$
which implies $(g_2, h_2) = (g_2', h_2')$ as long as $\lambda$ is small. The proof is completed.
\end{proof}


\section{Asymptotic properties for the  
linearized Einstein constraints} 
\label{section--6}

\subsection{Weighted H\"older regularity}

The previous section has provided us with a large class of solutions to 
Einstein's constraint equations. Our next objective is to analyze their asymptotic properties and, in the present section, we begin with the linearized version of the Einstein constraints 
  and their dual version. We consider first the adjoint Hamiltonian equation linearized around the Euclidean metric $\delta$, that is,  $d\Hcal^*_{(\delta,0)} [u,0]$, 
which is essentially the Hessian operator.
We are interested in regularity and asymptotic decay properties at the sub/super-harmonic levels.
For the convenience of the reader, the most technical proofs in this section 
  are postponed to Appendix~\ref{append-sec4}.

The reader is referred to Bartnik~\cite{Bartnik} and Choquet-Bruhat and Christodoulou~\cite{ChoquetC} for standard material. For instance, recall the following continuous embedding property for weighted Sobolev and H\"older spaces: 
$C^{l,\alpha}_{\theta  + \eps}(\RR^3) \subset W^{l,m}_\theta (\RR^3)\subset C^{k,\alpha}_{\theta}(\RR^3)$, 
valid for all $\theta >0$, $m>1$, $l-k-\alpha > n/m$, and $\eps \in (0,1)$. 
In particular, a function $f \in W^{l, \oldq}_{\theta}(\RR^3)$ with $l > 3/q$ enjoys the pointwise decay $f = \Obig(r^{-\theta})$ when $r \to +\infty$. The following statement is standard.

\begin{proposition}[Weighted H\"older elliptic regularity. I]
\label{propo.classcialregular}
The Laplace operator $\Delta: C^{k + 2, \alpha}_\theta (\RR^3) \to C^{k, \alpha}_{\theta  + 2}(\RR^3)$ is an isomorphism for all 
$\theta \in (0,1)$ and $\alpha \in (0,1)$ and $k \geq 0$. 
\end{proposition}

More generally, using for instance Douglis--Nirenberg's regularity theory~\cite{DouglisNirenberg} we have the following decay property. 

\begin{proposition}[Weighted H\"older elliptic regularity. II]
\label{propo.laplace - holder}
Let $k \geq 0$, $\theta >0$, and $\alpha \in (0,1)$. 
If $w\in C^0_\theta (\RR^3)$ and $ - \Delta w + fw \in C^{k,\alpha}_{\theta  + 2} (\RR^3)$ for some $f\in C^{k,\alpha}_2 (\RR^3)$, then $w \in C^{k + 2,\alpha}_\theta (\RR^3)$ and
$
\| w\|_{C^{k + 2, \alpha}_\theta(\RR^3)} 
\lesssim 
\| w\|_{C^{0}_\theta(\RR^3)} + \| \Delta w  -  f w \|_{C^{k,\alpha}_{\theta + 2}(\RR^3)}. 
$ 
\end{proposition}

\begin{proof} Applying Theorem~\ref{propo.ellipticity} with $t = k + 2$ and $s = - k$ and $d(x) = r(x)/3$, we find  
$$
\aligned
& \sum_{0 \leq i \leq k + 2}d(x)^i|\del^iw(x)| + d(x)^{k + 2 + \alpha}[\del^{k + 2} w]_{\alpha,B_{d(x)/2}(x)}
\\
& \lesssim
     \sup_{B_{3d(x)/4}(x)} |w(x)| + \sum_{0 \leq i \leq k}\sup_{B_{3d(x)/4}(x)}d(x)^{2 + i} \big|\del^i\big(\Delta w  -  fw\big)\big| + d(x)^{k + \alpha}[(\Delta w  -  fw)]_{\alpha,B_{3d(x)/4}(x)}, 
\endaligned
$$
and the desired conclusion follows from the definitions.
\end{proof}

Recall that, in view of our notation in Remark \ref{rem22}, the function $r$ in $\RR^3$ does coincide with the standard Euclidean distance from the origin ---except in the unit ball centered at the origin--- while $r$ is bounded below by a positive constant (actually, $r \geq e^{-1/2}$). 

\begin{proposition}[Weighted H\"older elliptic regularity. III]
\label{laplaceplus}

1. Given $k \geq 0$, $\alpha \in (0,1)$, and $\theta \in (0,2)$, and for any function $w$  
one has 
\bel{laplace +}
\| w \| _{C^{k + 2, \alpha}_\theta(\RR^3)} \lesssim \|\Delta w - 2 \, r^{-2} w\|_{C^{k,\alpha}_{\theta + 2}(\RR^3)}, 
\ee
 (which is relevant only when the right-hand side is finite). 
   In particular, $\|r^{-2} w\|_{C^{k + 2,\alpha}_{\theta + 2}(\RR^3)}$ is then finite. 

2. Given $\alpha \in (0,1)$ and $\theta \in [1, +\infty)$, for any function $w$ 
one has
\bel{laplace + 2}
\| r^{-2} w \|_{L^1(\RR^3)} 
\lesssim \|\Delta w - 2 \, r^{-2} w\|_{C^{0,\alpha}_{\theta + 2}(\RR^3)} + \|\Delta w - 2 \, r^{-2} w\|_{L^1(\RR^3)},
\ee
provided the right-hand side is finite. 
\end{proposition}

\begin{proof} 1. It is sufficient to check that 
$\| w \|_{C^0_\theta(\RR^3)} \lesssim \|\Delta w - 2 r^{-2} w\|_{C^{0}_{\theta + 2}(\RR^3)}$. 
From the definitions, we have
\bel{eq:6.3:1a}
- r^{-\theta - 2} \|\Delta w - 2 r^{-2} w\|_{C^{0}_{\theta + 2}(\RR^3)}  
\leq  - \Delta w + {2 \over r^2} w 
\leq 
r^{-\theta - 2} \|\Delta w - 2 r^{-2} w\|_{C^{0}_{\theta + 2}(\RR^3)}.
\ee
Next, we will control $w(1)$ as follows. Thanks to Kato inequality $\Delta |w| \geq \sgn(w) \Delta w$, we find 
$$
\aligned
- \Delta |w| 
 \leq \sgn(w) ( - \Delta w + 2 r^{-2} w ) - 2 r^{-2} |w| 
  \leq | \Delta w - 2 r^{-2} w |. 
\endaligned
$$ 
Let $v$ be the (unique) solution  (chosen to tend to zero at infinity)
to the Poisson problem $- \Delta v = E$ 
with a source term $E = | \Delta w - 2 r^{-2} w |$ (see \eqref{poisson-equation} below). By the maximum principle $|w| \leq v$ and, on the other hand, it follows from Proposition~\ref{propo.classcialregular} 
that 
$
\| v \|_{C^{2, \alpha}_{\min(\theta, 1 - \eps)} (\RR^3)}  
\lesssim \| \Delta w - 2 r^{-2} w \|_{C^{0, \alpha}_{\theta +2} (\RR^3)}, 
$
  in which $\eps \in (0,1)$ is introduced to guarantee that $\min(\theta, 1 - \eps)<1$.
In particular, we obtain 
$
|w(1)| \leq v(1) \lesssim  \| \Delta w - 2 r^{-2} w \|_{C^{0, \alpha}_{\theta +2} (\RR^3)}.
$
Combining this with \eqref{eq:6.3:1a}, we see that there exists a constant $K>0$ independent of $w$ such that for all $r \ge 1$
$$
\aligned
& - \Delta \Big(
w  - K r^{-\theta} \|\Delta w - 2 r^{-2} w\|_{C^{0}_{\theta + 2}(\RR^3)} \Big) 
+ {2 \over r^2} \, \Big( 
w  - K r^{-\theta}  \|\Delta w - 2 r^{-2} w\|_{C^{0}_{\theta + 2}(\RR^3)} \Big) 
\leq 0,
\\
&
 - \Delta \Big(w + K  r^{-\theta} \|\Delta w - 2 r^{-2} w\|_{C^{0}_{\theta + 2}(\RR^3)}  \Big) 
+ {2 \over r^2} \Big(w + K  r^{-\theta}  \|\Delta w - 2 r^{-2} w\|_{C^{0}_{\theta + 2}(\RR^3)} \Big) 
\geq 0, 
\endaligned
$$
and 
$
|w(1)| \leq K \, \|\Delta w - 2 r^{-2} w\|_{C^{0}_{\theta + 2}(\RR^3)}.
$
By the maximum principle we deduce that
$
|w | \lesssim r^{-\theta} \, \|\Delta w - 2 r^{-2} w\|_{C^{0}_{\theta + 2}(\RR^3)}
$
and, thanks to Proposition~\ref{propo.laplace - holder}, the inequality \eqref{laplace +} follows. 

\vskip.16cm

2. In order to derive \eqref{laplace + 2}, we use again Kato inequality  
$\Delta |w| - \sgn(w) ( \Delta w - 2 r^{-2} w) \geq 2 r^{-2} |w|$, which we integrate on an arbitrary ball $B_R$: 
$$
\aligned
\int_{B_R} r^{-2}|w| \, dx & \lesssim \int_{B_R} \Delta |w| \, dy + \int_{B_R} |\Delta w - 2 \, r^{-2} w| \, dy
 \lesssim \int_{\del B_R} \Big| \del_i (|w|) {x_i \over R} \Big| \, d\omega + \int_{B_R} |\Delta w - 2 \, r^{-2} w| \, dy.
\endaligned
$$
In view of \eqref{laplace +},  by writing 
$
\| r^{-2} w \|_{L^1(\RR^3)} \lesssim \|\Delta w - 2 \, r^{-2} w\|_{C^{0,\alpha}_{\theta + 2}(\RR^3)} + \|\Delta w - 2 \, r^{-2} w\|_{L^1(\RR^3)}$, 
the proof is completed. 
\end{proof}


\subsection{Harmonic and super-harmonic decay} 

Given a source function $E \in L^1(\RR^3)\cap L^p(\RR^3)$ with $p>3/2$, we consider the (continuous) solution $w: \RR^3 \to \RR$ to the Poisson problem 
\bel{poisson-equation}
\aligned
& -  \Delta w = E \quad \text{in } \RR^3,
\qquad
 \lim_{|x| \to + \infty} w(x) =  0.
\endaligned
\ee  
The proofs of the following statements are postponed to Appendix~\ref{append-sec4}. 
The second proposition considers solutions with faster decay and therein it is convenient to specify the value of $\sgn(0)=0$ of the standard sign function.

\begin{proposition}[Solutions with harmonic decay]
\label{proposition behavior at infinity} 
Given $\alpha\in (0,1)$ and any function $E \in L^1(\RR^3)\cap C^{0,\alpha}_3(\RR^3)$, 
the solution $w$ to the problem \eqref{poisson-equation} belongs to $C^{2,\alpha}_1(\RR^3)$ 
and, moreover, satisfies 
\bel{decayPoisson}
\aligned
\lim_{|x| \to + \infty} |x| \, w(x) 
& = {1 \over 4 \pi} \int_{\RR^3} E \, dy =: \cbf(E), 
\qquad 
\lim_{|x| \to + \infty} |x|^2| \nabla w(x)| = |\cbf(E)|.
\endaligned
\ee 
\end{proposition}


\begin{proposition}[Solutions with super-harmonic decay]
\label{proposition behavior at infinity k} 
Given an integer $k \geq 0$, a H\"older exponent $\alpha \in (0,1)$, and a decay exponent $\theta \in [1,2)$, as well as a source function 
$E \in L^1(\RR^3)\cap C^{k,\alpha}_{\theta + 2}(\RR^3)$ with vanishing integral 
$\int_{\RR^3} E \, dy = 0$,  
the solution $w : \RR^3 \to \RR$ to the Poisson problem \eqref{poisson-equation} decays at a super-harmonic rate, in the sense that 
\bel{eq:EEhat}
\| w \|_{C^{k + 2,\alpha}_\theta(\RR^3)} 
\lesssim \| E \|_{C^{k,\alpha}_{\theta + 2}(\RR^3)}.
\ee 
Furthermore, provided 
\begin{subequations}
\bel{eq:65cond}
\lim_{r = |x| \to +\infty} r^{\theta + 2 + i} \del^i E = 0, \qquad i = 0, 1, \ldots, k, 
\ee 
one also has
\bel{eq:65decay}
\lim_{r=|x| \to +\infty} r^{\theta + i} \del^i w = 0, 
\qquad i=0, 1, \ldots, k+2. 
\ee
\end{subequations}
\end{proposition}


Instead of the standard distance $|x|$ from the origin, we now make use of our definition of the function $r=r(x)$  for the Euclidean case (cf.~Remark~\ref{rem22}), and we observe that $r$ and $|x|$ coincide in the exterior of the unit ball. 
So the term $\Delta (1/r)$ (arising in the definition of $\cbf_1(E)$ below) vanishes identically in the exterior of this ball. The following statement is relevant provided $E$ has sufficiently fast decay at infinity, so that the norm in the right-hand side of \eqref{eq.66:15c} (below) is finite. 

\begin{corollary}
\label{remarkpoisson}
\begin{subequations}
Given an integer $k \geq 0$, $\theta \geq 1$, and $\alpha \in (0,1)$, let $w$ be the solution to the Poisson problem \eqref{poisson-equation} with $E \in L^1(\RR^3) \cap C^{k,\alpha}_{\theta + 2}(\RR^3)$, and assume that $(1 - \sgn(2 - \theta)) \, r \, E \in L^1(\RR^3)$. With
$\cbf(E) \coloneqq (1/(4\pi)) \int_{\RR^3} E \, dx$, 
one has   
\be
\aligned
E_1 \coloneqq - \Delta \Big( w - \cbf(E)/r \Big)  
& \in  L^1(\RR^3) \cap C^{k,\alpha}_{\theta + 2}(\RR^3),
\qquad 
\int_{\RR^3} E_1 \, dx= 0, 
\\
(1 - \sgn(2 - \theta)) \, r \, E_1 
& \in L^1(\RR^3). 
\endaligned
\ee
Consequently, from Proposition~\ref{proposition behavior at infinity k} with $p = \min(\theta, 2)$ it follows that 
\bel{eq.66:15c}
\| w - \cbf(E)/r \|_{C^{k + 2, \alpha}_p(\RR^3)} 
\lesssim
\| E_1 \|_{C^{k,\alpha}_{p+2}(\RR^3)}  
+ \|(1 - \sgn(2 - \theta)) \, r \, E_1\|_{L^1(\RR^3)}.
\ee
\end{subequations}
\end{corollary}

\subsection{The adjoint momentum operator}

We now study the divergence of the adjoint momentum equation linearized around the Euclidean metric, that is, we focus our attention on the following operator defined over all vector fields in $\RR^3$:  
\bel{equa-411}
\LMcal(W) \coloneqq \Div_\delta d\Mcal^*_{(\delta,0)} [0,W].
\ee  
Given any vector field $E = (E^i) : \RR^3 \to \RR^3$, we thus consider the following linear second-order elliptic problem on $\RR^3$: 
\bel{AM-equations}
\aligned
&(\LMcal (W))^i  = - \Delta W^i - \del_i(\del_j W^j) = E^i, \qquad i = 1,2,3,
\\
& \lim_{|x| \to +\infty} W(x) = 0.
\endaligned
\ee 
Before we can investigate the asymptotic behavior of solutions to \eqref{AM-equations}, let us first compute the fundamental solution associated with our operator.

\begin{lemma}  
\label{lemma-477} 
The matrix-valued field $\Mbb(x,y)  = (\Mbb_{ij}(x,y))_{1\leq i,j \leq 3}$ defined in $\big\{(x,y) \in \RR^3 \times \RR^3, \,  x \ne y \big \}$ by 
$
\Mbb_{ij}(x, y) \coloneqq {1 \over 16 \pi  \, |x -y|} \Big(  {(x_i - y_i) (x_j - y_j) \over |x -y|^2} + 3 \delta_{ij} \Big)
$
satisfies 
\bel{Elementmatrixsense}
- \Delta \Mbb_{ij}(\cdot, y) - \sum_{k=1,2,3} \del_{ik} \Mbb_{jk}(\cdot,y) 
= \delta_y \, \delta_{ij}, 
 \qquad y \in \RR^3
\ee
in the sense of distributions on $\RR^3$, 
where $\delta_y$ is the Dirac measure at the point $y \in \RR^3$ (also written below as $\delta(x - y)$) and $\delta_{ij}$ denotes the Kronecker symbol.
\end{lemma}

\begin{proposition}
\label{theorem.solutionAM}
Consider any $E=(E^1, E^2, E^3) \in L^1_\loc(\RR^3)$ such that the function $y \mapsto \Mbb_{ij}(x,y) \, E^j(y)$ a
  is integrable 
   (for $i,j = 1,2,3$ and almost every $x$). Then the vector field $W = (W^1, W^2, W^3)$ defined by
\begin{subequations}
\bel{eq:solM0}
W^i (x) \coloneqq \int_{\RR^3} \Mbb_{ij}(x,y) E^j(y) \, dy, 
\quad x \in \RR^3 
\ee
belongs to $L^1_\loc(\RR^3)$ and solves the linearized momentum equation in the sense of distributions, i.e. 
\bel{eq:solM}
\LMcal (W) = E. 
\ee
\end{subequations}
\end{proposition}

\begin{proof} We have 
$ 
| W^i(x) | \leq \int_{\RR^3} | \Mbb_{ij}(x,y) |  \,  | E^j (y) | \, dy
\leq {1 \over 4 \pi}\int_{\RR^3} {1 \over |x - y|}  | E^j (y) | \, dy$, 
and for any ball $B_R(x_0) \subset \RR^3$, using Fubini's theorem we find 
$$
\int_{B_R(x_0)} | W^i(y) | \, dy \leq {1 \over 4 \pi} \int_{\RR^3} \Big( \int_{B_R(x_0)} {1 \over |x - y|} \, dx \Big) | E^j(y) | \, dy.
$$
Since $\int_{B_R(x_0)} {1 \over |x - y|} \, dx$ is bounded and, moreover,
$
\int_{B_R(x_0)} {1 \over |x - y|} \, dx = {4 \pi R^3\over 3 |x_0 - y|}$ 
for $|x_0 -y| \geq R$  
  by the mean value property,
it follows that $\int_{B_R(x_0)} | W^i | \, dy < +\infty$, hence $W \in L^1_\loc(\RR^3)$ as claimed. 
Next, we show that
$$
- \int_{\RR^3} \big( W^i \Delta \phi + \sum_{1 \leq j \leq 3} W^j \del_{ij} \phi \big)  \, dx = \int_{\RR^3} E^i \phi \, dx
$$
 for any test-function $\phi: \RR^3 \to \RR$. 
Namely, by plugging \eqref{eq:solM0} in the left-hand side and applying Fubini's theorem, the above expression becomes from \eqref{Elementmatrixsense} 
\begin{equation*}
- \int_{\RR^3} 
\Bigg( \int_{\RR^3} \big( \Mbb_{ki}(x,y) \Delta \phi + \sum_{j=1,2,3} \Mbb_{jk}(x,y) \del_{ij}\phi \big) \, dy \Bigg) \, E^k(x) \, dx 
= \int_{\RR^3} E^i \phi \, dx. 
\qedhere
\end{equation*}
\end{proof}

We are now in a position to rely on a result established in Lockhart and McOwen \cite[Theorem 3]{Lockhart}. 

\begin{proposition}[Isomorphism property for the adjoint momentum equation. I]
\label{theorem elliptic system} 
Given an integer $k \geq 0$, a H\"older exponent $\alpha\in (0,1)$ and a decay exponent $\theta>0$, the operator $\LMcal: H^{k + 2}_\theta (\RR^3) \to H^k_{\theta + 2}(\RR^3)$ is continuous and, moreover if $\theta \in (0,1)$, is 
  an isomorphism. 
\end{proposition} 

Since the operator $\LMcal$ is elliptic in the sense of Douglis and Nirenberg (with $t_1 = t_2 = t_3 = k + 2$ and $s_1 = s_2 = s_3 = - k$), from Theorem~\ref{propo.ellipticity} we deduce also the following result. 

\begin{proposition}[Weighted regularity property for the adjoint momentum equation]
\label{proposition.LMregularity}
For any integer $k\ge 0$, H\"older exponent $\alpha \in (0,1)$, and decay exponent $\theta>0$, consider the solution $W$ to the problem \eqref{AM-equations} associated with a source-term $E \in C^{k, \alpha}_{\theta + 2}(\RR^3)$. Then, provided $W \in C^{0}_{\theta}(\RR^3)$ one has $W \in C^{k + 2, \alpha}_{\theta}(\RR^3)$ and, moreover,
$
\| W \|_{C^{k + 2, \alpha}_{\theta}(\RR^3)}
\lesssim
\| W \|_{C^0_{\theta}(\RR^3)} + \| E \|_{C^{k, \alpha}_{\theta + 2}(\RR^3)}. 
$
\end{proposition}


\begin{proposition}[Isomorphism property for the adjoint momentum equation. II]
\label{theorem elliptic system ii} 
Given an integer $k\ge 0$, a H\"older exponent $\alpha \in (0,1)$, and a decay exponent $\theta>0$, the operator $\LMcal : C^{k + 2, \alpha}_\theta(\RR^3) \to C^{k, \alpha}_{\theta + 2}(\RR^3)$ is continuous and, moreover if $\theta \in (0,1)$, is  an isomorphism.
\end{proposition}

\begin{proof}
\begin{subequations} 
For any 
$\eps >0$ we have 
$C^{k, \alpha}_{\theta}(\RR^3) \subset H^{k}_{\theta - \eps}(\RR^3)$, hence the injectivity of the operator follows from Proposition~\ref{theorem elliptic system}. Next, given $E \in C^{k, \alpha}_{\theta + 2}(\RR^3)$, we may consider a sequence of smooth and compactly supported functions $\{E_i\}$ converging to $E$ in $H^k_{\theta + 2 - \eps}(\RR^3)$. It follows from Proposition~\ref{theorem.solutionAM} that 
$W_i^j \coloneqq \int_{\RR^3} \Mbb_{jk}(x,y) E_i^k(y) \, dy$
satisfies $\LMcal W_i = E_i$. Letting $i \to +\infty$, thanks to Proposition~\ref{theorem elliptic system} we obtain $W_i \to W$ in $H^{k+2}_{\theta-\eps}(\RR^3)$. Moreover, we have $\LMcal W = E$ and
$
W^j(x) = \int_{\RR^3} \Mbb_{jk}(x,y) E^k(y) \, dy.
$
Since $|\Mbb (x,y)| \leq {1 \over 4 \pi |x - y|}$,  this identity tells us that
$
|W^j(x)| \leq {1  \over 4 \pi } \int_{\RR^3} {|E^k(y)| \over |x - y| }  \, dy, 
$
hence, by the isomorphism property enjoyed by the Poisson operator we find 
$\| W^j \|_{C^2_{\theta}(\RR^3)} \lesssim \| E \|_{C^{0, \alpha}_{\theta + 2}(\RR^3)}$.
Combining this with Proposition~\ref{proposition.LMregularity}, we obtain
$\| W \|_{C^{k + 2, \alpha}_{\theta}(\RR^3)} \lesssim \| E \|_{C^{k, \alpha}_{\theta + 2}(\RR^3)}$ and the proof is completed.
\end{subequations}
\end{proof}


\subsection{Harmonic and super-harmonic decay} 

We have arrived at the desired asymptotic statements concerning Einstein's momentum operator.

\begin{proposition}[Adjoint momentum equation: solutions with harmonic decay]
\label{proposition behavior LMcal} 
Given any H\"older exponent $\alpha\in (0,1)$, and a source $E \in L^1(\RR^3)\cap C^{0,\alpha}_3(\RR^3)$, 
the solution $W: \RR^3 \to \RR^3$
   to the problem \eqref{AM-equations} has 
 H\"older continuous second-order derivatives, namely $W \in C^{2,\alpha}_1(\RR^3)$ and, moreover, for every $\omega \in S^2$  
\bel{decayLMcal}
\aligned
\lim_{x \to +\infty} 
 |x| \, W^i(x)  
& =  \widetilde \Mbb_{ij}(\omega) \int_{\RR^3} E^j(y) \, dy,
\qquad
\quad 
\widetilde \Mbb_{ij} (\omega) 
 \coloneqq 
\lim_{x \to +\infty} 
|x| \, \Mbb_{ij}(x, 0)
= 
{1 \over 16 \pi} \big(  \omega_i \omega_j + 3 \delta_{ij} \big).
\endaligned
\ee
\end{proposition}

\begin{proof} In the proof of Proposition~\ref{theorem elliptic system ii}, we have established that  
$W^i = \int_{\RR^3} \Mbb_{ik}(x,y) E^k(y) \, dy$ for $i = 1, 2, 3$. 
Given any $\eps \in (0,1)$ and since $E\in L^1(\RR^3)$, we can choose a radius $R_\eps$ so large that 
$\sum_{i=1}^3 \int_{\RR^3 \setminus B_{R_\eps}(0)} |E^{i}| \, dy \leq \eps$.  
We obtain
$$
\aligned
|x| \, W^i(x) 
& = \int_{B_{R_\eps}(0)} |x| \, \Mbb_{ij}(x,0) E^j(y) \, dy + \int_{B_{R_\eps}(0)} |x| \big( \Mbb_{ij}(x,y) - \Mbb_{ij}(x,0) \big) E^j(y) \, dy
\\
& \quad + \int_{\RR^3 \setminus B_{R_\eps}(0)} |x| \, \Mbb_{ij}(x,y) E^j(y) \, dy
 =:  I_1^\eps + I_2^\eps + I_3^\eps. 
\endaligned
$$
For all sufficiently large $|x|$ we have 
$\sup_{y \in B_{R_\eps}(0)} |x| \, | \Mbb_{ij}(x,y) - \Mbb_{ij}(x,0) | \lesssim \eps$, 
and, therefore, $I_2^\eps \lesssim \eps$. Next, an analysis similar to the one for \eqref{esI2} in the proof of Proposition~\ref{proposition behavior at infinity} shows 
that $I_3^\eps \lesssim \sqrt{\eps}$. Taking these results into account, for all sufficiently large $|x|$ we have
$
 |x| \, \Big| W^i(x) 
- \int_{B_{R_\eps}(0)} \Mbb_{ij}(x,0) E^j(y) \, dy \Big| 
\lesssim \eps. 
$
Combining this with our condition $\sum_{i=1}^3 \int_{\RR^3 \setminus B_{R_\eps}(0)} |E^{i}| \, dy \leq \eps$, we arrive at \eqref{decayLMcal}. 
\end{proof}


\bse
\label{sub-300}
Now, we introduce our {\bf asymptotic model for the momentum equation} as follows ($i, j = 1, 2, 3$): 
\be
\aligned
W^j_{i} 
& \coloneqq - {1 \over r^3} ( x_i x_j + 3 r^2 \delta_{ij} ), 
\qquad
 W_i \coloneqq (W^1_i, W^2_i, W^3_i),  
\qquad 
E^j_i 
& \coloneqq (\LMcal W_i)^j. 
\endaligned
\ee
By construction we have $E^j_i \equiv 0$ for $r \geq 1$
and, furthermore,  
\be
\aligned
\int_{\RR^3} E^i_i \, dx & = - 16 \pi \quad (i = 1, 2,3), 
\qquad \quad 
\int_{\RR^3} E^i_j \, dx = 0   \quad (i \ne j).
\endaligned
\ee  
Therefore, for any $E = (E^1, E^2, E^3) \in L^1$ and by setting 
\bel{specialmodel} 
W_\infty(E) 
 \coloneqq - {1 \over 16 \pi} \sum_{i=1}^3 \Big(\int_{\RR^3} E^i \, dx \Big) W_i,
\ee 
\ese
we arrive at
\be
\int_{\RR^3} (\LMcal W_\infty (E))^i \, dx = \int_{\RR^3} E^i \, dx, 
\qquad i = 1, 2, 3. 
\ee 

\begin{proposition}[Adjoint momentum equation: solution with super-harmonic behavior]
\label{proposition behavior LMcal 2} 
Given any exponents $k \geq 0$, $\alpha \in (0,1)$, and $\theta \in [1,2)$, as well as a source term $E \in L^1(\RR^3)\cap C^{k,\alpha}_{\theta + 2}(\RR^3)$,
the solution  $W: \RR^3 \to \RR^3$
  to the problem \eqref{AM-equations} satisfies
$$
\aligned
\| W - W_\infty (E)\|_{C^{k + 2,\alpha}_\theta(\RR^3)} 
\lesssim \| E - \LMcal W_\infty(E) \|_{C^{k,\alpha}_{\theta + 2}(\RR^3)}, 
\endaligned 
$$
where the asymptotic model $W_\infty = W_\infty(E)$ is defined by \eqref{sub-300}. Moreover, if $\theta \geq 2$ and $rE \in L^1(\RR^3)$, one has 
$$
\| W - W_\infty (E)\|_{C^{k + 2,\alpha}_2(\RR^3)} 
\lesssim 
\| E - \LMcal W_\infty(E) \|_{C^{k,\alpha}_4(\RR^3)} + \| r(E - \LMcal W_\infty(E))\|_{L^1(\RR^3)}. 
$$
\end{proposition}

\begin{proof}
By a straightforward calculation, we have 
$
\LMcal (W - W_\infty(E))^i   
 \in L^1(\RR^3) \cap C^{k,\alpha}_\theta(\RR^3)
$
and
$
\int_{\RR^3} \LMcal (W - W_\infty(E))^i \, dx  = 0
$
(for $i = 1,2,3$). 
The rest of the proof is similar in spirit to the ones of Proposition~\ref{proposition behavior at infinity k} and Corollary~\ref{remarkpoisson} and is omitted. 
\end{proof}


\subsection{Reduction to the Euclidean background}
\label{remarkmanifold}
 
For future reference, we close this section with an immediate observation. 
Let $B_R$ be a large ball in the background manifold $(\Mbf, e)$ such that $\Mbf \setminus B_R$ is isomorphic to the exterior 
of a ball\footnote{The same notation will be used for the balls in the two manifolds.}
 in the Euclidean space $\RR^3$. Consider a solution $W$ to the equation (or system of equations) 
$L W = E$ in $\Mbf \setminus B_R$, 
where the operator $L$ denotes either  $- \Delta$ or $\LMcal$. Let $\phi$ be a cut-off function in $\RR^3$ such that $\phi$ is identically $1$ in $\RR^3 \setminus B_{R+2}$ and vanishes identically in $B_{R + 1}$. Setting 
$\Wt \coloneqq \phi W$ it follows that
$
L \Wt  = \phi E + \del \phi * \del W + \del^2 \phi * W =: E_1. 
$
Then, an estimate {\sl in the Euclidean space $\RR^3$}, say  
$
\| \Wt \|_{C^{k + 2,\alpha}_\theta (\RR^3)} \leq \| E_1  \|_{C^{k, \alpha}_\eta (\RR^3)}
$
(for some integer $k \geq 0$, H\"older exponent $\alpha \in (0,1)$, and decay exponents $\theta, \eta>0$)
implies an estimate {\sl in the manifold itself,} namely  
\be
\| W \|_{C^{k + 2}_\theta (\Mbf \setminus B_{R + 2})} 
\lesssim  \|E\|_{{C^{k, \alpha}_\eta (\Mbf \setminus B_{R + 1})}}
+ 
\| \del W \|_{{C^{k, \alpha}_\eta (B_{R + 2} \setminus B_{R + 1})}} + \| W \|_{{C^{k, \alpha}_\eta (B_{R + 2} \setminus B_{R + 1})}}.
\ee 


\section{Asymptotic properties of the seed-to-solution map} 
\label{section--7}
 
\subsection{Objective for this section}

\paragraph{Full set of assumptions for reaching the conclusions of this section.}

Building upon the material developed in the previous sections, we are now in a position to give a proof of Theorems~\ref{theo:deuxieme} and \ref{theo:deuxieme-2}. Our arguments rely on a detailed analysis of the nonlinear coupling taking place between the Hamiltonian and momentum constraints.  We are going to apply our existence theory for the seed-to-solution map with the ``weakest'' pair of exponents $(1/2, 3/2)$ and then successively improve these basic rates of decay by 
expressing the asymptotic properties in Section~\ref{section--6} for the (linearized and dual versions of) Hamiltonian and momentum constraint equations. 
Before we begin, recall that a H\"older exponent $\alpha \in (0,1)$ is fixed throughout the proof, and we work with 
admissible exponents
\be
1/2 \leq \ppG \leq 2 ( \qqG -1),
\qquad \quad
1/2 \leq \ppM \leq 2 ( \qqM -1),
\ee
and an admissible pair $(p,q)$ satisfying the effective conditions \eqref{eq--2-12}, that is, 
\bel{eq--2-12-repeat}
\aligned
& (p, q) < (1,2), 
\qquad \quad
 (p, q) \leq (\ppM, \qqM), 
\qquad \quad
 | q - p -  1 | \leq \qqG -1.
\endaligned
\ee
For Theorem~\ref{theo:deuxieme}, the strongly tame conditions \eqref{eq:strongtameH}, i.e.  
\bel{eq:strongtameH-repeat}
\ppM \geq 1, 
\qquad  \qqM  > \max(3 - \qqG, 3/2), 
\qquad  
\Hcal(g_1,h_1) - \Hstar \in L^1(\Mbf) 
\ee  
are assumed, while for Theorem~\ref{theo:deuxieme-2} the additional conditions \eqref{eq:strongtameM} are also assumed, i.e. 
\bel{eq:strongtameM-repeat}
\qqM \geq 2, 
\qquad \hskip4.2cm 
\Mcal(g_1,h_1) - \Mstar \in L^1(\Mbf). 
\ee   

\paragraph{Dealing with a large compact region.}

In addition, the problem is reduced to a single asymptotic end, as follows. 
Let $B_R \subset \Mbf$ be a sufficiently large geodesic ball so that $\Mbf \setminus B_R$ is diffeomorphic to the union of finitely many exterior domains of the form $\RR^3 \setminus B_R$. Within the ball with radius $2R$, thanks to Theorem \ref{theo-prec} with $(p, q)$ satisfying the 
conditions \eqref{eq--2-12} (that is, \eqref{eq--2-12-repeat}), we control the solution as follows:  
$$
\aligned
& \| g - g_1 \|_{L^2C_{\oldp}^{2, \alpha}(B_{2R})} 
 \lesssim 
 \| \Hcal( g_1, h_1) - \Hstar\|_{L^2C_{\oldp + 2}^{0, \alpha}(\Mbf)}
 +  \eeG 
\| \Mcal( g_1, h_1) - \Mstar\|_{L^2C_{\oldq+1}^{1, \alpha}(\Mbf)},
\\
&  
   \| h - h_1 \|_{L^2C_{\oldq}^{2, \alpha}(B_{2R})}
 \lesssim 
\eeG   \| \Hcal( g_1, h_1) - \Hstar\|_{L^2C_{\oldp + 2}^{0, \alpha}(\Mbf)}
 +  
\| \Mcal( g_1, h_1) - \Mstar\|_{L^2C_{\oldq+1}^{1, \alpha}(\Mbf)},
\endaligned 
$$
Consequently, it is obvious that we can also control the following weighted norms (for any fixed $\theta > 0$): 
\bel{inball} 
\aligned
& \| g - g_1 \|_{L^2C_{\theta}^{2, \alpha}(B_{2R})}
  \lesssim
 R^{\theta - \oldp} 
\Big(
 \| \Hcal( g_1, h_1) - \Hstar\|_{L^2C_{\oldp + 2}^{0, \alpha}(\Mbf)}
 +  \eeG 
\| \Mcal( g_1, h_1) - \Mstar\|_{L^2C_{\oldq+1}^{1, \alpha}(\Mbf)}
\Big),
\\
&  \| h - h_1 \|_{L^2C_{\theta}^{2, \alpha}(B_{2R})}
  \lesssim
R^{\theta - \oldp} 
\Big(\eeG 
 \| \Hcal( g_1, h_1) - \Hstar\|_{L^2C_{\oldp + 2}^{0, \alpha}(\Mbf)}
 +  
\| \Mcal( g_1, h_1) - \Mstar\|_{L^2C_{\oldq+1}^{1, \alpha}(\Mbf)}
\Big).  
\endaligned
\ee 
This provides us with suitable estimates in any compact domain. 
The arguments at the end of Section~\ref{remarkmanifold} explain how our analysis can be reduced to the Euclidean space $\RR^3$ and, therefore, from now on without loss of generality, we assume that $(\Mbf, e) = (\RR^3, \delta)$.

\paragraph{Basic decay rates.}

By applying Theorem~\ref{theo-prec}  with the exponents $p = 1/2$ and $q = 3/2$, we know that there exists a pair $(g_2, h_2) \in (d\Gcal)^{-1}_{(g_1, h_1)}\big( L^2C^{0,\alpha}_{5/2}(\RR^3) \times L^2C^{1,\alpha}_{5/2}(\RR^3) \big)$ such that 
$(g,h) = (g_1 + g_2, h_1 + h_2)$
is a solution to 
 the Einstein constraints
  enjoying the following (and possibly slow at this stage) decay  
\bel{8:equa-213}
\aligned
& \| g_2 \|_{L^2C_{1/2}^{2, \alpha}(\RR^3)} 
 \lesssim
\| \Hcal(g_1,h_1) - \Hstar \|_{L^2C_{5/2}^{\alpha}(\RR^3)}
 + \eeG \, \| \Mcal(g_1,h_1) - \Mstar \|_{L^2C_{5/2}^{1, \alpha}(\RR^3)}, 
\\
& \| h_2 \|_{L^2C_{3/2}^{2, \alpha}(\RR^3)}
 \lesssim
\eeG \, \| \Hcal(g_1,h_1) - \Hstar \|_{L^2C_{5/2}^{\alpha}(\RR^3)}
 +  
 \| \Mcal(g_1,h_1) - \Mstar \|_{L^2C_{5/2}^{1, \alpha}(\RR^3)}. 
\endaligned 
\ee
Moreover, $(g_2, h_2)$ belongs to the image of the adjoint Hamiltonian and momentum operators, and so 
there exists a pair $(u_2,Z_2) \in C^{4,\alpha}_{1/2}(\RR^3)  \times C^{3,\alpha}_{1/2}(\RR^3)$ such that  
  (after recalling the adjoint equations \eqref{eq2:hua})
\bel{2:hu} 
 - \big(\Delta_{g_1}u_2\big)g_1 + \Hess_{g_1}(u_2) - u_2\Ric_{g_1} + h_1*h_1*u_2 + \nabla_{g_1}(h_1*Z_2)
 = {g_2 \over r^2},
\qquad \quad 
 - {1 \over 2}  \Lcal_{Z_2}g_1 + h_1*u_2 =h_2. 
\ee
Here, $r=r(x)$ defined on $(\Mbf, e) = (\RR^3, \delta)$ was given in Remark~\ref{rem22}.

In this setup, it is clear that a decay rate for the dual variables $(u_2, Z_2)$ implies a decay rate for the original variables $(g_2, h_2)$ since
\bel{mind}
\aligned
\| g_2 \|_{C^{2, \alpha}_\theta} 
& \lesssim 
\| u_2 \|_{C^{4, \alpha}_\theta} + \eeG \, \| Z_2 \|_{C^{3, \alpha}_\theta}, 
\qquad\quad
\| h_2 \|_{C^{2, \alpha}_{\theta + 1}} 
\lesssim \eeG \, \| u_2 \|_{C^{4, \alpha}_\theta} + \| Z_2 \|_{C^{3, \alpha}_\theta}
\endaligned
\ee
for any decay exponent $\theta >0$ (provided the right-hand sides are finite). 
Before going further, we summarize our current decay result as 
\bel{eq:optim}
(g_2,h_2,u_2,Z_2) \in C^{2,\alpha}_{a_1} \times C^{2,\alpha}_{a_2}\times C^{4,\alpha}_{a_3}\times C^{3,\alpha}_{a_4}
\ee
with (at this initial stage of the proof) 
\bel{eq:ineq4}
a_1 = a_3 = a_4 = 1/2, 
\qquad \quad
a_2= 3/2.
\ee
Throughout our forthcoming arguments, the inequalities $a_1, a_3, a_4 \geq 1/2$ and $a_2 \geq 3/2$ always hold, and 
we work toward ``improving'' the basic rates in \eqref{eq:ineq4}.


\subsection{Preliminary sub-harmonic estimates for all variables} 

\paragraph{A weaker condition is sufficient to initiate the argument.} 

We begin by deriving {\sl sub-harmonic} estimates for both the metric and extrinsic curvature variables. 
In order to clearly explain the respective role of our conditions on decay exponents, we find it convenient to work
 first under a weakened version of  our conditions \eqref{eq:strongtameH-repeat} and \eqref{eq:strongtameM-repeat}, 
 namely instead of these two conditions we solely assume 
\bel{eq:strongtameH-repeat-weak}
\ppM \geq 3/4.  
\ee   
Our first aim is to establish \eqref{eq:optim} the {\sl first set of improved exponents}
(for any sufficiently small but fixed $\eps, \eps' \in (0,1)$) 
\bel{eq=claim4} 
\aligned
& a_1 = a_3 = \min(a,1-\eps'),
\qquad   \quad 
a_2 = \min(\qqM+1/2, 2), 
\qquad \quad
a_4 = \min(\qqM  -  1, 1 - \eps), 
\endaligned
\ee 
in which $a \geq 3/4$ (cf.~more precisely \eqref{equa:defpt} and \eqref{equa-the-trace-only-a}, below).

\paragraph{Equation for the trace of the metric deformation.}

From the definition of the scalar curvature, we have 
\bel{equa-Rscal}
\aligned
R_g - R_{g_1}  
& =  \sum_{i,j} \big( \del_{ij} g_{ij}  -   \del_{jj} g_{ii} \big)
- \sum_{i,j} \big( \del_{ij} g_{1 ij}  -   \del_{jj} g_{1 ii} \big)
\\
& \quad 
 + g_2* \del^2 g_1   + g_2 * \del g_1 * \del g_1 + \del g_1 * \del g_2 
 +  (g_1 -e) * \del^2 g_2 
 + g_2 * \del^2 g_2
 + \del g_2 * \del g_2 
\\ 
& =: \big( \sum_{i,j}\del_{ij} g_{2ij} \big) - \Delta_e (\Tr_e g_2) - \smallcal_0(g_2), 
\endaligned
\ee
where we have introduced the notation $\smallcal_0(g_2)$ for various terms involving $g_2$ 
(which, as we will see, enjoy comparatively better decay).
We recall that the star symbol correspond to various contractions involving possibly both metrics 
$g_1$ and $g=g_1+g_2$. 
In view of the Hamiltonian constraint $R_g = \Hstar + | h|_g^2  -  \big( \Tr_g h \big)^2/ 2$ satisfied by $(g,h)$, we can rewrite the above identity as
$$
\aligned
 -  \Delta_e (\Tr_e g_2) +   \textstyle  \sum_{i,j} \del_{ij} g_{2ij} 
& = \Hstar + |h|_g^2  - (1/2) \big( \Tr_g h \big)^2 -  R_{g_1}  + \smallcal_0(g_2)
\\
& = \Big(\Hstar + |h_1|_{g_1}^2  - (1/2) (\Tr_{g_1} h_1)^2 -  R_{g_1} \Big) 
+ \smallcal_0(g_2) +  ( h_1 * h_2 + h_2 * h_2) + g_2*h_1*h_1, 
\endaligned
$$
and, using the Hamiltonian operator applied to the seed data $(g_1,h_1)$, 
\bel{eq:hR}
\aligned
 -  \Delta_e (\Tr_e g_2) +   \textstyle  \sum_{i, j} \del_{ij} g_{2ij} 
& = \big( \Hstar - \Hcal( g_1, h_1) \big) + \smallcal_0(g_2) + \smallcal_1(h_2) + \smallcal_2(g_2), 
\endaligned
\ee
where we have introduced the notation $\smallcal_1(h_2) = h_1 * h_2 + h_2 * h_2$ and $\smallcal_2(g_2) := g_2*h_1*h_1$.  


Clearly, the following identity holds in $\RR^3$  
\be
 \textstyle  \sum_{i, j} 
 \del_{ij} \Big( - |x|^2 (\Delta_e u_2) \, e + |x|^2 \Hess_e(u_2)\Big)_{ij} =  - 4 \, \Delta_e u_2, 
\ee
where $|x|^2 = \sum_j (x^j)^2$ and $e=(e_{ij}) = (\delta_{ij})$. 
Using the distance function $r=r(x)$ in Remark~\ref{rem22} and differentiating the expression of~$g_2$ given in~\eqref{2:hu}, 
 it follows that
$$ 
\aligned
\sum_{i,j}
\del_{ij}\Bigg( {|x|^2 \over r^2} g_{2ij} \Bigg) 
& = \sum_{i,j}
\del_{ij}\Bigg( |x|^2  \Big( 
-\big(\Delta_{g_1}u_2 \big)g_1 + \Hess_{g_1}(u_2) -u_2\Ric_{g_1} + h_1*h_1*u_2 + \nabla_{g_1}(h_1*Z_2) 
\Big)_{ij}
\Bigg)
\\
& = -4 \, \Delta_e u_2 
+ \del^2\big( |x|^2 \big( (g_1 - e) * \del^2 u_2 + \del g_1 * \del u_2 \big) \Big)
\\
& \quad 
 + \del^2\Big( |x|^2 (u_2 \Ric_{g_1} +  h_1*h_1*u_2) \Big) 
+ \del^2\big( |x|^2 \nabla_{g_1} (h_1*Z_2) \big)
\\
&
=: -4 \, \Delta_e u_2 + \smallcal_3(u_2,Z_2). 
\endaligned
$$
We now use that $\sum_{i, j} \del_{ij}\big( {|x|^2 \over r^2} g_{2ij} \big) 
= \sum_{i, j}  \del_{ij} g_{2ij} + \del_{ij}\Big( {|x|^2 - r^2 \over r^2} g_{2ij} \Big)$ and arrive at 
\bel{hijij}
\sum_{i, j}  \del_{ij} g_{2ij} 
= 
- 4 \, \Delta_e u_2 + \smallcal_3(u_2,Z_2) 
+ \del_{ij} \big( ( 1 - |x|^2/r^2) \, g_{2ij} \big), 
\ee
which is one of the terms in the left-hand side of \eqref{eq:hR}. 
Observe the last term in \eqref{hijij} vanishes identically for all $|x| \geq 1$.

On the other hand, taking now the trace 
of~$g_2$ given in~\eqref{2:hu} 
(with respect to $g_1$) and then expanding $g_1 = e + (g_1-e)$, we obtain 
\bel{laplace u}
\aligned
 - 2\Delta_e u_2 
& = {1\over r^2} \Tr_e g_2 + R_{g_1} u_2 + 
\Big( 
h_1*h_1*u_2 + \nabla_{g_1}(h_1*Z_2)
+ \del g_1 *  \del u_2 + {1 \over r^2} (g_1 -e) * g_2
\Big)
\\
&
=: 
{1\over r^2} \Tr_e g_2 + R_{g_1} u_2 + \smallcal_4(g_2, u_2,Z_2) .
\endaligned
\ee 
Taking \eqref{hijij} and \eqref{laplace u} into account in \eqref{eq:hR}, 
we arrive at the following elliptic equation\footnote{recalling that $r$ is bounded below} (after absorbing numerical factors in our symbolic notation $*$). 

\begin{claim} 
\label{trace h}
The trace of the metric deformation satisfies 
$$ 
\aligned
- \Delta_e \big( \Tr_e g_2 \big) + {2 \over r^2} \Tr_e g_2  
& 
= \big( \Hstar - \Hcal( g_1, h_1) \big) + \smallcal_5(g_2,h_2,u_2,Z_2),  
\\
 \smallcal_5(g_2,h_2,u_2,Z_2) 
 & := \smallcal_0(g_2)  + \smallcal_1(h_2) + \smallcal_2(g_2)
 - \smallcal_3(u_2,Z_2)
 \\
 & \quad  - 2 \, \smallcal_4(g_2, u_2,Z_2)
 - \del_{ij} \big( ( 1 - |x|^2/r^2) \, g_{2ij} \big)
 - 2 \, R_{g_1} u_2. 
\endaligned
$$ 
\end{claim}

\vskip.15cm

\paragraph{Improving the decay of the trace of the metric deformation.} 

We now make use of the strong tame condition \eqref{eq:strongtameH} (restated in \eqref{eq:strongtameH-repeat}).  
We see that  the terms 
in the equation in Claim~\ref{trace h} 
consists of terms with decay exponents at least 
\bel{equ-table11} 
\aligned
\smallcal_0 : \quad & p_0 := a_1 + 2 + \min (\ppG, a_1)  \geq 3, 
\\
\smallcal_1 : \quad & p_1 := a_2 + \min (\qqG, a_2)  \geq 11/4, 
\\
\smallcal_2 : \quad & p_2 := a_1 + 2 \qqG \geq 3, 
\\
\smallcal_3 : \quad & p_3 := \min (a_3+2+\ppG, a_3+2q_G, a_4+1+\qqG)  \geq  \min(3, a_4+1+\qqG)  \geq 11/4,
 \\
\smallcal_4 : \quad &  p_4 := \min( a_1+2+\ppG, a_3 + 2 + \ppG, a_3+ 2 \qqG, a_4 + 1 + \qqG) 
\geq \min(3, a_4 + 1 + \qqG) \geq 11/4, 
\\
R_{g_1} u_2: \quad & a_3 + 2 + \ppG \geq 3,
\\
\Hstar - \Hcal( g_1, h_1): \quad & \ppM +2 \geq 11/4, 
\endaligned
\ee
while the term $ - \smallcal_3(u_2,Z_2) - 2 \, \smallcal_4(g_2, u_2,Z_2)$ has compact support. 
So the right-hand side $\big( \Hstar - \Hcal( g_1, h_1) \big) + \smallcal_5$ enjoys the decay $r^{-2 -\pt}$ 
for an exponent   
\bel{equa:defpt} 
\aligned
\pt 
& := -2 + \min \Big( p_0, p_1, p_2, p_3, p_4,  a_3 + 2 + \ppG, \ppM +2)
\\
& = \min(a_1+p_G,  2a_1, a_2+q_G-2,  2a_2-2, a_3+p_G,  a_4+q_G-1, \ppM +2).
\endaligned 
\ee
Then, by applying Proposition~\ref{laplaceplus} (concerning the ``modified'' Laplace operator $-\Delta + 2/r^2$) 
to the elliptic equation in Claim~\ref{trace h} with any exponent $a$ chosen such that 
\bse
\label{equa-the-trace-only}
\bel{equa-the-trace-only-a}  
3/4\leq a \leq \pt, \qquad a < 2, 
\ee
we see that $\Tr_e g_2$ decays at the rate $a$: 
\bel{equa:trac}
\|\Tr_e g_2 \|_{C^{2,\alpha}_a} 
\lesssim \|  \smallcal_5 (g_2,h_2,u_2,Z_2)  \|_{C^{0,\alpha}_{a + 2}}. 
\ee
Here, $a \geq 3/4$ therefore {\sl improves upon} the rate $a_1=1/2$ available in \eqref{eq:ineq4} ---for the trace of the metric, at least. 
\end{subequations}

\paragraph{Improving the decay of the (dual) metric variables.} 

Returning to the trace equation \eqref{laplace u}, we can infer an improved decay for the dual variable $u_2$. Indeed, applying Proposition~\ref{propo.classcialregular} (now concerning the standard Laplace operator $-\Delta$)
to \eqref{laplace u} we obtain the decay property~\eqref{eq:optim} now with the improved exponent \bse
\label{equa:15ab}
\bel{udecay} 
a_3 = \min(1 - \eps', a) \geq 3/4, 
\ee
$\eps'$ being sufficiently small but fixed. Indeed, Proposition~\ref{propo.classcialregular} applies since, in the equation~\eqref{laplace u}, 
 the source-terms
${1\over r^2} \Tr_e g_2 + R_{g_1} u_2 + \smallcal_4(g_2, u_2,Z_2)$ decays at least at the rate 
$\min( a+2, (a_3=1/2) + 2 + \ppG, p_4) \geq a_3+2$.

In turn, we can use the expression in \eqref{eq2:hua} which gives $g_2$ in terms of the adjoint Hamiltonian operator  and, by considering the decay of each term, it follows that in \eqref{eq:optim} for the metric deformation $g_2$ we find the same exponent as above
\bel{a11} 
a_1 = \min(1 - \eps', a) \geq 3/4. 
\ee 
This is all we can deduce from the Hamiltonian operator {\it ---at this stage}.
\ese
 

\paragraph{Second-order equation for the adjoint momentum variable.}

Taking now the divergence of the adjoint momentum equation with respect to the metric $g_1$ and expanding $g_1 = e + (g_1-e)$,
 we obtain 
$$
\aligned
2 \, \Div_{g_1}(h_2)
& = \big( - \Delta_e Z_2 - \Div_e(\nabla_e Z_2) \big) 
- \Div_{g_1}(h_1*u_2) + \del^2 \big((g_1 - e)*Z_2\big)
\endaligned
$$
thus 
\bel{div Z}
\aligned
 \LMcal(Z_2)
& = 
2 \, \Div_{g_1}(h_2) + \smallcal_6(u_2, Z_2),
\\
\smallcal_6(u_2, Z_2) 
& := \Div_{g_1}(h_1*u_2) - \del^2 \big((g_1 - e)*Z_2\big).
\endaligned
\ee
The divergence term $\Div_{g_1}(h_2)$ on the 
left-hand side is related to the momentum operator applied to $h$;
 namely, the equation $\Div_{g}(h) = \Mstar$ implies 
$$
\aligned
\Div_{g_1}(h_2) 
& =  \Div_g h_2 + \del g_2 * h_2 = \big( \Mstar - \Div_g h_1 \big) + \del g_2* h_2
= \big( \Mstar - \Mcal(g_1, h_1) \big) + \smallcal_7(g_2, h_2),
\endaligned
$$
in which $\smallcal_7(g_2, h_2) := \del g_2*h_1 + \del g_2* h_2$. 
Therefore, we can rewrite \eqref{div Z} in the form of a second-order elliptic system, as follows. 

\begin{claim} 
\label{div Z2}
The dual momentum variable satisfies  
$$ 
\LMcal(Z_2)  
= 2 \, \big( \Mstar - \Mcal(g_1, h_1) \big) + \smallcal_6(u_2, Z_2) + 2 \,\smallcal_7(g_2, h_2).  
$$ 
\end{claim}

\paragraph{Improving the decay of the momentum variables.}

We now make use of the strong tame condition \eqref{eq:strongtameM} (restated in \eqref{eq:strongtameM-repeat}) and proceed in two steps, as follows. 
\bei 

\item From the improved decay in~\eqref{udecay} and the property of the divergence of the adjoint momentum operator, as stated earlier
in Proposition~\ref{theorem elliptic system ii}, it follows that $Z_2 \in C^{3,\alpha}_{a_4}$ with the improved exponent 
(as announced in \eqref{eq=claim4}) 
\bel{equa-debut}
a_4 = \min(\qqM - 1,1 - \eps)
\ee
 for some sufficiently small but fixed $\eps \in (0,1)$. Namely, we use here that the right-hand side of 
enjoys: 
$$
\aligned
\Mstar - \Mcal(g_1, h_1): \quad & 1+\qqM,
\\
\smallcal_6(u_2, Z_2): \quad & p_6 := \min(a_3+1+\qqG, a_4+2+\ppG) 
\\
& \qquad \geq \min(\min(1 - \eps, a) +1 + \qqG, 3) 
\geq (3/4) + (5/4) + 1 = 3, 
\\
\smallcal_7(g_2, h_2): \quad & p_7 := a_1+1 + \min(\qqG, a_2) \geq 3. 
\endaligned
$$

\item Returning to the adjoint momentum equation $h_2 = -{1 \over 2}  \Lcal_{Z_2} g_1 + h_1*u_2$, 
which enjoys the decay rate 
$$
p_8= \min(a_4+1+\ppG, a_3+\qqG) 
\geq \min(2-\eps+1/2, \qqM +1/2, 3/4+5/4) 
\geq \min(\qqM+1/2, 2), 
$$  
 we infer that the extrinsic curvature satisfies 
$h_2 \in C^{2,\alpha}_{p_8}$.
In other words, we can now take the exponent for the momentum variable in \eqref{eq:optim} to be 
(announced in \eqref{eq=claim4}) 
\bel{a21:a31}
  a_2 = \min(\qqM+1/2, 2). 
\ee
\eei
This completes the analysis of the sub-harmonic level and the derivation of the first set of improved exponents~\eqref{eq=claim4}, which only required the conditions~\eqref{eq:strongtameH-repeat-weak}.


\subsection{Harmonic decay for the (dual) metric variables}

\paragraph{A weaker condition is sufficient for the metric at the harmonic level.} 

We now work with the following conditions (which are still weaker than~\eqref{eq:strongtameH-repeat} and \eqref{eq:strongtameM-repeat}): 
\bel{equa524} 
\ppM \geq 3/4,  
\qquad 
 \qqM  > \max(3 - \qqG, 3/2).
\qquad  
\Hcal(g_1,h_1) - \Hstar \in L^1(\Mbf). 
\ee   
We now seek to establish \eqref{eq:optim} in agreement with the following {\sl second set of improved exponents} 
(for any arbitrarily small but fixed $\eps \in (0,1)$): 
\bel{eq=claim4-bis} 
\aligned
& a_1 = a_3 = \min(\ppM,1),
\qquad   \quad 
  a_2 = 2,
\qquad \quad
a_4 = \min(\qqM  -  1, 1 - \eps). 
\endaligned
\ee 

\paragraph{Improved sub-harmonic decay for the metric variables.}  

Plugging our result \eqref{equa:15ab} and \eqref{a21:a31} in the expression of the exponent $\pt$ defined in \eqref{equa:defpt}, we see that the exponent $\pt$ associated with $\smallcal_5 (g_2,h_2,u_2,Z_2)$ and 
the trace in \eqref{equa-the-trace-only} satisfies $\pt \geq 1$, since we can now improve upon 
the worst estimates in \eqref{equ-table11}, namely  
\bel{equ-table11-2} 
\aligned
\smallcal_1 : \quad & p_1 = a_2 + \min (\qqG, a_2)  \geq 3, 
\\
\smallcal_3 : \quad & p_3 \geq \min(3,a_4+1+\qqG) \geq \min(3, \qqM + \qqG )\geq 3,  
 \\
\smallcal_4 : \quad &  p_4 \geq \min(3, a_4 + 1 + \qqG) \geq \min(3, \qqM + \qqG )\geq 3,  
\endaligned
\ee
so we can now pick up $a \geq 1$.  Consequently, this exponent $a$ is no longer a restriction in a second application of our technique in Step~3 and 
we can choose the exponents associated with the dual variable $u_2$ 
to be
\be
a_3 = \min(\ppM, 1 -\eps').
\ee 
Consequently, in view of the equation~\eqref{2:hu}, we can also take 
\be
a_1 = \min(\ppM, 1 -\eps')
\ee
for the metric deformation $g_2$. 
But this is not (yet) the desired decay and it remains to ``remove'' the small parameter~$\eps'$.

\paragraph{Almost harmonic decay for the metric.}

We revisit the inequalities in \eqref{equ-table11} and thanks to the assumptions \eqref{equa524} now write 
\bel{equ-table11-better} 
\aligned
\smallcal_0 : \quad & p_0 = \min (a_1 + 2 + \ppG, 2 a_1 + 2 ) = \min (\ppM + 2 + \ppG,  3-\eps+ \ppG, 2 \ppM + 2, 4 - 2\eps )  > 3, 
\\
\smallcal_1 : \quad & p_1 = \min (a_2 + \qqG, 2a_2) = \min (\ppM+1/2 + \qqG, 2\ppM+1, 2 + \qqG, 4)  >3, 
\\
\smallcal_2 : \quad & p_2 = a_1 + 2 \qqG =  \min(\ppM + 2 \qqG, 1 -\eps + 2\qqG) > 3, 
\\
\smallcal_3 : \quad & p_3 = \min (a_3+2+\ppG, a_3+2q_G, a_4+1+\qqG) 
\geq \min (13/4, \qqM+\qqG, \qqM+5/2)  >3,
 \\
\smallcal_4 : \quad &  p_4 = \min( a_1+2+\ppG, a_3 + 2 + \ppG, a_3+ 2 \qqG, a_4 + 1 + \qqG) 
\geq \min(13/4, \qqM + \qqG) >3, 
\\
R_{g_1} u_2: \quad & a_3 + 2 + \ppG = \min(2+\ppM+\ppG, 3 -\eps + \ppG) > 3, 
\endaligned
\ee
Consequently, using also the assumption $\Hstar - \Hcal(g_1, h_1) \in L^1(\RR^3)$ in \eqref{equa524} and applying  
\eqref{laplace + 2} in Proposition~\ref{laplaceplus} to the elliptic equation for the trace in Claim~\ref{trace h}, we find the {\sl integrability property} 
\bel{Trg:L1}
r^{-2}\Tr_e g_2 \in L^1.
\ee 
In turn, we find that the right-hand side 
${1\over r^2} \Tr_e g_2 + R_{g_1} u_2 + \smallcal_4(g_2, u_2,Z_2)$ of the equation~\eqref{laplace u} (for the dual variable $u_2$) is integrable and, by applying Proposition~\ref{proposition behavior at infinity} concerning solutions with harmonic decay 
we can now choose $a_3 = \min(\ppM, 1)$. Finally, returning once more to the linearized Hamiltonian equation, namely the 
 expression in \eqref{eq2:hua} giving $g_2$, 
  overall we conclude that we can choose  
\bel{a21:a12}
a_1 = \min(\ppM, 1), \qquad \quad a_3 = \min(\ppM, 1). 
\ee
Together with  \eqref{equa-debut} and \eqref{a21:a31}, this is the decay announced in \eqref{eq=claim4-bis}.

\paragraph{Harmonic decay estimates.}

We now express our results in terms of norms in order to keep track of the factors $\eeG$. 
Applying Proposition~\ref{propo.classcialregular} to  
 the trace equation~\eqref{laplace u}, 
we obtain (for any $\theta_1 \in (0,1)$)  
$$
\aligned
\| u _2 \|_{C^{4,\alpha}_{\theta_1}} 
& \lesssim
\| r^{-2} \Tr_e g_2 \|_{C^{2,\alpha}_{\theta_1 + 2}}
+ \|R_{g_1} u_2 + h_1 * h_1 * u_2 \|_{C^{2,\alpha}_{\theta_1 + 2}}
+ \| \nabla_{g_1}( h_1 * Z_2) \|_{C^{2,\alpha}_{\theta_1 + 2}}
\\
& \quad 
+ \| \del g_1 * \del u_2 \|_{C^{2,\alpha}_{\theta_1 + 2}} 
+ \| r^{-2} (g_1-e) * g_2 \|_{C^{2,\alpha}_{\theta_1 + 2}}. 
\endaligned
$$
\bse
\label{equa-collecti}%
Using the condition $ \qqM  > \max(3 - \qqG, 3/2)$, 
we find 
$$
\| \nabla_{g_1}( h_1 * Z_2) \|_{C^{2,\alpha}_{\theta_1 + 2}} 
\lesssim 
\eeG \, \| Z_2 \|_{C^{3,\alpha}_{\theta_1 +2-\qqG}}  
\lesssim \eeG \, \| Z_2 \|_{C^{3,\alpha}_{\min(\qqM - 1,\theta_1)}}
$$
and similarly 
 $\| \del g_1 * \del u_2  \|_{C^{2,\alpha}_{\theta_1 + 2}} \lesssim \eeG \, \| u_2 \|_{C^{3,\alpha}_{\theta_1 + 2 +\ppG}}$ 
and $\| r^{-2} (g_1-e) * g_2  \|_{C^{2,\alpha}_{\theta_1 + 2}} \lesssim \eeG \, \| g_2 \|_{C^{2,\alpha}_{\theta_1 + 2 + \ppG-+2}}$. 
Hence a fortiori we obtain 
\be
\| u _2 \|_{C^{4,\alpha}_{\theta_1}} \lesssim
\| \Tr_e g_2 \|_{C^{2,\alpha}_{\theta_1 + 4}} + \eeG \, \Omega, 
\qquad
\Omega := \| u_2 \|_{C^{4, \alpha}_{\theta_1}} + \| g_2 \|_{C^{2, \alpha}_{\theta_1}} + \| Z_2 \|_{C^{3,\alpha}_{\min (\qqM - 1,\theta_1 )}}.
\ee
Similarly, for the momentum equation in Claim~\ref{div Z2}, we write 
$$
\| \smallcal_6(u_2, Z_2) \|_{C^{1,\alpha}_{\min ( \qqM - 1,\theta_1+2 )}} 
\lesssim
\eeG \, \|  u_2 \|_{C^{3,\alpha}_{\qqG+\min ( \qqM - 2,\theta_1+1 )}} 
+ \eeG \,  \| Z_2 \|_{C^{3,\alpha}_{\ppG + \min ( \qqM - 3,\theta_1)}} 
\lesssim  \eeG \, \Omega, 
$$
 $$
\| \smallcal_7(g_2, h_2) \|_{C^{1,\alpha}_{\min ( \qqM - 1,\theta_1 +2)}} 
\lesssim 
 \eeG \, \| \del g_2 \|_{C^{1,\alpha}_{\qqG + \min ( \qqM - 1,\theta_1 +2)}} 
 + \eeG \, \| \del g_2 \|_{C^{1,\alpha}_{a_1 + \min ( \qqM - 1,\theta_1 +2)}} 
 \lesssim \eeG \, \Omega,  
 $$
 hence  thanks to Proposition~\ref{theorem elliptic system ii} 
 we find 
\be
\| Z_2 \|_{C^{3,\alpha}_{\min ( \qqM - 1,\theta_1 )}} 
\lesssim \| \Mcal (g_1, h_1) -  \Mstar \|_{C^{1,\alpha}_{\min (\qqM + 1,\theta_1 + 2)}} 
+ \eeG \, \Omega. 
\ee 
Finally, we similarly control the trace of the metric deformation, namely by applying Proposition~\ref{laplaceplus} to the equation in Claim~\ref{trace h}: 
\be
\aligned
\| \Tr_e g_2 \|_{C^{2,\alpha}_{\theta_1}} 
& \lesssim
\| \Hcal(g_1, h_1) - \Hstar \|_{C^{2,\alpha}_{\theta_1 + 2}} 
+  
\| \smallcal_5(g_2,h_2,u_2,Z_2) \|_{C^{2,\alpha}_{\theta_1 + 2}} 
\lesssim
\| \Hcal(g_1, h_1) - \Hstar \|_{C^{2,\alpha}_{\theta_1 + 2}} 
+  \eeG \, \Omega.  
\endaligned
\ee
\ese
Consequently, recalling \eqref{mind} and collecting the inequalities \eqref{equa-collecti}, we arrive at the desired estimate for the dual unknowns $(u_2, Z_2)$ (for any $\theta_1 \in (0,1)$): 
\bel{8:subhar}
\aligned
\| u _2 \|_{C^{4,\alpha}_{\theta_1}}  
& \lesssim \| \Hcal(g_1, h_1) - \Hstar \|_{C^{2,\alpha}_{\theta_1 + 2}} + \eeG \, \| \Mcal (g_1, h_1) -  \Mstar \|_{C^{1,\alpha}_{\min (\qqM + 1,\theta_1 + 2)}},
\\
\| Z _2 \|_{C^{3,\alpha}_{\min (\qqM - 1,\theta_1 )}}  
& \lesssim \eeG \, \|\Hcal(g_1, h_1) - \Hstar \|_{C^{2,\alpha}_{\theta_1 + 2}} +  \| \Mcal (g_1, h_1) -  \Mstar \|_{C^{1,\alpha}_{\min (\qqM + 1,\theta_1 + 2)}}. 
\endaligned
\ee
The main unknowns $(g_2, h_2)$ are also controlled thanks to \eqref{mind}. 
This completes our analysis at the harmonic level. 


\subsection{Super-harmonic estimates}
\label{sec-super-cri}

\paragraph{Mass corrector and super-harmonic decay property for the metric.} 

From now on we do impose the full list of conditions~\eqref{eq:strongtameH-repeat}. 
It remains to exhibit the Hessian term and study the mass corrector. We thus seek estimates at the super-harmonic rate in \eqref{eq:strongexpo}, more precisely: 
\bel{eq:strongexpo-repeat} 
\aligned
&  (1,2) \leq (\pstar, \qstar) \leq (\ppM, \qqM), 
\qquad
\\
& \pstar \leq \min (\ppG + 1, \qqG + \qqM - 2),
\qquad  
&& \pstar < \min (2, \qqG), 
\\
& \qstar \leq \min (\ppG + 2, \qqG +1), 
\qquad 
&&
\qstar < 3. 
&
\endaligned
\ee   
We control first 
$\| u_2 - \mt /r \|_{C^{4,\alpha}_{\pstar}}$ for some (real) constant $\mt$, as follows.  
For instance we have 
$$
\aligned
\|R_{g_1} u_2\|_{C^{2,\alpha}_{\pstar + 2}} 
& \lesssim \|R_{g_1}\|_{C^{2,\alpha}_{\ppG + 2}}\| u_2 - \mt/r \|_{C^{4,\alpha}_{\pstar}} + |\mt| \, \| r^{-1} R_{g_1} \|_{C^{2,\alpha}_{\pstar + 2}} 
\lesssim \eeG \, \|u_2 - \mt r^{-1}\|_{C^{4,\alpha}_{\pstar}} + |\mt| \, \| r^{-1} R_{g_1} \|_{C^{2,\alpha}_{\pstar + 2}}, 
\endaligned
$$
and similarly for other terms in the equation \eqref{laplace u} for $\Delta_e u_2$. 
By applying Corollary~\ref{remarkpoisson} it then follows that the dual metric variable $u_2$ satisfies (for some constant $\mt$)
\bel{u2:maines1}
\| u_2 - \mt /r \|_{C^{4,\alpha}_{\pstar}} 
\lesssim |\mt| + \| r^{-2} \Tr_e g_2 \|_{C^{2,\alpha}_{\pstar + 2}} 
+ \eeG \, \| Z_2 \|_{C^{3,\alpha}_{1 + \pstar - \qqG}}.
\ee

On the other hand, by our choice of $\pstar$ we have $1 + \pstar - \qqG \in (0,1)$ as well as 
$1+ \pstar - \qqG \leq \qqM-1$, therefore 
\eqref{8:subhar} applies 
 (by picking up a value $\theta_1 \in (0,1)$ such that $1 + \pstar - \qqG \leq \min (\qqM - 1,\theta_1)$) 
and 
gives us the following control of the dual variable $Z_2$: 
$$
\| Z_2 \|_{C^{3,\alpha}_{1 + \pstar - \qqG}} \lesssim \eeG \, \|\Hcal(g_1, h_1) - \Hstar \|_{C^{2,\alpha}_{3 + \pstar - \qqG}}
 +  \| \Mcal (g_1, h_1) -  \Mstar \|_{C^{1,\alpha}_{\min(\qqM+1, 3 + \pstar - \qqG)}}. 
$$
Furthermore, applying Proposition~\ref{laplaceplus} to the equation in Claim~\ref{trace h} 
for the operator $- \Delta_e \Tr_e g_2 + (2/r^2) \Tr_e g_2$, we find 
$$
\aligned
\| \Tr_e g_2 \|_{C^{2,\alpha}_{\pstar}} 
& \lesssim
\| \Hcal(g_1, h_1) - \Hstar \|_{C^{0,\alpha}_{\pstar + 2}} 
 +  
 \| \smallcal_5(g_2,h_2,u_2,Z_2) \|_{C^{0,\alpha}_{\pstar + 2}}
\endaligned
$$ 
and we can also control $\| \smallcal_5(g_2,h_2,u_2,Z_2) \|_{C^{0,\alpha}_{\pstar + 2}}$ by completely 
similar arguments. Hence, recalling \eqref{mind} we can rewrite the estimate \eqref{u2:maines1} as 
\bel{equa-5266}
\|u_2 - \mt /r \|_{C^{4,\alpha}_{\pstar}} 
\lesssim \| \Hcal(g_1, h_1) - \Hstar \|_{C^{2,\alpha}_{3 + \pstar - \qqG}}  
+  \eeG \, \|\Mcal(g_1, h_1) - \Mstar \|_{C^{1,\alpha}_{\min(\qqM+1, 3 + \pstar - \qqG)}}  
+ |\mt|.
\ee

\paragraph{Controlling the mass corrector.} 

The constant $\mt$ is now estimated in terms of the data.  
First of all, in \eqref{equa:trac} and \eqref{Trg:L1} we have established that $r^{-2} \Tr_e g_2 \in  L^1\cap C^{2,\alpha}_1$, so, in view of Claim~\ref{trace h}, we also have $\Delta (\Tr_e g_2) \in L^1 \cap C^{2,\alpha}_1$. 
Applying Proposition~\ref{proposition behavior at infinity},  we obtain 
$$ 
\lim_{r\to +\infty} r \, \Tr_e g_2 = - {1 \over 4\pi} \int_{\RR^3} \Delta (\Tr_e g_2) \, dx = 0, 
\qquad \qquad 
\lim_{r\to +\infty} \big| r^2 \nabla (\Tr_e g_2) \big|
 = {1 \over 4\pi} \Big| \int_{\RR^3} \Delta (\Tr_e g_2) \, dx \Big| = 0, 
$$
that is, 
\bel{81:forremark}
\lim_{r \to +\infty} r^3 \, \Big| {\Tr_e g_2 \over r^2} \Big| 
=
0 
=
\lim_{r \to +\infty}  r^4 \, \Big| \nabla \Bigg({\Tr_e g_2 \over r^2} \Bigg) \Big|.
\ee 
Next, in view of Proposition~\ref{proposition behavior at infinity}, $\mt$ is the coefficient of the asymptotic behavior of $u_2$, that is, 
$$
\mt = \lim_{r \to +\infty} r \, u_2
= - {1 \over 4 \pi } \int_{\RR^3} \Delta u_2 \, dx.
$$  
Combining this observation with \eqref{hijij} and checking that 
the terms therein decay sufficiently fast, 
we find 
\bel{minfty.formula}
\aligned
\mt 
& = {1 \over 16 \pi} \int_{\RR^3} \del_{ij} ((|x|^2/r^2) g_{2ij}) \, dx - {1 \over 16 \pi} \int_{\RR^3} \smallcal_3(u_2,Z_2)  \, dx 
\\
& = 
{1 \over 16 \pi} \lim_{r \to +\infty} {}\int_{S_r(0)} \del_j g_{2ij} {x_j \over r} \, d\omega 
- {1 \over 16 \pi} \int_{\RR^3} \smallcal_3(u_2,Z_2)  \, dx, 
\endaligned
\ee
after an integration by parts.

On the other hand, integrating the identity \eqref{eq:hR} over $\RR^3$ 
and using the asymptotic property \eqref{81:forremark} we obtain 
$$
\aligned
 \lim_{r \to +\infty} {}\int_{S_r(0)} \del_j g_{2ij} {x_j \over r} \, d\omega  = 
 \int_{\RR^3}  \del_{ij} g_{2ij} \, dx 
& = \int_{\RR^3} (\Hstar - \Hcal( g_1, h_1) ) \, dx
  + \int_{\RR^3} \Big( \smallcal_0(g_2) + \smallcal_1(h_2) + \smallcal_2(g_2) \Big) \, dx, 
\endaligned
$$
and, in combination with \eqref{minfty.formula}, we arrive at a formula for the mass corrector: 
\bel{minfty.fola:2}
\aligned
\mt & = {1 \over 16 \pi} \int_{\RR^3} (\Hstar - \Hcal( g_1, h_1) ) \, dx
  + {1 \over 16 \pi} \int_{\RR^3}  \Big( \smallcal_0(g_2) + \smallcal_1(h_2) + \smallcal_2(g_2) - \smallcal_3(u_2,Z_2)  \Big) \, dx. 
\endaligned
\ee
Using H\"older's inequality 
and \eqref{8:equa-213}, from \eqref{minfty.fola:2} we deduce the desired estimate of the mass corrector: 
\bel{82:mexact:a}
\Big| \mt - {1 \over 16 \pi} \int_{\RR^3} (\Hstar - \Hcal( g_1, h_1) ) \, dx \Big|
\lesssim \eeG \, \| \Hcal(g_1,h_1) - \Hstar \|_{L^2C_{5/2}^{\alpha}}
 + \eeG \, \| \Mcal(g_1,h_1) - \Mstar \|_{L^2C_{5/2}^{1, \alpha}}, 
\ee 
where for simplicity we only state estimates with decay rate $5/2$ as this is sufficient for our main purpose. 


\paragraph{Super-harmonic decay for the extrinsic curvature.}

From our earlier analysis of the adjoint momentum operator, especially Proposition~\ref{proposition behavior LMcal 2}, similar arguments to those above show that, provided the second strongly tame conditions \eqref{eq:strongtameM} hold 
and by 
relying on the elliptic system in Claim~\ref{div Z2}, there exists a constant vector $\Pt \in \RR^3$ such that
\bel{82:Pinfty}
\Big| \Pt - {1 \over 8 \pi} \int_{\RR^3} (\Mstar - \Mcal (g_1, h_1)) \, dx \Big| 
\lesssim 
\eeG \, \| \Hcal(g_1,h_1) - \Hstar \|_{L^2C_{5/2}^{\alpha}}
 + \eeG \, \| \Mcal(g_1,h_1) - \Mstar \|_{L^2C_{5/2}^{1, \alpha}}
\ee
and
\bel{eq:6}
\| Z_2 + 2 V \cdot \Pt \|_{C^{3,\alpha}_{\qstar - 1}} 
\lesssim  \eeG \, \| \Hcal(g_1, h_1) - \Hstar \|_{C^{2,\alpha}_{\pstar + 2}}  +  \|\Mcal(g_1, h_1) - \Mstar \|_{C^{2,\alpha}_{\qstar + 1}} + \eeG| \mt | +  | \Pt |, 
\ee  
in which $V_{ij} = \big( x_i x_j + 3 r^2 \delta_{ij} \big)/(2 r^3)$ (for $i,j = 1,2,3$). 
Consequently, under the conditions in Theorem~\ref{theo:deuxieme}, by taking our results above into account in the  dual version of the linearized Einstein constraints
  \eqref{2:hu} and recalling our earlier  
  estimate \eqref{8:subhar}
  together with \eqref{u2:maines1}-\eqref{equa-5266},
  we obtain
$$
\aligned
\| g_2 - \mt \, r^2 \, \Hess_{e}(1/r) \|_{C_{\pstar}^{2,\alpha}} 
& \lesssim  
\eeG \, \| \Hcal(g_1,h_1) - \Hstar \|_{L^2C_{5/2}^{\alpha}} 
+ \| \Hcal(g_1,h_1) - \Hstar \|_{C^{2,\alpha}_{\pstar + 2}} + \Big| \int_{\RR^3} ( \Hcal(g_1,h_1) - \Hstar ) \, dx \Big|
\\
& \quad+ \eeG \, \| \Mcal(g_1,h_1) - \Mstar \|_{L^2C_{5/2}^{1, \alpha}} + \eeG  \, \| \Mcal(g_1,h_1) - \Mstar \|_{C^{2,\alpha}_{\min(\qqM + 1, \pstar + 2)}} 
\endaligned
$$
and
$$
\aligned
\| h_2 - \Lcal_{(V \cdot \Pt)} e \|_{C_{\qstar}^{2,\alpha}}
& \lesssim  \eeG \, \| \Mcal(g_1,h_1) - \Mstar \|_{L^2C_{5/2}^{\alpha}}
 +
\| \Mcal(g_1,h_1) - \Mstar \|_{C^{2,\alpha}_{\qstar + 1}} 
+ \sum_{1 \leq j \leq 3} \Big| \int_{\RR^3} \big( \Mcal( g_1, h_1) - \Mstar \big)_j \, dx \Big|
\\
& \quad + \eeG \, \| \Hcal(g_1,h_1) - \Hstar \|_{L^2C_{5/2}^{\alpha}}
 + 
 \eeG \, \| \Hcal(g_1, h_1) - \Hstar \|_{C^{2,\alpha}_{\pstar + 2}} + \eeG \Big| \int_{\RR^3} ( \Hcal(g_1,h_1) - \Hstar ) \, dx \Big|. 
\endaligned
$$


\subsection{Final observations concerning the case $\pstar = 1$ and $\qstar = 2$.}

We now study the asymptotic behavior of $(g_2, h_2)$ in the {\sl special case} $\pstar = 1$ and $\qstar = 2$. We claim that, based on our previous arguments, when $\pstar = 1$ we also have\footnote{which, of course, is also true if it happens that $\pstar >1$}: 
\bel{uZc0:a}
\lim_{r \to +\infty} \big( r \, \big| g_2 - \mt \, r^2 \, \Hess_{e} (1 /r ) \big| 
+ r^2 \, \big| \del( g_2 - \mt \, r^2 \, \Hess_{e}  (1 /r ) ) \big| \big) = 0.
\ee 
Indeed, observing that
$\int_{\RR^3} \Delta ( u_2 - \mt r^{-1}) \, dx = 0$
and 
$\Delta ( u_2 - \mt r^{-1}) = \Delta u_2$ (for $r \geq 1$), 
and applying Proposition~\ref{proposition behavior at infinity k} to $\Delta (u_2 - \mt r^{-1} )$, 
we deduce from the elliptic equation \eqref{laplace u} and the property of the trace in \eqref{81:forremark}
$$
\lim_{r \to + \infty} \max_{0 \leq i \leq 3} r^{1 + i} \, |\del^i (u_2 - \mt r^{-1} )|  = 0.
$$
Taking into account the Hamiltonian constraint and the 
 adjoint equations \eqref{eq2:hua},
we obtain \eqref{uZc0:a}
as claimed.
Similarly, in the case $\qstar = 2$ with 
$V^\infty := \Pt \cdot V$
we also have
$$
\lim_{r \to +\infty} \max_{0 \leq i \leq 3} r^{1 + i} \, \big| \del^i (Z_2 + V^\infty) \big| = 0
$$
and, by the adjoint momentum equation, 
$\lim_{r \to +\infty} \max_{i = 0,1,2} r^{2 + i} \, \big| \del^i (h_2 - (1/2)\Lcal_{V^\infty} e) \big| = 0$. 


We also observe that in the super-harmonic regime of interest \eqref{eq:strongtameH}--\eqref{eq:strongtameM} and, more precisely, thanks to the 
integrability condition $\Mcal(g_1, h_1) - \Mstar \in L^1$, the strict inequality $\pstar < \qqG$ can be {\sl improved to}
 $\pstar \leq \qqG$, by writing 
$$
\aligned 
\| g_2 - \mt \, r^2 \, \Hess_{e} (1/r) \|_{C_{\pstar}^{2,\alpha}} 
\lesssim  
& \, \eeG \, 
\| \Hcal(g_1,h_1) - \Hstar \|_{L^2C_{5/2}^{\alpha}} 
+ \| \Hcal(g_1,h_1) - \Hstar \|_{C^{2,\alpha}_{\pstar + 2}} + \Big| \int_{\RR^3} ( \Hcal(g_1,h_1) - \Hstar ) \, dx \Big|
\\
& + \eeG \, \| \Mcal(g_1,h_1) - \Mstar \|_{L^2C_{5/2}^{1, \alpha}}
 + \eeG  \, \| \Mcal(g_1,h_1) - \Mstar \|_{C^{2,\alpha}_{3}}
 \\
 &  
 +  \eeG \hskip-.12cm \sum_{1 \leq j \leq 3} \Big| \int_{\RR^3} \big( \Mcal( g_1, h_1) - \Mstar \big)_j \, dx \Big|.
\endaligned
$$
This completes the proof of Theorems~\ref{theo:deuxieme} and \ref{theo:deuxieme-2}.

\section{The asymptotic localization method}
\label{sectio6}
 
\subsection{Preliminaries}

\paragraph{Heuristics.}
 
We now apply our method in order to solve a new localization problem motivated by a question raised by  Carlotto and Schoen in~\cite{CarlottoSchoen}. The main theorem established by the authors therein provides solutions to 
 Einstein's constraint equations, 
  which are {\sl localized at infinity} in the sense that their geometry is prescribed within an angular sector (to be, for instance, the Schwarzschild metric), 
while the solution in the remaining angular sector is identically Euclidean except for a (``small'') {\sl transition region.} In the latter transition region, the solutions constructed in~\cite{CarlottoSchoen} are controlled at 
only {\sl sub-harmonic} decay (namely $r^{-1+\eps}$ for $\eps \in (0,1)$). 
By analogy with Carlotto and Schoen's localization problem, we propose here 
the {\bf asymptotic localization problem}, as we call it, which consists of finding solutions whose {\sl asymptotic} behavior is arbitrarily prescribed at a {\sl super-harmonic} rate,
 {\sl except} within an (arbitrarily small) transition region where the solution has {\sl harmonic} decay $r^{-1}$ (or possibly decay faster). 
In order to achieve this result, our strategy is based on carefully {\sl designing a seed data set}
 adapted to the localization problem. 
 As a matter of fact, in order to achieve the desired harmonic decay 
 we find it quite natural ---from a physical standpoint, at least---
  to 
 relax the requirement originally proposed in~\cite{CarlottoSchoen}
  and allow for an {\sl asymptotic} localization only. 

Let us further explain our strategy and results in the present section. 

\bei 

\item Starting from a seed data $(g_1, h_1)$, we have proven in earlier sections of this paper that iterations based on a linearization of 
  the Einstein constraints
 can be performed and generate an actual solution $(g,h)$. This solution, however,
 contains an extra contribution at infinity, accounting for an ``error'' in the prescribed seed data which our algorithm ``propagates'' to 
infinity. 

\item More precisely, under the assumptions of Theorem \ref{theo:deuxieme} there exists an Einstein solution $(g, h)$ such that 
\bel{82:motivation}
g - g_1 +  \sum_{1 \leq a \leq n} \mt_a \chi_a r^2 \, \Hess_{e}(1/r) \in C^{2, \alpha}_{\pstar} (\Mbf), 
\ee
in which the constant $\mt_a$ is usually non-vanishing. Our purpose now is to further 
analyze the Hessian contribution $r^2 \, \Hess_{e} (1/r)$, 
and establish that it can be {\sl suppressed} in all asymptotic directions {\sl except}
 possibly within a small asymptotic angular region; cf.~Theorem~\ref{theorem:trois}, below. 
 
\item Our proof of~Theorem~\ref{theorem:trois} is based on designing a parametrized family of seed data $s \mapsto (g_1(s), h_1)$ by adding (in a neighborhood of infinity) the contribution
\bel{equa-611}
s \, 
\big( (1/2) \Deltas \Phi (x/r)- \Phi (x/r) \big) \, r^2 \, \Hess_{e} (1/r). 
\ee
Here, the ``shape function'' $\Phi$ is defined on the unit two-sphere $S^2$ and $\Deltas$ denotes the Laplace operator on $S^2$. 
We control the mass coefficient at infinity by suitably choosing the function $\Phi$ as well as the ``magnitude'' parameter $s$ in \eqref{equa-611}.  
More precisely, \eqref{eq-choice} below defines our parametrized family of seed data,
 and  
\eqref{83:minfty:1a} below gives the expression of the mass term $\mt(s)$. 
The parameter $s$ is determined by solving the algebraic equation \eqref{equa-our-equation}
and using the vanishing condition \eqref{equa-vanish}, below. 
 
\item  Next, in Section~\ref{sec:revisit}, we solve the asymptotic localization problem specifically for the choice of 
the Minkowski and the Schwarzschild metrics; cf.~Theorem~\ref{corollary-optimalCS}. Our Ansatz for the seed metric is now more involved ; see \eqref{eq:8.3:seeddata} which now depends upon two magnitude parameters denoted by $(s,t)$ and two shape 
functions $\Phi, \Psi$ defined on the sphere $S^2$. The mass coefficient $\mt(t,s)$ is now given by \eqref{83:equ:mSch} and 
finally we select the parameters $(s,t)$ in order to achieve the desired asymptotic localization.  
\eei 


\subsection{The asymptotic localization method for general seed metrics}

We now revisit the conclusion of Theorem \ref{theo:deuxieme} in the 
harmonic regime where the exponents $(p,q)=(1,2)$ are taken therein. Our theorem singles out a cone $\Cscr_a := \big\{ x \cdot \nu \geq |x| \, \cos a \text{ and } |x| \geq 1 \big\}$ in coordinates at infinity for some unit vector $\nu$ and some $a \in (0,2 \pi)$. In short, we prove that the Hessian term can be ``concentrated'' in an arbitrarily small cone.

\begin{theorem}[The asymptotic localization method for a general metric] 
\label{theorem:trois}  
As in Theorem~\ref{theo:deuxieme}, consider a seed data set $(g_1, h_1, \Hstar, \Mstar)$ defined 
on a background manifold $\Mbf = (\Mbf,e,r)$ and satisfying the strongly tame conditions \eqref{eq:strongtameH} for some decay exponents $(\ppG, \qqG,\ppM, \qqM)$ and H\"older exponent $\alpha \in (0,1)$. Consider\footnote{In the balanced regime $(\qqG, \qqM, \qstar) = (\ppG+1, \ppM+1, \pstar+1)$, we can take 
$1/2 \leq \ppG$ and 
$1 \leq \pstar \leq \min(\ppG+1,\ppM)$, together with 
$\pstar < \min(2,\ppG+1)$.}
 a pair of effective 
exponents $(\pstar, \qstar)$ for strongly tame data  in the sense \eqref{eq:strongexpo}.  
Consider an asymptotic cone $\Cscr_a$ with angle\footnote{arbitrarily small or large} $a \in (0, 2\pi)$.
Then, to the solution of Einstein's constraint equations 
  $(g, h)$ given by Theorem~\ref{theo:deuxieme}
one can associate another solution $(\gloc, \hloc)$, called an {\bf 
asymptotically localized solution}, which enjoys the same harmonic control as $(g,h)$:
\be
\aligned
(\gloc - g_1, \hloc - h_1)& \in C^{2,\alpha}_1 (\Mbf) \times C^{2,\alpha}_2 (\Mbf), 
\qquad \quad 
(g - g_1, h - h_1) \in C^{2,\alpha}_1 (\Mbf) \times C^{2,\alpha}_2 (\Mbf) 
\endaligned
\ee
and, in addition, $\gloc$ enjoys the following (super-harmonic) localization property outside the given cone: 
\be
\gloc - g_1\in C^{2,\alpha}_{\pstar} (\Cscr_a^c). 
\ee
Hence, $\gloc-g_1$ as a super-harmonic decay in all directions, except inside the cone $\Cscr_a$ within which 
it has harmonic decay only. 
\end{theorem}

\begin{proof} As in Section 5, we can reduce the problem to the case $(\Mbf,  e) = (\RR^3, g_\Eucl)$ 
and we decompose the proof in several steps.   
All of the operators $\Hess$ and $\Delta$ are computed with the Euclidian metric $e$ 
unless  specified differently.

\vskip.16cm

\noindent{\bf Step 1. A parametrized family of new seed data.} 

{\it 1.1 Defining the new seed data.}
Let $\Phi : S^2 \to \RR$ be a smooth function defined on the ``sphere at infinity'' and to be chosen at the end of our argument. Let  $\xi$ be a smooth cut-off function that $\xi$ is identically $0$ inside the unit ball $B(0,1) \subset \RR^3$ and $1$ outside $B(0,2)$. We denote by $\Delta = \Deltas_{S^2}$ the Laplace operator on the unit two-sphere.
For any $s \in \RR$, we set
\bel{eq-choice}
g_1(s) := g_1 + \ghat_1(s),
\qquad\qquad
\ghat_1(s) := s \, \xi_s \, \big({1 \over 2}\Deltas \Phi(x/r)  - \Phi(x/r) \big) \, r^2 \, \Hess(1/r),
\ee
in which $\xi: \RR^3 \to [0,1]$ with $\xi_s := \xi(x / \rho_s )$
and $\rho_s >0$ is a sufficiently large constant to be chosen later in the proof. Clearly, we have $|\ghat_1(s)| \lesssim 1/r$ so that this term enjoys harmonic decay. 


Some elementary identities involving the correction term $\Phi$ will be used in the following calculations, especially the observation
\bel{equa-65} 
\del_{ij} \ghat_1(s)_{ij} - \del_{jj} \ghat_1(s)_{ii} 
=
 s \, \Delta \Bigg( {r^2 \over 2} \Delta \Big( {\Phi \over r} \Big) \Bigg) -  s \, \Delta \Bigg( {\Phi \over r} \Bigg) 
 =: \Acal(s), 
 \qquad \text{ for all } r \geq \rho_s.
\ee
Namely, observe first that for any $C^1$-regular functions $\Psi : S^2 \to \RR$ and $f : \RR \to \RR$, the scalar product 
$\langle \nabla (\Psi(x/|x|)), \nabla (f(|x|) ) \rangle = 0$ vanishes in $\RR^3$. 
So, for all $r \geq 1$ we have 
$$ 
\aligned
\del_i \Big(- {1 \over 2} \Deltas \Phi + \Phi \Big) \del_j \Big( r^2 \del_{ij} \Big( {1 \over r} \Big) \Big) 
= \del_i \Big(- {1 \over 2} \Deltas \Phi + \Phi \Big)  \Bigg( {6 x_i x_j^2 \over r^5} - {2 \delta_{ij} x_j \over r^3} \Bigg) 
+ r^2 \del_i \Big(- {1 \over 2} \Deltas \Phi + \Phi \Big) \del_i \Big( \Delta \Big( {1 \over r} \Big) \Big)
= 0
\endaligned
$$
and, for second-order derivatives,  
$$
\aligned
& \del_{ij} \Big(- {1 \over 2} \Deltas \Phi + \Phi \Big) \del_{ij} \Big( {1 \over r} \Big) 
= \del_j \Big( \del_{i} \Big(- {1 \over 2} \Deltas \Phi + \Phi \Big) \del_{ij} \Big( {1 \over r} \Big) \Big) 
- \Big \langle  \nabla \Big(- {1 \over 2} \Deltas \Phi + \Phi \Big), \, \nabla \Big( \Delta \Big( {1 \over r} \Big) \Big) \Big \rangle
\\
& = - \del_j \Bigg( {\delta_{ij} \over r^3} \del_{i} \Big(- {1 \over 2} \Deltas \Phi + \Phi \Big) \Bigg) 
- \del_j \Big( x_j \Big \langle \nabla \Big(- {1 \over 2} \Deltas \Phi + \Phi \Big), \, \nabla \Big( {1 \over r^3} \Big)  \Big \rangle \Big)
\\
& = - \del_i \Big( {1 \over r^3} \del_{i} \Big(- {1 \over 2} \Deltas \Phi + \Phi \Big) \Big)
  = - {1 \over r^3} \Delta \Big(- {1 \over 2} \Deltas \Phi + \Phi \Big) 
- \Big \langle \nabla \Big( {1 \over r^3} \Big), \,  \nabla \Big(- {1 \over 2} \Deltas \Phi + \Phi \Big) \Big \rangle, 
\endaligned
$$
hence, using $\Delta (\Phi(x/r)) = r^{-2} \Deltas (\Phi(x/r))$ and $\nabla( \Phi(x/r) )\cdot \nabla (1/r)= 0$, 
$$
\aligned
& \del_{ij} \Big(- {1 \over 2} \Deltas \Phi + \Phi \Big) \del_{ij} \Big( {1 \over r} \Big) 
        = - {1 \over r^3} \Delta \Big(- {1 \over 2} \Deltas \Phi + \Phi \Big) 
         = {1 \over r^2} \Delta \Bigg( {r^2 \over 2} \Delta \Big( {\Phi \over r} \Big) \Bigg) 
-  {1 \over r^2}  \Delta \Big( {\Phi \over r} \Big)
= {1 \over s r^2} \Acal(s). 
\endaligned
$$
Observing also that the trace term $\del_{ij}\Big( r^2 \del_{ij} \Big( {1 \over r} \Big) \Big) = 0$ vanishes for all $r \geq 1$, 
and taking the factor $r^2$ into account, this leads us to \eqref{equa-65} .

 \vs

{\it 1.2 Einstein constraints for the new seed data.} 
In view of \eqref{equa-Rscal}, the scalar curvature $R_{g_1(s)}$ of $g_1(s)$ reads 
(recalling that $g_1(s)$ is the modified data and $g_1=g_1(0)$ is the ``fixed' data)
$$ 
\aligned
R_{g_1(s)} - R_{g_1}
& =  \big( \textstyle \sum_{i,j}\del_{ij} \ghat_1(s)_{ij} \big) - \Delta (\Tr_e \ghat_1(s))  - \Theta_0(\ghat_1(s)),
\\
 - \Theta_0(\ghat_1(s)) 
& := \ghat_1(s) * \del^2 g_1 
 + \ghat_1(s)* \del g_1 * \del g_1 + \del g_1 * \del \ghat_1(s) 
 \\
 & \quad 
 +  (g_1 -e) * \del^2 \ghat_1(s) 
 + \ghat_1(s)* \del^2 \ghat_1(s)  + \del \ghat_1(s) * \del \ghat_1(s), 
 \endaligned
 $$
so that the Hamiltonian takes the form (the numbering of the contributions $\Theta$
being chosen by analogy with the decompositions in Section~5): 
\bel{82:Hcals}
\aligned
(\Hcal(g_1(s), h_1) - \Hstar) 
&= \big( \Hcal(g_1, h_1) - \Hstar \big) +  \del_{ij} \ghat_1(s)_{ij} - \del_{jj} \ghat_1(s)_{ii} 
+ \Theta_{02}(s)
\\
&  
= \big( \Hcal(g_1, h_1) - \Hstar \big) 
+  \Acal(s)  + \Theta_{10}(s) + \Theta_{02}(s),
\endaligned
\ee
in which we make explicit the contribution $\Acal(s)$ from \eqref{equa-65} with 
 $\Theta_2(\ghat_1(s)) := \ghat_1(s)*h_1*h_1$ and 
$$ 
\aligned  
\Theta_{10}(s)
 := \del_{ij} \ghat_1(s)_{ij} - \del_{jj} \ghat_1(s)_{ii} - \Acal(s), 
 \qquad\quad
\Theta_{02}(s) := \Theta_0(\ghat_1(s)) + \Theta_2(\ghat_1(s)). 
\endaligned
$$
 We follow closely the notation from Section~5 (but there is no analogue of $\smallcal(h_2)$ here since we do not modify the extrinsic curvature). 
Thanks to \eqref{equa-65}, we have $\Theta_{10}(s) = 0$ identically 
for all (sufficiently large) radius $r \geq \rho_s$. 


On the other hand, writing the divergence   
$\Div_{g_1(s)} h_1 = \Div_{g_1} h_1 + h_1 * \del \ghat_1(s) +  \ghat_1(s) * \del g_1 * h_1$
of the prescribed extrinsic curvature
we write the Einstein momentum as ($8$ being chosen again by analogy with Section~5)
\bel{82:Mcalt} 
(\Mcal(g_1(s), h_1) - \Mstar)
 = (\Mcal(g_1, h_1) - \Mstar) + \Theta_8(s), 
 \qquad
 \Theta_8(s) := h_1 * \del \ghat_1(s) + 
 \ghat_1(s) * \del g_1 * h_1. 
\ee
Using the above decompositions, it is not difficult to check that, provided we choose 
$\rho_s \gtrsim ( (1 + |s|) /  \min(\eeM, \eeG) )^2$, 
the new seed data $(g_1(s), h_1)$ is well-defined and enjoys the conditions in Theorem~\ref{theo-prec} in which we choose $(p, q) = (1/2, 3/2)$.

In turn, the new solution $(g (s), h(s))$ generated by $(g_1(s), h_1)$ is well-defined and, in view of \eqref{eq2:hua} with $(p, q) = (1/2, 3/2)$ and $(g_2, h_2) (s) := (g- g_1, h - h_1)(s)$, we have 
\begin{subequations}
\label{eq:8.3:hub-all} 
\begin{align}
{g_2 (s) \over r^2} 
& = - \Big(\Delta_{g_1(s)}u_2 (s)\Big) \, g_1(s) + \Hess_{g_1(s)}(u_2 (s)) - u_2 (s)\Ric_{g_1(s)}   
 + h_1 *h_1 *u_2 (s) + \nabla_{g_1(s)}\big(h_1 *Z_2 (s)\big),
\label{eq:8.3:hua}
\\
h_2 (s) & = - {1 \over 2}  \Lcal_{Z_2 (s)}g_1(s) + h_1 *u_2 (s), 
\label{eq:8.3:hub}
\end{align}
\end{subequations}
where the dual unknowns are denoted by $(u_2 (s), Z_2 (s)) \in C^{4,\alpha}_{1/2} \times  C^{3,\alpha}_{1/2}$. 
These basic decay rates are now going to be improved by following the method in Section~5 and eventually choosing the parameter $s$. 


\vskip.3cm

\noindent{\bf Step 2. Super-harmonic estimates.}
We proceed as  
 in the proof of Theorem~\ref{theo:deuxieme}, and we emphasize the dependency in $s$. Terms like $\smallcal(s)$ with various indices corresponding to terms with favorable decay of the kind already analyzed in Section 5. 
 For the metric variable we have (using that $h_1(s)=h_1$) 
\begin{subequations}
\label{eq:8.3.general}
\bel{eq:8.3:g2 hess}
\aligned
 -  \Delta (\Tr_e g_2(s)) +   \textstyle  \sum_{i, j} \del_{ij} g_{2ij}(s) 
& = \big( \Hstar - \Hcal( g_1(s), h_1) \big) + \smallcal_{012}(s), 
\\
 - 2\Delta u_2 (s)
& =
{1\over r^2} \Tr_e g_2(s) + R_{g_1(s)} u_2 (s) + \smallcal_4(s),
\qquad 
\del_{ij}\Bigg( {|x|^2 \over r^2} g_{2ij}(s) \Bigg)  
 = -4 \, \Delta u_2 + \smallcal_3(s). 
\endaligned
\ee
On the other hand, for the variable $Z_2$, we can also derive 
\be
\LMcal(Z_2)  
 = 2 \, \big( \Mstar - \Mcal(g_1, h_1) \big) + \smallcal_{67}(s). 
\ee

\end{subequations}

As before, combining these equations we obtain (for all $r \geq 1$) 
\bel{eq:8.3:Tr}
\aligned
&- \Delta \big( \Tr_e g_2 (s) \big) + {2 \over r^2} \Tr_e g_2 (s) 
 := \big( \Hstar - \Hcal( g_1(s), h_1) \big) + \smallcal_5(s).
\endaligned
\ee
with $
 \smallcal_5(s) := \smallcal_{012}(s)
 - \smallcal_3(s) - 2 \, \smallcal_4(s)
 - \del_{ij} \big( ( 1 - |x|^2/r^2) \, g_{2ij}(s) \big)
 - 2 \, R_{g_1(s)} u_2(s)$.

We can rewrite the equation for $\Delta u_2 (s)$ in the form 
\bel{eq:8.3:uh0b}
\aligned
- 2 \Delta \Bigg(u_2 (s) + {s \, \Phi \over 4r} \Bigg) 
& = {1 \over r^2} \Bigg( \Tr_e g_2 (s) -  {s \, r^2 \over 2} \Delta \Big( {\Phi \over r} \Big) \Bigg) 
+  R_{g_1(s)} u_2 (s) + \smallcal_4(s),
\endaligned
\ee
and, in view of the expression \eqref{82:Hcals} of the Hamiltonian for the seed data, 
\bel{eq:8.3:uh0}
\aligned
&- \Delta \Bigg( \Tr g_2 (s) - {s \, r^2 \over 2} \Delta \Big( {\Phi \over r} \Big) \Bigg) 
+ {2 \over r^2} \Bigg( \Tr g_2 (s) - {s \, r^2 \over 2} \Delta \Big( {\Phi \over r} \Big) \Bigg) 
   = 
  \big( \Hstar  - \Hcal(g_1, h_1) \big)
+ \smallcal_5(s)
 + \Theta_{10}(s) + \Theta_{02}(s). 
\endaligned
\ee

Therefore, an analysis similar to
the one in the proof of Theorem \ref{theo:deuxieme} leads to the following result. 

\begin{claim}[Harmonic estimates]
 The metric and the dual variables satisfy 
$$
{1 \over r^2} \Bigg( \Tr g_2 (s) - {s \, r^2 \over 2} \Delta \Big( {\Phi \over r} \Big) \Bigg) \in L^1 \cap C^{2, \alpha}_1,
\qquad \quad
u_2 (s) + {s \, \Phi \over 4r} \in C^{4, \alpha}_1,
\quad\qquad Z_2 (s) \in C^{3, \alpha}_1. 
$$
\end{claim}

Using this result in \eqref{eq:8.3:hub-all}, we obtain $(g_2 (s), h_2 (s)) \in C^{2, \alpha}_{1} \times C^{2, \alpha}_{2}$
and, therefore, thanks again to \eqref{eq:8.3:uh0b}-\eqref{eq:8.3:uh0} 
\begin{align}
- \Delta \Bigg( \Tr g_2 (s) - {s \, r^2 \over 2} \Delta \Big( {\Phi \over r} \Big) \Bigg) 
+ {2 \over r^2} \Bigg( \Tr g_2 (s) - {s \, r^2 \over 2} \Delta \Big( {\Phi \over r} \Big) \Bigg)
& \in L^1 \cap C^{0,\alpha}_{\pstar + 2}, 
\label{eq:8.3:uh1}
\\
- 2 \Delta \Bigg(u_2 (s) + {s \, \Phi \over 4r} \Bigg) - {1 \over r^2} \Bigg( \Tr g_2 (s) - {s \, r^2 \over 2} \Delta \Big( {\Phi \over r} \Big) \Bigg) 
& \in L^1\cap C^{2,\alpha}_{\pstar + 2}. 
\label{eq:8.3:uh2}
\end{align}
By applying Proposition \ref{laplaceplus} to \eqref{eq:8.3:uh1} and for $\pstar$ given in the statement of the theorem, we have  
\bel{eq:8,3:fastdecay0}
\Bigg( \Tr g_2 (s) - {s \, r^2 \over 2} \Delta \Big( {\Phi \over r} \Big) \Bigg) \in C^{2, \alpha}_{\pstar}, 
\qquad \quad
{1 \over r^2} \Bigg( \Tr g_2 (s) - {s \, r^2 \over 2} \Delta \Big( {\Phi \over r} \Big) \Bigg) \in L^1. 
\ee
Therefore, by applying Corollary \ref{remarkpoisson} to \eqref{eq:8.3:uh2}, there exists a constant $\mt(s)$ such that
$u_2 (s) + {s \, \Phi \over 4 r } - {\mt(s) \over r } \in C^{4, \alpha}_{\pstar}$.
In the same way as 
 in the proof of Theorem~\ref{theo:deuxieme}, we also have
$Z_2 (s) + 2 V \cdot \Pt(s) \in C^{3, \alpha}_{\qstar - 1}$, 
with the notation $V \cdot \Pt (s)$ introduced earlier with $V_{ij} := \Big( x_i x_j + 3 r^2 \delta_{ij} \Big)/ r^3$. 
Taking the last two results into account in \eqref{eq:8.3:hub-all}, we arrive at the following conclusion for the new  solution. 

\begin{claim}[Super-harmonic estimates]
 The metric and dual metric variables of the new solution enjoy
\bel{eq:8.3:main1} 
g_2 (s) - s \, r^2 \Delta \Big({\Phi \over 4 r}\Big) \, \delta 
+ s \, r^2 \, \Hess \Big( {\Phi \over 4 r} \Big) 
- \mt(s) \, r^2 \, \Hess \Big( {1 \over r}\Big) \in C^{2, \alpha}_{\pstar}, 
\qquad \quad
h_2 (s) - \Lcal_{V \cdot \Pt (s)} \delta \in C^{2, \alpha}_{\qstar}.
\ee
\end{claim}


\vskip.16cm

\noindent{\bf Step 3. Estimating the mass corrector.} We can apply similar arguments to the ones in the proof of Theorem \ref{theo:deuxieme} and control the term $\mt(s)$. By applying Proposition \ref{proposition behavior at infinity} to \eqref{eq:8.3:uh0b} 
we have
\be
\mt(s) = - {1 \over 4 \pi} \int_{\RR^3} \Delta \Big( u_2 (s) + {s \, \Phi \over 4 r } \Big) \, dx
\ee
and, consequently with the third equation in \eqref{eq:8.3:g2 hess},
\bel{83:minfty:1a}
\mt(s) = {1 \over 16 \pi} \int_{\RR^3} \Bigg(  \del_{ij} g_2 (s)_{ij} - s \, \Delta \Big( {\Phi \over r} \Big)
+ \smallcal_3(s) \Bigg)  \, dx.
\ee

On the other hand, after analyzing the decay with the same arguments as the ones leading to \eqref{81:forremark}  
we can check the formula 
$$ 
\int_{\RR^3} \Delta \Bigg( \Tr g_2 (s) - {s \, r^2 \over 2} \Delta \Big( {\Phi \over r} \Big) \Bigg) \, dx = 0
$$
and consequently, thanks to the first equation in \eqref{eq:8.3:g2 hess}, 
$$
\aligned
& \int_{\RR^3} \Bigg(  \del_{ij} g_2 (s)_{ij} - s \, \Delta \Big( {r^2 \over 2} \Delta \Big( {\Phi \over r} \Big) \Big)  \Bigg) \, dx  
= 
\int_{\RR^3} \big( \Hstar - \Hcal(g_1(s),h_1) \big) \, dx
+ \int_{\RR^3} \smallcal_{012}(s) \, dx.
\endaligned
$$
Taking also \eqref{82:Hcals} into account or equivalently 
$$
 \big( \Hcal(g_1(s),h_1) - \Hstar\big)
 =
\big( \Hcal(g_1, h_1) - \Hstar \big) 
+  s \, \Delta \Bigg( {r^2 \over 2} \Delta \Big( {\Phi \over r} \Big) \Bigg) -  s \, \Delta \Bigg( {\Phi \over r} \Bigg) 
 + \Theta_{10}(s) + \Theta_{02}(s),
$$
we can thus rewrite \eqref{83:minfty:1a} as  
\bel{82:minfty.b1}
\aligned
& 16 \pi \, \mt(s) 
 = \int_{\RR^3} \big(\Hstar - \Hcal(g_1, h_1) \big) \, dx
 + \int_{\RR^3} \Big( - \Theta_{10}(s) - \Theta_{02}(s) + \smallcal_{012}(s) \Big) \, dx. 
\endaligned
\ee


Next, by integration by parts we can rewrite the contribution $\Theta_{10}$ as  
$$
\int_{\RR^3} \Theta_{10}(s) \, dx
= \lim_{r \to +\infty} \int_{S_r}  \Bigg( 
\del_i \ghat_1(s)_{ij} - \del_j \ghat_1(s)_{ii} - s \, \del_j \Big( {r^2 \over 2} \Delta \Big( {\Phi \over r} \Big) \Big) 
+  s \, \del_j \Big( {\Phi \over r} \Big) \Bigg) \, 
{x_j \over r} \, d\omega,
$$
and 
a straightforward calculation gives us ($r \geq 2 \rho_s$)
\begin{subequations}
\label{82:Phiinteg}
\be
\aligned
\int_{S_r} \del_j \ghat_1(s)_{ii} {x_j \over r}\, d\omega 
& = s \int_{S_r} \del_j \Bigg( \Big( {1 \over 2} \Deltas \Phi - \Phi \Big) \, r^2 \Delta \Big(1/r \Big) \Bigg) \, {x_j \over r} \, d\omega
   = 0,
\endaligned
\ee
and 
$$
\aligned
\int_{S_r} \del_i \ghat_1(s)_{ij} {x_j \over r}\, d\omega 
& = s \int_{S_r} 
\Bigg( 
\Big( {1 \over 2} \Deltas \Phi - \Phi \Big) \Bigg( r \del_j \Big( \Delta \Big( {1 \over r} \Big) \Big) x_j  + {6x_i^2 x_j^2 \over r^6} - {2x_i^2 \over r^4} \Bigg)
+ \del_i \Big( {1 \over 2} \Deltas \Phi - \Phi \Big) \Bigg( {3x_i x_j^2 \over r^4} - {\delta_{ij} x_j \over r^2} \Bigg) \Bigg) \, d\omega
\\
& = 4 s \int_{S^2} \Big( {1 \over 2} \Deltas \Phi - \Phi \Big) \, d\omega, 
\endaligned
$$
thus
\be
\aligned
\int_{S_r} \del_i \ghat_1(s)_{ij} {x_j \over r}\, d\omega 
&  
   = - 4 s \int_{S^2} \Phi \, d\omega,
\endaligned
\ee
as well as 
\be
\aligned
\int_{S_r}  \Bigg( \del_j \Big( {r^2 \over 2} \Delta \Big( {\Phi \over r} \Big) \Big) 
-  \del_j \Big( {\Phi \over r} \Big) \Bigg) \, {x_j \over r} \, d\omega 
& = \int_{S_r} \Bigg( \del_j (\Deltas \Phi) {x_j \over 2r^2} \, d\omega - \Deltas \Phi {x_j^2 \over 2 r^4} - \del_j \Phi {x_j \over r^2} + \Phi {x_j^2 \over r^4} \Bigg) \, d\omega
\\
& = - \int_{S^2} \Big({1 \over 2} \Deltas \Phi - \Phi \Big) \, d\omega
  = \int_{S^2} \Phi \, d\omega. 
\endaligned
\ee
\end{subequations}
Therefore, it follows that
\be
\int_{\RR^3} \Theta_{10}(s) \, dx
= - 5 \, s \int_{S^2} \Phi \, d\omega. 
\ee
Using this result in \eqref{82:minfty.b1}, we arrive at the following {\sl formula for the mass corrector} of the solution associated with our choice of parametrized seed data: 
\bel{eq:82:m_identity}
\aligned
16 \pi \, \mt(s) 
 & = 
 \int_{\RR^3} \big( \Hstar - \Hcal(g_1, h_1)  \big) \, dx 
 + 5 \, s \int_{S^2} \Phi \, d\omega + \int_{\RR^3} \Big( - \Theta_{02}(s) + \smallcal_{012}(s) \Big) \, dx.
\endaligned
\ee
Finally, proceeding as in Section~5 for the proof of Theorem \ref{theo:deuxieme} (see \eqref{82:mexact:a})
and observe that all the terms in $ - \Theta_{02}(s) + \smallcal_{012}(s)$ vanish when $s=0$ and are proportional to $\eeG$, 
 we control the mass corrector as follows: 
\bel{82:minfty.c}
\aligned
& \Big| 16 \pi \, \mt(s) -  \int_{\RR^3} \big(\Hstar - \Hcal(g_1, h_1) \big) \, dx
- 5 \, s  
 \int_{S^2}  \Phi \, d\omega \Big|
 \\
 &
 \lesssim \eeG \, \| \Hcal(g_1,h_1) - \Hstar \|_{L^2C_{5/2}^{\alpha}}
 + \eeG \,  \| \Mcal(g_1,h_1) - \Mstar \|_{L^2C_{5/2}^{1, \alpha}}
 + \eeG s, 
\endaligned
\ee
in which we are going to choose $\int_{S^2}  \Phi \, d\omega$ to vanish (cf.~next step). 


\vskip.3cm

\noindent{\bf Step 4. Continuity of the mass corrector.} We now study the function  
$F : I \to \RR$, $s \mapsto {4 \over 5} \mt(s)$ determined in previous steps. We consider this function in the 
range 
\bel{equa-interv} 
I := [ - C_1 \, \beta - \eps, C_1 \, \beta + \eps], 
\qquad\quad
\beta := \|\Hcal(g_1, h_1) - \Hstar \|_{L^1} \leq C_0 \eeG, 
\ee
for some (large) constant $C_1>0$ and (small) constant $\eps>0$. The constant $C_0>0$ is fixed in our argument and depends upon the data $(g_1,h_1)$. 

We claim that $F$ is continuous. We thus fix some point 
 $s \in I$, and we consider a sequence $\{s_i\} \subset I$ converging to~$s$. In view of \eqref{eq:8.3:hub-all} we have
\bse
\bel{eq:8.3b:hua}
\aligned
g_2 (s) - g_2(s_i) 
& = - r^2 \Big( \big(\Delta_{g_1(s)}u_2 (s)\big)g_1(s) - \big(\Delta_{g_1(s_i)}u_2 (s_i)\big)g_1(s_i)\Big) 
\\
& \quad + r^2 \Big( \Hess_{g_1(s)}(u_2 (s)) - \Hess_{g_1(s_i)}(u_2 (s_i)) \Big) - r^2 \big( u_2 (s)\Ric_{g_1(s)} - u_2 (s_i)\Ric_{g_1(s_i)} \big) 
\\
& \quad + r^2 \big( h_1 *h_1 * (u_2 (s) - u_2 (s_i)) \big) + r^2 \big( \nabla_{g_1(s)}(h_1 *Z_2 (s)) - \nabla_{g_1(s_i)}(h_1 *Z_2 (s_i)) \big) 
\endaligned
\ee
and
\bel{eq:8.3b:hub}
h_2 (s) - h_2(s_i) = - {1 \over 2}  \big( \Lcal_{Z_2 (s)}g_1(s) - \Lcal_{Z_2 (s_i)}g_1(s_i) \big) + h_1 * \big(u_2 (s) - u_2 (s_i) \big),
\ee
\ese
where $(g_2, h_2, u_2, Z_2)(s)$ and $(g_2, h_2, u_2, Z_2)(s_i)$ belong to 
$C^{2,\alpha}_1 \times C^{2,\alpha}_2 \times C^{4, \alpha}_1 \times C^{3,\alpha}_3$. 
For instance, for any $\theta \in (0,1)$ we have 
$$
\aligned
& \big\| \big(\Delta_{g_1(s)}u_2 (s)\big)g_1(s) - \big(\Delta_{g_1(s_i)}u_2 (s_i)\big)g_1(s_i)  \big\|_{C^{2,\alpha}_{\theta + 2}}
\\
& \lesssim \big\| \Delta_{g_1(s)}u_2 (s)  - \Delta_{g_1(s_i)}u_2 (s)\big) g_1(s) \big\|_{C^{2,\alpha}_{\theta + 2}}
+ \big| \big( \Delta_{g_1(s_i)} ( u_2 (s) - u_2(s_i) ) \big) g_1(s) \big\|_{C^{2,\alpha}_{\theta + 2}} + \big| \big (\Delta_{g(s_i)} u_2 (s_i) \big) (g_1(s) - g_2(s_i) \big\|_{C^{2,\alpha}_{\theta + 2}}
\\
& \lesssim \| u_2(s) - u_2(s_i) \|_{C^{4, \alpha}_{\theta}} + \Ecal(s - s_i),
\endaligned
$$
in which $\Ecal=\Ecal(t)$ stands for some function satisfying
$\lim_{t \to 0} \Ecal (t) = 0$. 
So it follows from \eqref{eq:8.3b:hua} that
\bse
\label{eq:8.3b2:hua-all}
\bel{eq:8.3b2:hua}
\| g_2(s) - g_2(s_i) \|_{C^{2,\alpha}_{\theta}} \lesssim \| u_2(s) - u_2(s_i) \|_{C^{4, \alpha}_\theta} + \eeG \, \| Z_2(s) - Z_2(s_i) \|_{C^{3, \alpha}_\theta} + \Ecal(s - s_i). 
\ee
Similarly, we also obtain by \eqref{eq:8.3b:hub}
\bel{eq:8.3b2:hub}
\| h_2(s) - h_2(s_i) \|_{C^{2,\alpha}_{\theta + 1}} \lesssim \eeG \, \| u_2(s) - u_2(s_i) \|_{C^{4, \alpha}_\theta} + \| Z_2(s) - Z_2(s_i) \|_{C^{3, \alpha}_\theta} + \Ecal(s - s_i). 
\ee
\ese
On the other hand, recalling \eqref{eq:8.3.general}
we find 
$$ 
\aligned
- \Delta \big( \Tr g_2 (s) - \Tr g_2 (s_i ) \big) + \big(  \del_{ij} g_2 (s)_{ij} -  \del_{ij} g_2 (s_i )_{ij} \big) 
& = \big( \Hcal( g_1(s_i), h_1) - \Hcal( g_1(s), h_1)\big) 
 + \smallcal_{012}(s) - \smallcal_{012}(s_i),
\\
- \Delta ( \Tr g_2 (s) - \Tr g_2 (s_i ) ) + {2 \over r^2} \big( \Tr g_2 (s) - \Tr g_2 (s_i ) \big) 
& = \big( \Hcal( g_1(s_i), h_1) - \Hcal( g_1(s), h_1) \big) 
  + \smallcal_{67}(s) - \smallcal_{67}(s_i)
\endaligned
$$
and
$$
\aligned
- 2\Delta (u_2 (s) - u_2 (s_i ) ) 
& = {1\over r^2} \big( \Tr_e g_2 (s) - \Tr_e g_2 (s_i ) \big) 
+ R_{g_1(s)} u_2 (s) - R_{g_1(s_i)} u_2 (s_i)
+ \smallcal_4(s) -  \smallcal_4(s_i),
\\
-4 \, \Delta ( u_2 (s) - u_2 (s_i ) ) & = \del_{ij}\Bigg( {|x|^2 \over r^2} ( g_2 (s)_{ij} - g_2 (s_i )_{ij} ) \Bigg)
-\smallcal_3(s) + \smallcal_3(s_i),
\\
\Delta Z_2 (s) + \Div(\nabla Z_2 (s)) 
& = \big( \Mcal(g_1(s_i), h_1) - \Mcal(g_1(s), h_1) \big)
+  \smallcal_{67}(s) -  \smallcal_{67}(s_i).
\endaligned
$$
An analysis similar to the one 
in the proof of Theorem \ref{theo:deuxieme}, but using \eqref{eq:8.3b2:hua-all} instead of \eqref{mind}, shows that
$$
\aligned
\| u(s) - u(s_i)\|_{C^{4,\alpha}_{\theta}} & \lesssim \| \Hcal(g_1(s), h_1) -  \Hcal(g_1(s_i), h_1) \|_{C^{2,\alpha}_{\theta + 2}} + \eeG \, \| \Mcal(g_1(s), h_1) -  \Mcal(g_1(s_i), h_1) \|_{C^{2,\alpha}_{\theta + 2}}
\\
& \quad + \eeG \, \big(\| g_2(s) - g_2(s_i) \|_{C^{2,\alpha}_{\theta}} + \| h_2(s) - h_2(s_i) \|_{C^{2,\alpha}_{\theta + 1}} \big) + \Ecal(s - s_i),
\\
\| Z(s) - Z(s_i)\|_{C^{3,\alpha}_{\theta}} & \lesssim \eeG \, \| \Hcal(g_1(s), h_1) -  \Hcal(g_1(s_i), h_1) \|_{C^{2,\alpha}_{\theta + 2}} + \| \Mcal(g_1(s), h_1) -  \Mcal(g_1(s_i), h_1) \|_{C^{2,\alpha}_{\theta + 2}}
\\
& \quad + \eeG \, \big(\| g_2(s) - g_2(s_i) \|_{C^{2,\alpha}_{\theta}} + \| h_2(s) - h_2(s_i) \|_{C^{2,\alpha}_{\theta + 1}} \big) + \Ecal(s - s_i). 
\endaligned
$$
Combining this result with the bounds \eqref{eq:8.3b2:hua-all} 
we obtain 
\bel{mind:82}
\aligned
\| g_2(s) - g_2(s_i) \|_{C^{2,\alpha}_{\theta}} \lesssim \| \Hcal(g_1(s), h_1) -  \Hcal(g_1(s_i), h_1) \|_{C^{2,\alpha}_{\theta + 2}} + \eeG \, \| \Mcal(g_1(s), h_1) -  \Mcal(g_1(s_i), h_1) \|_{C^{2,\alpha}_{\theta + 2}} + \Ecal(s -s_i),
\\
\| h_2(s) - h_2(s_i) \|_{C^{2,\alpha}_{\theta + 1}} \lesssim \eeG \, \| \Hcal(g_1(s), h_1) -  \Hcal(g_1(s_i), h_1) \|_{C^{2,\alpha}_{\theta + 2}} + \| \Mcal(g_1(s), h_1) -  \Mcal(g_1(s_i), h_1) \|_{C^{2,\alpha}_{\theta + 2}} + \Ecal(s -s_i). 
\endaligned
\ee
In view of \eqref{eq:82:m_identity} we have 
$$
{5 \over 4 } F(s) = 
16 \pi \, \mt(s) 
  = - \int_{\RR^3} \big( \Hcal(g_1, h_1) - \Hstar \big) \, dx 
 + 5 \, s \int_{S^2} \Phi \, d\omega + \int_{\RR^3} \Big( - \Theta_{02}(s) + \smallcal_{222}(s) \Big) \, dx
$$
and, thanks to \eqref{mind:82}, 
$$
\aligned
| F(s) - F(s_i)| 
& \lesssim \| \Hcal(g_1(s), h_1) -  \Hcal(g_1(s_i), h_1) \|_{C^{2,\alpha}_{\theta + 2}} + \| \Mcal(g_1(s), h_1) -  \Mcal(g_1(s_i), h_1) \|_{C^{1,\alpha}_{\theta + 2}}
\\
& \quad
 + |s - s_i| \, \Big| \int_{S^2} \Phi \, d\omega \Big| + \Ecal(s -s_i).
\endaligned 
$$
Since 
$$
\lim_{s_i \to s}
 \big(\| \Hcal(g_1(s), h_1) -  \Hcal(g_1(s_i), h_1) \|_{C^{2,\alpha}_{\theta + 2}} + \| \Mcal (g_1(s), h_1) - \Mcal (g_1(s_i), h_1) \|_{C^{1,\alpha}_{\theta + 2}} \big) = 0,
$$
we have the convergence $F(s_i) \to F(s)$ for any given $s \in I$ and any sequence converging to $s$. 
Therefore, $F$ is continuous, as claimed.



\vskip.3cm

\noindent{\bf Step 5. Hessian term outside the cone.} 

It remains to choose the parameter $s$ in the range \eqref{equa-interv}, namely  $|s - C_1 \, \beta | \leq \eps$. 
In view of \eqref{82:minfty.c} we have 
$$
\aligned
|F(s)|
& = {4 \over 5}\,  |\mt(s)| 
\leq {1 \over 20 \pi} \Big| \int_{\RR^3} \big(\Hstar - \Hcal(g_1, h_1) \big) \, dx +
 5 \, s \int_{S^2} \Phi \, d\omega \Big| + C_3 \eeG (\eeM + s)
 \\
 & \leq  {1 \over 20 \pi} \beta + C_3 \eeG \big( \eeM + C_1 \, \beta + \eps \big)
 < C_1 \, \beta + \eps 
 \endaligned
$$
for a suitable choice of constants (namely $C_1$ large and $\eeG$ small, while $C_3$ is given by our estimates), 
provided we choose $\Phi$ such that
\bel{equa-vanish} 
\int_{S^2} \Phi \, d\omega = 0. 
\ee
Hence, 
we can guarantee that $|F(s)| < C_1 \, \beta + \eps$, 
namely $F(I) \subset I$ and,  by the Intermediate Value Theorem, 
there exists $s^* \in I$ such that 
\bel{equa-our-equation}
{4 \over 5} \mt (s^*) = F(s^*) = s^*.
\ee

In turn, thanks to the decay property of the metric \eqref{eq:8.3:main1} and by focusing to the exterior of the cone $\Cscr_a$ we have arrived at our main conclusion
$$
g_2 (s^*) - s^* r^2 \Delta \Big({\Phi \over 4 r}\Big) \, \delta + s^* r^2 \, \Hess \Big( {\Phi \over 4 r} \Big) - {5 \, s^* \over 4} r^2 \, \Hess \Big( {1 \over r}\Big) \in C^{2, \alpha}_{\pstar} (\Cscr_a^c). 
$$
It now remains to choose 
\be
\Phi(x/r) = 1 \qquad \text{ for all } \, x \in \Cscr_a^c, 
\ee
so that the above property reduces to saying 
$g_2(s^*) -  s^* r^2 \, \Hess(1/r) \in C^{2,\alpha}_{\pstar} (\Cscr_a^c)$. 
We have completed the proof by taking $\gloc=g(s^*)$ since, in view of the choice of $g_1(s)$ in \eqref{eq-choice} and for sufficiently large $r$, 
\begin{align*}
g(s^*) - g_1 
= \ghat_1(s^*) + g_2(s^*)
& =  s^* \, \big({1 \over 2}\Deltas \Phi(x/r)  - \Phi(x,r) \big) \, r^2 \, \Hess(1/r)
  + g_2(s^*)
\\
 \hskip-3.cm  & =  - s^* \, r^2 \, \Hess(1/r) + g_2(s^*)
 \in C^{2,\alpha}_{\pstar} (\Cscr_a^c). 
\qedhere
\end{align*}
\end{proof}


\subsection{The asymptotic localization problem for the Schwarzschild and Minkowski metrics} 
\label{sec:revisit}

Another application of our method is now presented.

\begin{theorem}[The asymptotic localization problem for the Schwarzschild and Minkowski metrics]
\label{corollary-optimalCS}
Consider solutions to the Einstein's vacuum constraint equations 
  with topology $M\simeq \RR^3$, and let 
$g_\Sch$ be the Schwarzschild solution with small mass $m>0$ (defined outside a large compact region)
and let $g_\Eucl = \delta$ be the Euclidean metric. 
Decompose asymptotic infinity in $\RR^3$ into three angular regions denoted by $\Cscr_a \cup \Cscr^c_{a+\eps} \cup \Tscr_{a,\eps}$, where $\Cscr_a$ is a cone with (possibly arbitrarily small) angle $a \in (0, 2\pi)$, 
$\Cscr^c_{a+\eps}$ is the complement of the same cone but with a (possibly only slightly) larger angle, 
and $\Tcal_{a,\eps}$ is the remaining transition region.  
Then, there exists a choice of seed data set such that the seed-to-solution method generates a solution to the Einstein constraints,
   denoted by $(g,h)$ and defined on $\RR^3$, having (at least) 
harmonic decay $1/r$ in every angular  direction, that is, 
\be
(g - g_\Eucl , h) \in C^{2,\alpha}_1 (\Mbf) \times C^{2,\alpha}_1(\Mbf)
\ee
and being asymptotic to the Schwarzschild metric in the cone $\Cscr_a$ 
and to the Euclidean metric in the cone $\Cscr^c_{a+\eps}$. Specifically, for some $\theta \in (1,2)$ one has 
\bel{eq:71a}
\aligned
g - g_\Sch  & \in C^2_\theta (\Cscr_a),
\qquad 
g - g_\Eucl  \in C^2_\theta (\Cscr^c_{a + \eps}). 
\endaligned
\ee   
\end{theorem}  

We now give a proof of this theorem. 

\paragraph{Step 1. A parametrized family of seed data.}
Let $\Psi : S^2 \to \RR$ be a  given $C^2$-regular function such that 
\bel{equa-choice-Psi} 
\Psi(x/r) = 1 \, \text{ in } \, \Cscr_a,
\qquad \quad 
 \Psi(x/r) = 0 \, \text{ in } \, \Cscr^c_{a + \eps}.
 \ee 
Let $\xi$ be a cut-off function such that $\xi$ is identically $0$ inside the unit ball $B(0,1)$ and $1$ outside the ball $B(0,2)$. Consider the seed data defined by $h_1=h_1(t,s) = 0$ and  
\bel{eq:8.3:seeddata} 
\aligned
g_1(s,t) & 
: = 
\delta
+ \Big(- {1 \over 2} \Deltas \Psi + \Psi \Big)(x/r) \, \big( g_\Sch  - \delta \big) + \xi_{R_{s,t}} \, \mu \, r^2 \, \Hess \big( 1/r \big) 
\\
&
\mu := t \, \Big( {1 \over 2} \Deltas \Psi - \Psi \Big)(x/r)     
 + s \,  \Big( {1 \over 2} \Deltas \Phi - \Phi \Big)(x/r), 
\endaligned
\ee
where $t, s$ are parameters and $\Phi : S^2 \to \RR$ is also a given $C^2$-regular function, and 
$
\xi_{R_{s,t}} := \xi(x / R_{s,t}).
$
Here, the radius $R_{s,t} >0$ is chosen to be a sufficiently large constant so that $(g_1(s,t), h_1)$ satisfies the assumptions in Theorem \ref{theo-main3} with $(p, q) = (1/2, 3/2)$. 
We can apply Theorem \ref{theo-prec} and the corresponding Einstein solution $(g, h)(s, t)$ 
determined from the seed data set $(g_1(s, t), 0)$ has the form 
\begin{subequations}
\label{3eq:8.3:huensemble}
\begin{align}
{g_2 (s, t) \over r^2} & = - \big(\Delta_{g_1(s, t)}u_2 (s, t)\big)g_1(s, t) + \Hess_{g_1(s, t)}(u_2 (s, t)) - u_2 (s, t)\Ric_{g_1(s, t)},
\label{3eq:8.3:hua}
\\
h_2 (s, t) & = - {1 \over 2}  \Lcal_{Z_2 (s, t)}g_1(s, t),
\label{3eq:8.3:hub}
\end{align}
\end{subequations}
where 
$(g_2, h_2) (s, t) := (g - g_1, h)(s, t) 
$
and $ (u_2, Z_2) (s, t) \in C^{4,\alpha}_{1/2}(\RR^3) \times  C^{3,\alpha}_{1/2}(\RR^3). $


From now on, the notation $\Ocal_\eps^{2, \alpha} (r^{-4})$ stands for a function, say $f$, satisfying 
$\| f \|_{C^{2, \alpha}_{\theta}(\RR^3)} \lesssim \eps$. 
Denoting by $m$ the ADM mass of the Schwarzschild metric $g_\Sch$ under consideration, we observe that 
$$
\aligned
\Big(- {1 \over 2} \Deltas \Psi + \Psi \Big) \, g_\Sch + \Big(1 + {1 \over 2} \Deltas \Psi - \Psi \Big) \, \delta - \delta 
& = \Big(- {1 \over 2} \Deltas \Psi + \Psi \Big) ( g_\Sch - \delta)
\\
&
  = - 2m \, \Big({1 \over 2 r} \Deltas \Psi - {1 \over r} \Psi \Big) \, \delta + \Ocal_\eps^{2, \alpha} (r^{-4}) 
\endaligned
$$
and
$$
\aligned
\Bigg(  \del_{ij} \Big({1 \over 2 r} \Deltas \Psi - {1 \over r} \Psi \Big) \, \delta \Bigg)_{ij} 
- \Bigg(\del_{jj} \Big({1 \over 2 r} \Deltas \Psi - {1 \over r} \Psi \Big) \, \delta \Bigg)_{ii} 
& = - 2m \, \Delta \Big({1 \over 2 r} \Deltas \Psi - {1 \over r} \Psi \Big)
 = - 2 \, \Bigg({1 \over 2 } \Delta \Big( r^2 \Delta \Big( {\Psi \over r} \Big) \Big) - \Delta \Big( {\Psi \over r} \Big) \Bigg).
\endaligned
$$
Straightforward calculations similar to the ones done for the Hamiltonian operator in the proof of Theorem \ref{theorem:trois} (see for instance \eqref{82:Hcals}) lead us to
\bel{83:Hcal}
\aligned
\Hcal(g_1(s,t), h_1) & = \Delta \Bigg( {r^2 \over 2} \Delta \Big( {(4m + t) \Psi + s \, \Phi \over r} \Big) \Bigg)  
- \Delta \Big( {(4m + t) \Psi + s \, \Phi \over r} \Big) \Bigg)
+ K_0(s,t)
+ \Ocal_\eps^{2,\alpha}(r^{-4}), 
\\
K_0(s,t):= 
& \del_{ij} g_1(s,t)_{ij} - \del_{jj} g_1(s,t)_{ii} 
-  \Delta \Bigg( {r^2 \over 2} \Delta \Big( {(4m + t) \Psi + s \, \Phi \over r} \Big) \Bigg) 
+  \Delta \Bigg( {(4m + t) \Psi + s \, \Phi \over r} \Bigg). 
\endaligned
\ee
Here, it is helpful to keep in mind that $K_0(s,t) = 0$ for all $r \geq 2 \, R_{s,t}$.


\paragraph{Step 2. Super-harmonic estimates for the metric and dual metric variables.}
Proceeding as in 
the proof of Theorem~\ref{theorem:trois}, we obtain the following equation for the metric variable
\begin{subequations}
\label{3eq:8.3.general}
\bel{3eq:8.3:g2 hess}
\aligned
 - \Delta (\Tr g_2 (s, t) ) +  \del_{ij} g_2 (s, t)_{ij} 
&= - \Hcal( g_1(s, t), 0 ) + \smallcal_{012}(s,t), 
\endaligned
\ee
and, for the dual variables $u_2$ and $Z_2$, 
\begin{align}
\label{3eq:8.3:u1}
- 2\Delta u_2 (s, t) 
& = {1\over r^2} \Tr_e g_2 (s, t) + R_{g_1(s, t)} u_2 (s, t) + \smallcal_4(s,t), 
\\
\label{3eq:8.3:u2}
-4 \, \Delta u_2 (s, t) & = \del_{ij}\Big( {|x|^2 \over r^2} g_2 (s, t)_{ij} \Big) 
- \smallcal_3(s,t), 
\\
\Delta Z_2 (s, t) + \Div(\nabla Z_2 (s, t)) & =  
 \smallcal_{67}(s,t) .
\label{3eq:8.3:Z0}
\end{align}
Taking \eqref{3eq:8.3:u1} and \eqref{3eq:8.3:u2} into account in \eqref{3eq:8.3:g2 hess}, we deduce that 
\bel{3eq:8.3:Tr}
\aligned
& - \Delta (\Tr g_2) (s, t) + {2 \over r^2} \Tr g_2 (s, t) 
=  - \Hcal( g_1(s, t), 0 ) + \smallcal_{5} (s,t),
\endaligned
\ee
\end{subequations}
By recalling the expansion of the Hamiltonian operator in \eqref{83:Hcal}, the equations \eqref{3eq:8.3:u1} and \eqref{3eq:8.3:Tr} can be rewritten as 
\bse
\bel{3eq:8.3:uh0}
\aligned
&- \Delta \Bigg( \Tr g_2 (s, t) - {r^2 \over 2} \Delta \Big( {(4m + t) \Psi + s \, \Phi \over r} \Big) \Bigg) 
+ {2 \over r^2} \Bigg( \Tr g_2 (s, t) - {r^2 \over 2} \Delta \Big( {(4m + t) \Psi + s \, \Phi \over r} \Big) \Bigg) 
\\
& = \Delta \Bigg( {r^2 \over 2} \Delta \Big( {(4m + t) \Psi + s \, \Phi \over r} \Big) \Bigg) 
-  \Delta \Big( {(4m + t) \Psi + s \, \Phi \over r} \Big) -  \del_{ij} g_1(s,t)_{ij} +  \del_{jj} g_1(s,t)_{ii} 
           + \smallcal_{67}(s,t)
+ \Ocal_\eps^{2,\alpha}(r^{-4})
\endaligned
\ee
and
\bel{3eq:8.3:uh0b}
\aligned
- 2 \Delta \Bigg(u_2 (s, t) + {(4m + t) \Psi + s \, \Phi \over 4r} \Bigg) 
& = {1 \over r^2} \Bigg( \Tr g_2 (s, t) -  {r^2 \over 2} \Delta \Big( {(4m + t) \Psi + s \, \Phi \over r} \Big) \Bigg) 
+ R_{g_1(s, t)} u_2 (s, t).
\endaligned
\ee
\ese


Therefore, an analysis similar to 
the one in the proof of Theorem \ref{theo:deuxieme} shows that
$$
\aligned
& {1 \over r^2} \Bigg( \Tr g_2 (s, t) - {r^2 \over 2} \Delta \Big( {(4m + t) \Psi + s \, \Phi \over r} \Big) \Bigg) 
\in L^1(\RR^3) \cap C^{2, \alpha}_3 (\RR^3),
\qquad  
\\
&u_2 (s, t) + {(4m + t) \Psi + s \, \Phi \over 4r} \in C^{4, \alpha}_1 (\RR^3),
\qquad \qquad \qquad 
\qquad\quad
Z_2 (s, t) \in C^{3, \alpha}_1(\RR^3).
\endaligned
$$
It follows from \eqref{3eq:8.3:hua}-\eqref{3eq:8.3:hub} that $(g_2 (s, t), h_2 (s, t)) \in C^{2, \alpha}_{1}(\RR^3) \times C^{2, \alpha}_{2}(\RR^3)$, hence by \eqref{3eq:8.3:uh0}-\eqref{3eq:8.3:uh0b}
\bse
\begin{align} 
- 2\Delta \Bigg(u_2 (s, t) + {(4m + t) \Psi + s \, \Phi\over 4r} \Bigg) 
- {1 \over r^2} \Bigg( \Tr g_2 (s, t) - {r^2 \over 2} \Delta \Big( {(4m + t) \Psi + s \, \Phi \over r} \Big) \Bigg) 
& \in C^{2,\alpha}_4(\RR^3),
\label{3eq:8.3:uh2}
\\
- \Delta \Bigg( \Tr g_2 (s, t) - {r^2 \over 2} \Delta \Big( {(4m + t) \Psi + s \, \Phi \over r} \Big) \Bigg) 
+ {2 \over r^2} \Bigg( \Tr g_2 (s, t) - {r^2 \over 2} \Delta \Big( {(4m + t) \Psi + s \, \Phi \over r} \Big) \Bigg) 
& \in C^{0,\alpha}_4(\RR^3). 
\label{3eq:8.3:uh1}
\end{align}
\ese
By applying Proposition~\ref{laplaceplus} to \eqref{3eq:8.3:uh1} we have for any $\theta \in (1,2)$ 
\bel{3eq:8,3:fastdecay0}
\Big( \Tr g_2 (s, t) - {r^2 \over 2} \Delta \Big( {(4m + t) \Psi + s \, \Phi \over r} \Big) \Big) \in C^{2, \alpha}_{\theta} (\RR^3), 
\ee
and therefore, by applying Corollary \ref{remarkpoisson} to \eqref{3eq:8.3:uh2}, there exists a constant $\mt (s, t)$ such that
\bel{3eq:8.3:fastdecay}
u_2 (s, t) + {(4m + t)\Psi + s \, \Phi \over 4 r } - {\mt (s, t) \over r } \in C^{4, \alpha}_\theta (\RR^3). 
\ee
Similarly as in 
the proof of Theorem \ref{theo:deuxieme}, we also find 
$Z_2 (s, t) + V^\infty (s, t) \in C^{3, \alpha}_\theta(\RR^3)$, 
where 
$
V^\infty (s,t) := \Pt (s,t) \cdot V,
$
with some constant vector $\Pt(s,t)$  and with $V_{ij} = \Big( x_i x_j + 3 r^2 \delta_{ij} \Big)/ r^3$ as defined earlier. 
Plugging these results  
in \eqref{3eq:8.3:huensemble},  
we arrive at the desired decay properties. 

\begin{claim}
The metric and extrinsic curvature of the parametrized family of Einstein solutions satisfy 
\begin{subequations}\label{3eq:8.3:main1}
\begin{align}
&
g_2 (s, t) - r^2 \Delta \Bigg({(4m + t) \Psi + s \, \Phi \over 4 r}\Bigg) \, \delta 
+ r^2 \, \Hess \Big( {(4m + t) \Psi + s \, \Phi \over 4 r} \Big) 
- \mt (s, t) r^2 \, \Hess \Big( {1 \over r}\Big)  \in C^{2, \alpha}_\theta,
\label{3eq:8.3:c0a}
\\
&  
h_2 (s, t) - \Lcal_{V^\infty (s, t)} \delta  \in C^{2, \alpha}_{\theta + 1}.
\label{3eq:8.3:h2fastdecay}
\end{align}
\end{subequations}
\end{claim}

\paragraph{Step 3. Estimating the mass corrector.} We now estimate the coefficient $\mt(s,t)$.
Proceeding as 
in the proof of Theorem \ref{theorem:trois}, we find 
$$
\aligned
& \mt (s, t) = {1 \over 16 \pi} \int_{\RR^3} 
\Bigg(  \del_{ij} g_2(s,t)_{ij} - \Delta \Big( {(4m + t)\Psi + s \, \Phi \over r}\Big)\Bigg) \, dx
           = {1 \over 16 \pi} \int_{\RR^3}  K_0(s,t)  \, dx 
  + \int_{\RR^3} \Ocal_\eps^{0,\alpha} (r^{-4}) \, dx
\\
& = {1 \over 16 \pi} \lim_{r \to +\infty} \int_{S_r} 
\Bigg( g_1(s, t)_{ii,j} - g_1(s, t)_{ij,i} 
+ \del_j \Bigg( {r^2 \over 2} \Delta \Big( {(4m + t)\Psi + s \, \Phi \over r} \Big) \Bigg) 
- \del_j \Big( {(4m + t)\Psi + s \, \Phi \over r} \Big) \Bigg) 
\, {x_j \over r} \, d\omega
\\
& \quad
 + \int_{\RR^3} \Ocal_\eps^{0,\alpha}(r^{-4}) \, dx.
\endaligned
$$
By integration by parts we obtain 
$$
\aligned
& \int_{\RR^3} \Big( \del_{ij} \Big( \Big( {1 \over 2} \Deltas \Psi - \Psi \Big) (g_\Sch - \delta)_{ij} \Big) - \del_{ii} \Big( \Big( {1 \over 2} \Deltas \Psi - \Psi \Big) (g_\Sch - \delta)_{jj} \Big) \Big) \, dx
\\
& = \lim_{r \to +\infty} \int_{S_r}
 \Bigg( \del_{i} \Big( \Big( {1 \over 2} \Deltas \Psi - \Psi \Big) (g_\Sch - \delta)_{ij} \Big) - \del_{j} \Big( \Big( {1 \over 2} \Deltas \Psi - \Psi \Big) (g_\Sch - \delta)_{ii} \Big) \Bigg) 
 \, {x_j \over r} \, d\omega
 = - 4m \int_{S^2} \Psi \, dx.
\endaligned
$$
Furthermore, in view of \eqref{82:Phiinteg}, we have 
$$
\aligned
\int_{S_r}  \Big( \del_j \Big( {r^2 \over 2} \Delta \Big( {\Phi \over r} \Big) \Big) 
-  \del_j \Big( {\Phi \over r} \Big) \Big) {x_j \over r} \, d\omega
& = \int_{S^2} \Phi \, d\omega,
\qquad
\int_{S_r}  \Big( \del_j \Big( {r^2 \over 2} \Delta \Big( {\Psi \over r} \Big) \Big) 
-  \del_j \Big( {\Psi \over r} \Big) \Big) {x_j \over r} \, d\omega
= \int_{S^2} \Psi \, d\omega
\endaligned
$$
and
$$
\aligned
\int_{S_r} 
\Bigg( \del_{j}\Big( \Big( -{1 \over 2} \Deltas \Phi + \Phi \Big) \, r^2 \del_{ii} \Big( {1 \over r} \Big) \Big)
- \del_{i}\Big( \Big( -{1 \over 2} \Deltas \Phi + \Phi \Big) \, r^2 \del_{ij} \Big( {1 \over r} \Big) \Big) \Bigg) 
\, 
{x_j \over r} \, d\omega 
& = 4 \int_{S^2} \Phi \, d\omega,
\\
\int_{S_r} \Bigg( \del_{i}\Big( \Big( -{1 \over 2} \Deltas \Psi + \Psi \Big) \, r^2 \del_{ij} \Big( {1 \over r} \Big) \Big) - \del_{j}\Big( \Big( -{1 \over 2} \Deltas \Psi + \Psi \Big) \, r^2 \del_{ii} \Big( {1 \over r} \Big) \Big) \Bigg) 
\, 
{x_j \over r} \, d\omega 
& = 4 \int_{S^2} \Psi \, d\omega. 
\endaligned
$$
Taking these identities into account, we arrive at
\bel{83:equ:mSch}
\mt (s, t)= {5 \over 16 \pi} \Bigg( t \int_{S^2} \Psi \, d\omega + s \int_{S^2} \Phi \, d\omega \Bigg) 
+ \int_{\RR^3} \Ocal_\eps^{0,\alpha}(r^{-4}) \, dx.
\ee 


\paragraph{Step 4. Selecting a particular seed data.} It remains to specify the field $\Phi$ and as well as the real parameters $s,t$.
In view of \eqref{eq:8.3:main1} the Einstein solution $(g, h)$ satisfies
$$
\aligned
& g(s, t) - g_1(s,t) - {1 \over 4r} \Deltas ((4m + t) \Psi + s \, \Phi ) \delta + r^2 \, \Hess \Bigg( {(4m + t) \Psi + s \, \Phi \over 4 r} \Bigg) - \mt (s, t) \, r^2 \, \Hess \Big( {1 \over r}\Big) 
\in C^{2, \alpha}_\theta
\endaligned
$$
and $h(s, t) - \Lcal_{V^\infty (s, t)} \delta \in C^{2, \alpha}_{\theta + 1}$. 
By the definition of $g_1(s,t)$, for $r \geq 2 \, R_{s,t}$ we have 
$$
\aligned
& g(s, t) - g_1(s,t) - {1 \over 4r} \Deltas ((4m + t) \Psi + s \, \Phi ) \delta + r^2 \, \Hess \Bigg( {(4m + t) \Psi + s \, \Phi \over 4 r} \Bigg) - \mt (s, t) r^2 \, \Hess \Big( {1 \over r}\Big)
\\
& = g(s,t) + \Big({1 \over 2} \Deltas \Psi - \Psi \Big) g_\Sch  - \Big(1 + {1 \over 2} \Deltas \Psi - \Psi \Big) \, \delta
      - {t \over 2} \Deltas \Psi \, r^2 \, \Hess \big( 1/r \big) 
      \\
      & \quad
       -{s \over 2} \Deltas \Phi r^2 \, \Hess \Big( {1 \over r}\Big) 
      - {1 \over 4r} \Deltas ((4m + t) \Psi + s \, \Phi ) \delta
\\
& \quad + t \, \Psi r^2 \, \Hess \Big( {1 \over r} \Big) + {(4m + t) \over 4} r^2 \, \Hess \Big( {\Psi \over r} \Big)
  + {s \over 4} r^2 \, \Hess \Big( {\Phi \over r} \Big)
+  s \, \Phi r^2 \, \Hess \Big( {1 \over r} \Big) -  \mt(s,t) \, r^2 \, \Hess \Big( {1 \over r} \Big).
\endaligned
$$
Provided  
$
t = - {4 m \over 5}
$
and $
\Phi \big( {x \over r} \big) = 1$ for all $x \in \Cscr_a \cup \Cscr^c_{a + \eps}, 
$
from the definition of $\Psi$ it follows that 
\bel{eq:8.3:main1b}
g(s,-4m/5) - \Psi g_\Sch  - \Big(1 - \Psi \Big) \, \delta + \Big( {5 \, s \over 4} - \mt (s, -4m/5) \Big) \, r^2 \, 
\Hess \Big( {1 \over r}\Big) \in C^{2,\alpha}_\theta (\Cscr_a \cup \Cscr^c_{a + \eps}).
\ee
Now thanks to \eqref{83:equ:mSch}, proceeding as 
in the proof of Theorem \ref{theorem:trois} and provided the 
vanishing integral condition \eqref{equa-vanish}  
is assumed by our seed data, there exists a constant $s^*$ satisfying 
$s^* = {4 \over 5} \mt ( s^*, - 4m/5 )$. 
Hence, \eqref{eq:8.3:main1b} associated with $s^*$ is rewritten as
$$
g(s^*, - 4m/5) - \Psi g_\Sch  - (1 - \Psi) \, \delta  \in C^{2,\alpha}_\theta (\Cscr_a \cup \Cscr^c_{a + \eps}), 
$$
and the desired conclusion follows in view of our choice of $\Psi$ in \eqref{equa-choice-Psi}. This completes the proof of Theorem~\ref{corollary-optimalCS}. 


\paragraph*{Acknowledgments.} 

The authors gratefully thank Bruno Le Floch (Sorbonne Universit\'e and CNRS) and two referees  
for their very careful reading of the original manuscript 
and for making many comments,
which helped us improve the presentation in the final version.
This paper was written during the Academic year 2018-2019 when the first author was
a visiting research fellow at the 
Courant Institute for Mathematical Sciences, New York University. The second author (TCN) is grateful to the Fondation des Sciences Math\'ematiques de Paris for a postdoc fellowship spent at Sorbonne Universit\'e during the period 
2016--2018. The authors were also partially supported by the Innovative Training Network (ITN) grant 642768 (ModCompShock).


\appendix  

\section{Technical arguments on the seed-to-solution map (Section~\ref{section--3})} 
\label{appendix-AAA}

\begin{proof}[Proof of Lemma~\ref{lemma WP1}]
 {\bf Case 1: $a = 1/2$.} 
Since the set of smooth and compactly supported functions on $\RR^3$ is dense in $H^k_a(\RR^3)$, we can assume that $w$ has bounded support.  
The weight function $v= \ln|x|$ satisfies 
$
\nabla v = |x|^{- 1} \nabla |x|$
and $\Delta v = |x|^{- 2}$ at any $x \neq 0$. 
Integrating by parts the integral of $w^2\Delta v$ over the exterior of the unit ball $B_1(0) \subset \RR^3$, we obtain
$$
\int_{\RR^3\setminus B_1(0)} w^2\Delta v \, dx 
=  - \int_{\del B_1(0)}w^2|x|^{- 1}dA(x) - 2\int_{\RR^3\setminus B_1(0)} w \, \nabla w \, \nabla v \, dx.
$$
The boundary term has a favorable sign and thus 
$
\int_{\RR^3\setminus B_1(0)}w^2 |x|^{-2} dx \leq 2\int_{\RR^3\setminus B_1(0)} w \, \nabla w \, \nabla v \, dx$, 
hence by the Cauchy-Schwarz inequality
$$
\int_{\RR^3\setminus B_1(0)}w^2|x|^{- 2} \, dx 
\leq  \Big(\int_{\RR^3\setminus B_1(0)}w^2|x|^{- 2} \, dx \Big)^{1/2} \Big(\int_{\RR^3\setminus B_1(0)} | \nabla w|^2 \, dx \Big)^{1/2}. 
$$
Therefore, we have established the Hardy-type inequality  
\bel{outsideball}
\int_{\RR^3\setminus B_1(0)}w^2|x|^{- 2} \, dx \leq \int_{\RR^3\setminus B_1(0)} | \nabla w|^2 \, dx.
\ee

Now let $\eta$ be a cut-off function supported in the ball $B_2(0)$ with radius $2$ and satisfying $\eta \equiv1$ in $B_1(0)$. By the Poincar\'e inequality for function supported in $B_2(0)$, we have 
$
\int_{\RR^3}(w\eta)^2 \, dx 
\lesssim
\int_{\RR^3} | \nabla(\eta w)|^2 \, dx,
$
which implies 
$$
\aligned
\int_{B_1(0)} w^2 \, dx 
& 
\lesssim
\int_{B_2(0)} | \nabla w|^2 \eta^2 \, dx + \int_{B_2(0)}w^2| \nabla\eta|^2 \, dx
\lesssim
\int_{B_2(0)} | \nabla w|^2 \, dx + \int_{B_2(0)\setminus B_1(0)}w^2 \, dx.
\endaligned
$$
Since $r$ is bounded above and below by positive constants on $B_2(0)$, from \eqref{outsideball} it follows that 
$$
\int_{B_1(0)}w^2r^{- 2} \, dx 
\lesssim 
\int_{B_2(0)} | \nabla w|^2 \, dx + \int_{\RR^3} | \nabla w|^2 \, dx.
$$
Combining the above two inequalities, we conclude that
$ 
\int_{\RR^3}w^2r^{- 2} \, dx 
\lesssim \int_{\RR^3} | \nabla w|^2 \, dx, 
$
which is the first inequality in Proposition~\ref{lemma WP1}.

\hskip.15cm

\noindent{\bf Case 2: $a \ne 1/2$.} For definiteness, we treat the case $a \in (0, 1/2)$.
 Setting $v=|x|^{- 1 + 2a}$, for all  $x \ne 0$ we have 
$\nabla v=  - (1 - 2a)|x|^{- 2 + 2a} \nabla |x|$
and 
$\Delta v= - 2a(1 - 2a)|x|^{- 3 + 2a}$. 
Integrating by parts we establish that
$$
\int_{\RR^3\setminus B_1(0)}w^2\Delta v dx = (1 - 2a)\int_{\del B_1(0)}w^2|x|^{- 2 + 2a}dA(x) - 2\int_{\RR^3\setminus B_1(0)} w \, \nabla w \, \nabla v \, dx.
$$
The left-hand side and the boundary term have opposite sign, and we deduce that 
$$
a\int_{\RR^3\setminus B_1(0)}w^2|x|^{- 3 + 2a} \, dx \leq  \int_{\RR^3\setminus B_1(0)} |x|^{- 2 + 2a}w \, \nabla w \nabla |x| dx
$$
and, by the Cauchy-Schwarz inequality, 
$$
a\int_{\RR^3\setminus B_1(0)}w^2|x|^{- 3 + 2a} \, dx \leq \Big(\int_{\RR^3\setminus B_1(0)}w^2|x|^{- 3 + 2a} \, dx \Big)^{1/2} \Big(\int_{\RR^3\setminus B_1(0)} | \nabla w|^2|x|^{- 3 + 2(a + 1)} \, dx \Big)^{1/2}.
$$
It follows that
$$
\int_{\RR^3\setminus B_1(0)}w^2|x|^{- 3 + 2a} \, dx 
\lesssim
 {1 \over a} \int_{\RR^3\setminus B_1(0)} | \nabla w|^2|x|^{- 3 + 2(a + 1)} \, dx.
$$
Finally, an analysis similar to the one in Case~1 leads us to 
$$
\int_{\RR^3}w^2r^{- 3 + 2a} \, dx 
\lesssim
{1 \over a^2} \int_{\RR^3} | \nabla w|^2r^{- 1 + 2a} \, dx.
$$


\vskip.16cm

\noindent{\bf Hessian estimate.} In the same spirit, we use the previous arguments to handle each partial derivative of $w$ with $v=|x|^{1 + 2a}$ and we get
$$
\int_{\RR^3} | \nabla w|^2 r^{- 1 + 2a} \, dx \lesssim \int_{\RR^3} | \Hess(w)|^2r^{1 + 2a} \, dx.
$$  
By applying the first inequality in Proposition~\ref{lemma WP1}, we obtain immediately the second inequality. 
\end{proof}
 

\begin{proof}[Proof of Lemma~\ref{prop. weighted KKK}] 
\bse
{\bf Step 1.}  We write $Z=Y + Z_ne_n$, where $e_n = \del_r$ and $Y$ is orthogonal to the radial direction. Since $q<2$, we have $a \coloneqq 2  -  q >0$ and so, by Proposition~\ref{lemma WP1}, 
\bel{Yn}
	\int_{\RR^3}(Z_n)^2r^{- 3 + 2a} \, dx 
=  - \frac{1}{a} \int_{\RR^3}(Z_n\del_rZ_n)r^{- 2 + 2a} \, dx.
\ee
Denoting by $D$ the Euclidean derivative operator, we have $D_{e_n}e_n = 0$ and $\del_rZ_n ={1 \over 2}  \Dcal(Z)_{nn}$. 
 Therefore, \eqref{Yn} gives us 
\bel{Yn2}
\int_{\RR^3}(Z_n)^2r^{- 3 + 2a} \, dx
\lesssim
\int_{\RR^3} | \Dcal(Z)|^2r^{- 1 + 2a} \, dx,
\ee
which is the desired bound for $Z_n$. A similar argument shows that
\bel{Yn1}
\int_{\RR^3} |Z|^2r^{- 3 + 2a} \, dx
= - \frac{1}{a} \int_{\RR^3}(Z \cdot \del_rZ)r^{- 2 + 2a} \, dx.
\ee

Now, with $Z=Y + Z_ne_n$ we observe that
$Z \cdot \del_r Z = \Dcal(Z)(e_n,Z) - D_ZZ \cdot e_n$ 
and 
$$
D_ZZ \cdot e_n 
= D_YY \cdot e_n + Z_n D_{e_n}(Z_ne_n) \cdot e_n + Z_nD_{e_n}Y \cdot e_n + D_Y(Z_ne_n) \cdot e_n.
$$
The second identity is rewritten as
$
D_ZZ \cdot e_n = - r^{- 1} |Y|^2 + Z_ne_n(Z_n) + Y(Z_n) 
= - r^{- 1} |Y|^2 + Z(Z_n)$. 
Hence, by recalling the identity $Z \cdot \del_r Z = \Dcal(Z)(e_n,Z) - D_ZZ \cdot e_n$, we find 
$Z \cdot \del_r Z = \Dcal(Z)(e_n,Z) + r^{- 1} |Y|^2 - Z(Z_n)$. 
Combining this result with \eqref{Yn1}, we obtain
$$
\int_{\RR^3} |Z|^2r^{- 3 + 2a} \, dx 
\leq \frac{1}{a} \int_{\RR^3} |Z|| \Dcal(Z)|r^{- 2 + 2a} \, dx 
+ \frac{1}{a} \int_{\RR^3}Z(Z_n)r^{- 2 + 2a} \, dx.
$$
Integrating the second term by parts, we have
$$
\int_{\RR^3} |Z|^2r^{- 3 + 2a} \, dx 
\leq \frac{1}{a} \int_{\RR^3} |Z|| \Dcal(Z)|r^{- 2 + 2a} \, dx
 - \frac{1}{a} \int_{\RR^3} \Div(r^{- 2 + 2a}Z)Z_n \, dx.
$$
Now since $\Div(Z) \lesssim | \Dcal(Z)|$, we have 
$$
\aligned
\int_{\RR^3} |Z|^2r^{- 3 + 2a} \, dx
& \lesssim
\int_{\RR^3} |Z|| \Dcal(Z)|r^{- 2 + 2a} \, dx
+ \int_{\RR^3} | \Dcal(r^{- 2 + 2a}Z)|Z_n \, dx
\\
&
\lesssim
\int_{\RR^3} |Z|| \Dcal(Z)|r^{- 2 + 2a} \, dx + \int_{\RR^3} (Z_n)^2r^{- 3 + 2a} \, dx
\\
& \lesssim
\Big(\int_{\RR^3} |Z|^2r^{- 3 + 2a} \, dx \Big)^{1/2} \Bigg(\int_{\RR^3} | \Dcal(Z)|^2r^{- 1 + 2a} \, dx \Bigg)^{1/2} 
+ \int_{\RR^3}(Z_n)^2r^{- 3 + 2a} \, dx.
\endaligned.
$$
In combination with \eqref{Yn2}, we find 
$$
\int_{\RR^3} |Z|^2r^{- 3 + 2a} \, dx 
\lesssim
\Big( \int_{\RR^3} |Z|^2r^{- 3 + 2a} \, dx \Big)^{1/2} \Big(\int_{\RR^3} | \Dcal(Z)|^2r^{- 1 + 2a} \, dx \Big)^{1/2} 
+ \int_{\RR^3} | \Dcal(Z)|^2r^{- 1 + 2a} \, dx,
$$
hence  
$
\int_{\RR^3} |Z|^2r^{- 3 + 2a} \, dx 
\lesssim
\int_{\RR^3} | \Dcal(Z)|^2r^{- 1 + 2a} \, dx. 
$

\ese

\vskip.3cm

\bse

\noindent{\bf Step 2.}  We will prove the result 
when $a>0$ with $a \neq 1/2$, since the case where $a = 1/2$ is similar.
In view of Step~1, it suffices to show that 
\bel{nablaZ1}
\int_{\RR^3} | \nabla Y|^2r^{- 1 + 2a} \, dx
\lesssim
\int_{\RR^3} | \Dcal(Y)|^2r^{- 1 + 2a} \, dx + \int_{\RR^3} |Y|^2 r^{- 3 + 2a} \, dx.
\ee
By the definition of the Lie derivative, we have
$$
\int_{\RR^3}(Y^2_{i;j} + Y^2_{j;i})r^{- 1 + 2a} \, dx
\lesssim
\int_{\RR^3} | \Dcal(Y)|^2r^{- 1 + 2a} \, dx - 2\int_{\RR^3}Y_{i;j}Y_{j;i}r^{- 1 + 2a} \, dx.
$$
Thus to establish \eqref{nablaZ1}, 
 we now show that,  for some sufficiently small $\eps \in (0,1)$,
\bel{nablaZ2}
\aligned
- 2\int_{\RR^3}Y_{i;j}Y_{j;i}r^{- 1 + 2a} \, dx
& \leq 
 \eps \int_{\RR^3} | \nabla Y|^2r^{- 1 + 2a} \, dx
+ \eps^{- 1} \Big(\int_{\RR^3} | \Dcal(Y)|^2r^{- 1 + 2a} \, dx + \int_{\RR^3} |Y|^2 r^{- 3 + 2a} \, dx \Big). 
\endaligned
\ee
In fact, by the divergence theorem we get
\bel{nablaZ3}
 - \int_{\RR^3}Y_{i;j}Y_{j;i}r^{- 1 + 2a} \, dx \leq \int_{\RR^3}Y_jY_{i;ji}r^{- 1 + 2a} \, dx 
+ \int_{\RR^3} |Y_jY_{i;j} |r^{- 2 + 2a} \, dx
\ee
and, observing that $Y_jY_{i;ji}=Y_jY_{i;ij}$ and applying the divergence theorem again,  
$$
\int_{\RR^3}Y_jY_{i;ji}r^{- 1 + 2a} \leq  - \int_{\RR^3}Y_{i;i}Y_{j;j}r^{- 1 + 2a} \, dx 
+ \int_{\RR^3} |Y_jY_{i;i} |r^{- 2 + 2a} \, dx.
$$
In combination with \eqref{nablaZ3}, we arrive at 
$$
\aligned
 - \int_{\RR^3}Y_{i;j}Y_{j;i}r^{- 1 + 2a} \, dx  
& \lesssim
\int_{\RR^3} |Y_{i;i}Y_{j;j} |r^{- 1 + 2a} \, dx
+ \int_{\RR^3} |Y_jY_{i;j} |r^{- 2 + 2a} \, dx + \int_{\RR^3} |Y_jY_{i;i} |r^{- 2 + 2a} \, dx
\\
& \lesssim
\int_{\RR^3} | \Dcal(Y)|^2r^{- 1 + 2a} \, dx + \int_{\RR^3} |Y|| \nabla Y|r^{- 2 + 2a} \, dx.
\endaligned
$$
By the Cauchy-Schwarz inequality, we have 
$$
\int_{\RR^3} |Y|| \nabla Y|r^{- 2 + 2a} \, dx
\leq \eps 
\int_{\RR^3} | \nabla Y|^2r^{- 1 + 2a} \, dx  + \eps^{- 1} \int_{\RR^3} |Y|^2r^{- 3 + 2a} \, dx 
$$
 (with $\eps \in (0,1)$) 
and, therefore, we arrive at \eqref{nablaZ2} as claimed.
\ese
\end{proof}


\begin{proof}[Proof of Lemma \ref{lemma-onemore}]
In view of the expressions \eqref{eq=formuleexactes} and \eqref{eq=formuleexactes2} of the linearized operators, we estimate the terms involving the function $u$ as follows. For the Laplace term in \eqref{eq=formuleexactes}, by denoting by $\Gamma^l_{jk}$ the Christoffel symbols of a metric we have 
$$
\aligned
& \|(\Delta_{g_1}u) \, g_1 - (\Delta u)e\|_{L^2_{3-p}(\Mbf)} 
 \lesssim \|(\Delta_{g_1} u - \Delta u)e\|_{L^2_{3-p}(\Mbf)} + \|(\Delta_{g_1}u)(g_1 - e)\|_{L^2_{3-p}(\Mbf)}
\\
& \lesssim \, \Big( \int_\Mbf  |g_1^{jk} - e^{jk} |^2|\del_{jk} u|^2 r^{3 -2p} \, dV_e \Big)^{1/2} 
+ \Big( \int_\Mbf  |g_1^{jk} \Gamma^{l}_{jk}(g_1) - e^{jk} \Gamma^{l}_{jk}(e)|^2|\del_{l} u|^2 r^{3 -2p} \, dV_e \Big)^{1/2}
\\
& \quad
    +   \eeG \, \Big( \int_\Mbf  | g_1^{jk} |^2 \big( |\del_{jk} u| + | \Gamma_{ij}^k | | \del_k u | \big)^2 r^{3 -2p}|g_1 - e|^2 \, dV_e \Big)^{1/2}
\\
& \lesssim \eeG \, \Big( \int_\Mbf  |\del_{jk} u|^2 r^{3 -2p - 2\ppG} \, dV_e \Big)^{1/2} 
       + \eeG\Big( \int_\Mbf  |\del_{l} u|^2 r^{1 - 2p - 2\ppG} \, dV_e \Big)^{1/2}, 
\endaligned
$$ 
therefore
$
\|(\Delta_{g_1}u) \, g_1 - (\Delta u)e\|_{L^2_{3-p}(\Mbf)}  \lesssim \eeG \, \|u\|_{H^2_{1-p}(\Mbf)}.
$
For the Hessian term, we find  
$$
\aligned
 \|\Hess_{g_1}(u) - \Hess_e(u)\|_{L^2_{3 -p}(\Mbf)} 
&
\lesssim \, \Big(\int_\Mbf  \big|\big( \Gamma^k_{ij}(g_1) - \Gamma^k_{ij}(e) \big) \del_ku \big|^2 r^{3 - 2p} \, dV_e\Big)^{1/2}
\\
&
 \lesssim \eeG \, \Big(\int_\Mbf  |\del_ku|^2 r^{1 - 2 p - 2\ppG} \, dV_e\Big)^{1/2}
\endaligned
$$
and, therefore, 
$
\|\Hess_{g_1}(u) - \Hess_e(u)\|_{L^2_{3-p}(\Mbf)}    \lesssim \eeG\|u\|_{H^2_{1-p}(\Mbf)}. 
$
For the zero-order terms in $u$ we write  
$$
\aligned
& 
\Big\|u \big( \Ric(g_1) - \Ric(e) \big) + \Big((\Tr_{g_1} h_1)h_1 - h_ 1 \times h_1\Big) \, u\Big\|_{L^2_{3-p}(\Mbf)} 
\\
& \lesssim \, \Big(\int_\Mbf  |u \big(\Ric(g_1) - \Ric(e)  \big) |^2 r^{3 - 2p} \, dV_e\Big)^{1/2}
    + \Big(\int_\Mbf  \Big|(\Tr_{g_1} h_1)h_1 - h_1\otimes h_1\Big|^2u^2 r^{3 - 2p} \, dV_e\Big)^{1/2} 
\\
& \lesssim 
\eeG\Big(\int_\Mbf  |u|^2 r^{- 1 - 2p - 2\ppG} \, dV_e\Big)^{1/2}
   + \eeG\Big(\int_\Mbf  |u|^2 r^{3 - 2p - 4\qqG} \, dV_e\Big)^{1/2}
       \lesssim \eeG\|u\|_{H^2_{1-p}(\Mbf)}.
\endaligned
$$
Next, we consider the terms in \eqref{eq=formuleexactes}  involving $Z$. By the definition the Lie derivative, we have  
(using coma for partial derivatives and implicit summation over $i,j,k$)
$$
\aligned
\|\Lcal_Zh_1\|_{L^2_{3-p}(\Mbf)} 
& \lesssim \, \Big(\int_\Mbf  |h_{1ij,k}Z^i|^2r^{3 - 2p} \, dV_e\Big)^{1/2} + \Big(\int_\Mbf  |h_{1kj} \del_iZ^i|^2r^{3 - 2p} \, dV_e\Big)^{1/2}
\\
& \lesssim \eeG \, \Big(\int_\Mbf  |Z^i|^2r^{1 - 2p - 2\qqG} \, dV_e\Big)^{1/2} + \eeG \, \Big(\int_\Mbf  |\del_iZ^i|^2r^{3 - 2p - 2\qqG} \, dV_e\Big)^{1/2}
  \lesssim \eeG \, \|Z\|_{H^1_{2-\oldq}(\Mbf)},
\endaligned
$$
and
$$
\aligned
\|\Lcal_Z g_1 - \Lcal_Z e\|_{L^2_{3-\oldq}(\Mbf)} 
& \lesssim 
\Big(\int_\Mbf  \Big(
|(g_{1ij,k} - e_{1ij,k}) Z^k|^2 + |(g_{1ij} - e_{ij})\del_k Z^{i} |^2 
\Big)
\, r^{3 - 2\oldq} \, dV_e\Big)^{1/2} 
\\
& \lesssim 
\eeG \, \Big(
\int_\Mbf  |Z^k|^2r^{1 - 2\oldq -2\ppG} + |\del_k Z^{i} |^2r^{3 - 2\oldq -2\ppG} \, dV_e\Big)^{1/2} 
\lesssim \eeG \, \|Z\|_{H^1_{2-\oldq}(\Mbf)}.
\endaligned
$$
Also, we have
$$
\aligned
\|g_1\big(\Lcal_Z g_1, h_1\big)\|_{L^2_{3-p}(\Mbf)}  & \lesssim \, \Big(\int_\Mbf  \big(|g_{1ij,k} Z^k|^2 + |g_{1kj} |^2|\del_iZ^k|^2 \big) |h_1 |^2r^{3 - 2p} \, dV_e \Big)^{1/2}
\\
& \lesssim \eeG\Big(\int_\Mbf  \big(|Z^k|^2r^{1 - 2p - 2\ppG - 2\qqG} + |\del_iZ^k|^2r^{3 - 2p -2\qqG} \big) \, dV_e \Big)^{1/2}
 \lesssim \eeG \, \|Z\|_{H^1_{2-\oldq}(\Mbf)}. 
\endaligned
$$
On the other hand, from the definition of the divergence operator we obtain 
$$
\aligned
\|\Div_{g_1} (Z) h_1\|_{L^2_{3-p}(\Mbf)} & \lesssim \, \Big(\int_\Mbf  |\del_i Z^j h_{1kl} |^2r^{3- 2p} \, dV_e \Big)^{1/2} 
\\
&
\lesssim \, \Big(\int_\Mbf  |\del_i Z^j|^2r^{3- 2p -2\qqG} \, dV_e \Big)^{1/2}
    \lesssim \eeG \, \|Z\|_{H^1_{2-\oldq}(\Mbf)}.
\endaligned
$$
Similarly, we find 
$$
\aligned
\|Z\otimes\Div_{g_1} h_1 + \Div_{g_1} h_1\otimes Z\|_{L^2_{3-p}(\Mbf)}  
 &\lesssim \eeG \, \Big(\int_\Mbf  |h_{1ij,k} Z^l|^2 r^{3 - 2p} \, dV_e\Big)^{1/2}
 \\
 &
 \lesssim \eeG  \, \Big(\int_\Mbf  |Z^l|^2 r^{1 - 2p - 2\qqG} \, dV_e\Big)^{1/2}
   \lesssim \eeG \, \|Z\|_{H^1_{2-\oldq}(\Mbf)}. 
\endaligned
$$
Finally, we estimate the last term in \eqref{eq=formuleexactes} by 
$$
\aligned
\|g_1(Z,\Div_{g_1} h_1)\|_{L^2_{3-p}(\Mbf)} 
& \lesssim \, \Big(\int_\Mbf  |Z^kh_{1ij,l} |^2r^{3 -2p} \, dV_e\Big)^{1/2} 
\\
&
\lesssim \eeG \, \Big(\int_\Mbf  |Z^k|^2r^{1 -2p - 2\eta} \, dV_e\Big)^{1/2}
   \lesssim \eeG \, \|Z\|_{H^1_{2-\oldq}(\Mbf)},
\endaligned
$$
which completes the study for the Hamiltonian constraint \eqref{eq=formuleexactes}. 
Finally, we deal with \eqref{eq=formuleexactes2} similarly, by writing for instance 
$$
\aligned
& \Big\|\Big((\Tr h_1)g_1^{-1} - h_1\Big) \, u\Big\|_{L^2_{3-\oldq}(\Mbf)} 
 \lesssim \, \Big(\int_\Mbf  \Big|(\Tr h_1)g_1^{-1} - h_1\Big|^2u^2 r^{3 - 2\oldq} \, dV_e \Big)^{1/2}
\\
& \lesssim \eeG\Big(\int_\Mbf u^2 r^{3 - 2\oldq -2\qqG} \, dV_e \Big)^{1/2}
 \lesssim \eeG \, \Big(\int_\Mbf u^2 r^{-1 - 2p} \, dV_e \Big)^{1/2}
  \lesssim \eeG\|u\|_{H^2_{1-p}(\Mbf)}. 
\endaligned
$$
Taking also our previous estimates into account, we can complete the argument for the Hamiltonian operator. 
\end{proof}


\bse

\begin{proof}[Proof of Proposition \ref{pro_integral}] 
1. 
 Recall that $(u,Z)$ is the unique minimizer of $\Jbf_{(g_1, h_1, f, V)}$ over the function space $H^2_{1-\oldp}(\Mbf)\times H^1_{2-\oldq}(\Mbf)$ and  we have 
$$
d\Gcal_{(g_1,h_1))}[g_2,h_2] 
=  (f, V)$$
 with 
$g_2  = r^{3-2p} d\Hcal_{(g_1,h_1)}^*[u,Z]$, 
and 
$h_2 = r^{3-2\oldq} d\Mcal_{(g_1,h_1)}^*[u,Z]$.  
It follows that
$$
\Jbf_{(g_1, h_1, f, V)}(u,Z) \leq \Jbf_{(g_1, h_1, f, V)}(0,0) = 0,
$$
hence
$$
\big(\| g_2\|_{L^2_{p}(\Mbf)} + \| h_2 \|_{L^2_{q}(\Mbf)} \big)^2
\lesssim 
 \big(\| u\|_{L^2_{1 - p}(\Mbf)} +  \| Z \|_{L^2_{2-\oldq}(\Mbf)} \big) \big(\| f \|_{L^2_{\oldp + 2}(\Mbf)} + \| V \|_{L^2_{\oldq+1}(\Mbf)} \big).
$$
On the other hand, Proposition~\ref{prop.key} gives us 
\bel{eq:43a}
\aligned
\| u \|_{H^2_{1-p}(\Mbf)} 
& \lesssim 
\big\|   g_2  \big\|_{L^2_{p}(\Mbf)}
+ \eeG \, \big\|  h_2   \big\|_{L^2_{\oldq}(\Mbf)},
\qquad 
 \| Z \|_{H^1_{2-\oldq}(\Mbf)} 
 \lesssim 
\eeG \, \big\|  g_2 \big\|_{L^2_{p}(\Mbf)}
+ 
\big\|  h_2   \big\|_{L^2_{\oldq}(\Mbf)}.
\endaligned
\ee
Taking these inequalities into account, we have
\bel{eq:43b}
\| g_2\|_{L^2_{p}(\Mbf)} + \| h_2 \|_{L^2_{q}(\Mbf)}
\lesssim 
  \| f \|_{L^2_{\oldp + 2}(\Mbf)} + \| V \|_{L^2_{\oldq+1}(\Mbf)}. 
\ee

\vskip.16cm 

2. 
Next in view of  $\Jbf_{(g_1, h_1, f, V)}(u,Z) \leq \Jbf_{(g_1, h_1, f, V)}(0,Z)$,   
we obtain 
$$
\aligned
\| g_2 \|_{L^2_{p}(\Mbf)}^2
\lesssim
& \int_\Mbf  \Big( |d \Hcal^*_{(g_1, h_1)} [0, Z] |^2 r^{3 - 2p} + \Big( |d \Mcal^*_{(g_1, h_1)} [0, Z] |^2 - |d \Mcal^*_{(g_1, h_1)} [u, Z] |^2 \Big) r^{3 - 2q} \Big) \, dV_{g_1} 
\\
&  +
\| u \|_{L^2_{1 - \oldp}(\Mbf)} \| f \|_{L^2_{\oldp + 2}(\Mbf)}.
\endaligned
$$
Therefore, a straightforward calculation (similar to the one in the proof of Lemma \ref{lemma-onemore}) leads us to 
$$
\| g_2 \|_{L^2_{p}(\Mbf)}^2
\lesssim \eeG^2 \big(\| Z \|^2_{H^1_{2 - q}(\Mbf)} + \| u \|^2_{H^2_{1-p}(\Mbf)} \big) + \| u \|_{L^2_{1 - \oldp}(\Mbf)} \| f \|_{L^2_{\oldp + 2}(\Mbf)}
$$
and so, by recalling \eqref{eq:43a} and \eqref{eq:43b}, 
$$
\aligned
\| g_2 \|_{L^2_{p}(\Mbf)}^2
& \lesssim 
\eeG^2 \big(\| g_2 \|_{L^2_{p}(\Mbf)} + \| h_2 \|_{L^2_q(\Mbf)} \big)^2 + \big(\| g_2 \|_{L^2_{p}(\Mbf)} + \eeG \, \| h_2 \|_{L^2_q(\Mbf)} \big) \| f \|_{L^2_{\oldp + 2}(\Mbf)}
\\
& \lesssim
\eeG^2 \big(\| f \|_{L^2_{\oldp + 2}(\Mbf)} + \| V \|_{L^2_{\oldq+1}(\Mbf)} \big)^2 +  \eeG \, \big(\| f \|_{L^2_{\oldp + 2}(\Mbf)} + \| V \|_{L^2_{\oldq+1}(\Mbf)} \big) \| f \|_{L^2_{\oldp + 2}(\Mbf)} + \| g_2 \|_{L^2_{p}(\Mbf)} \| f \|_{L^2_{\oldp + 2}(\Mbf)}.
\endaligned
$$
This establishes the first inequality of the proposition.  

\vskip.16cm

3. Similarly, from $\Jbf_{(g_1, h_1, f, V)}(u,Z) \leq \Jbf_{(g_1, h_1, f, V)}(u,0)$ we deduce the second inequality and the proof is completed. 
\end{proof}

\ese

\bse

\begin{proof}[Proof of Proposition~\ref{prop. estimate 4 order}]
 1. We follow here a strategy  found in~\cite{ChruDelay} 
    and~\cite{CarlottoSchoen}, 
  and we apply Theorem~\ref{propo.ellipticity}. 
The inequality within any compact region of $\Mbf$ being standard by a local elliptic regularity argument, we focus on any of the asymptotic ends, denoted below by $N = \RR^3 \setminus B_{R_1} \subset \RR^3$ for some $R_1>0$. 
Recall that we are working with the weighted unknown $\ut = r^{-\oldp}u$ and $\Zt = r^{-\oldq}Z$. Since the Einstein constraints 
 satisfy Douglis-Nirenberg's  ellipticity conditions (as pointed out as the end of Section~\ref{section Douglis-Nirenberg}), we can apply the interior regularity estimate in Theorem~\ref{propo.ellipticity}. We use the bounded domain $\Gamma(x) = B(x, r(x)/3)$ centered at any arbitrary point $x \in \RR^3 \setminus B_{2R_1}$, so that $\Gamma(x)  \subset N$.  With $d(x) = r(x)/3$ we obtain 
\bel{eq:theo44.a}
\aligned
& \sum_{i= 0}^{4} 
\Big( 
d(x)^i \big| \del^i \ut(x)\big| + d(x)^{4 + \alpha}[\del^4 \ut]_{\alpha, B_{d(x)/2}(x)}  
\Big)
+  \sum_{i= 0}^3d(x)^{i + 1} 
\Big( \big| \del^i\Zt(x)\big| + d(x)^{4 + \alpha}[\del^3 \Zt]_{\alpha,B_{d(x)/2}(x)} \Big)
\\
& 
\lesssim
\max_{B_{3d(x)/4}(x)} | \ut | 
+ r(x) \max_{B_{3d(x)/4}(x)} | \Zt |  
  + d(x)^4r(x)^{\oldp -3}  \sup_{B_{3d(x)/4}(x)} |f| + d(x)^\alpha[f]_{\alpha,B_{3d(x)/4}(x)}
\\
& \quad + d(x)^3r(x)^{\oldq  - 3}  \sup_{B_{3d(x)/4}(x)} |V| + d(x)\sup_{B_{3d(x)/4}(x)} | \del V| + d(x)^{1 
+ \alpha}[\del V]_{\alpha,B_{3d(x)/4}(x)}.
\endaligned
\ee

We next control the sup norm as follows. 
Given any sufficiently large radius $R > 0$, by applying Caccioppoli-Leray's inequality 
(which is based on an integration by parts and the inequality of arithmetic and geometric means)
to $\Delta_e \ut (R x)$ and $\Delta_e (\del \ut (R x))$ respectively, we have (for any $\eps > 0$)
$$
\aligned
\int_{B_{1/2} (x_0 / R)} | \nabla \ut (R x) |^2 \, dx 
& \lesssim 
\eps^{-2} \int_{B_1 (x_0 / R)} \ut^2 (R x) \, dx + \eps^2 \int_{B_1 (x_0 / R)} \big(\Delta \ut  (R x) \big)^2 \, dx,
\\
\int_{B_{1/4}(x_0 / R)} | \nabla^2 \ut (Rx) |^2 \, dx & \lesssim 
\eps^{-1} \int_{B_{1/2} (x_0 / R)} | \nabla \ut (R x) |^2 \, dx + \eps \int_{B_{1/2} (x_0 / R)}( \Delta (\nabla \ut (R x)) )^2 \, dx.
\endaligned
$$
It follows that 
$$
\| \ut (R x) \|^2_{W^{2, 2} (B_{1/4}(x_0 / R)) } \lesssim \eps^{-3} \int_{B_1 (x_0 / R)} \ut^2 (R x) \, dx + \eps \int_{B_1 (x_0 / R)} ( \Delta \ut (R x)) ^2 \, dx + \eps  \int_{B_1 (x_0 / R)}( \Delta (\nabla \ut (R x)) )^2 \, dx 
$$
and so, by the Sobolev embedding theorem,
$$
\max_{B_{1/4}(x_0 / R)} |\ut (R x) |^2 
\lesssim 
\eps^{-3} \int_{B_1 (x_0 / R)} \ut^2 (R x) \, dx 
+ \eps  (\max_{B_1 (x_0 / R)} | \Delta \ut (R x) |)^2 
+ \eps  \big(\max_{B_1 (x_0 / R)} \big| \Delta (\nabla \ut (R x) ) \big| \big)^2,
$$
which is equivalent to
\bel{eq:theo44.b}
\max_{B_{R/4}(x_0)} |\ut (x) |^2 
\lesssim 
{1 \over \eps^3 R^3} \int_{B_R (x_0)} \ut^2 (x) \, dx + \eps R^4 \big(\max_{B_R (x_0)} | \Delta \ut (x) | \big)^2 
+ \eps R^6 \big(\max_{B_R (x_0)} \big| \Delta (\nabla \ut (x) ) \big| \big)^2.
\ee
Similarly, we also have
\bel{eq:theo44.c}
\max_{B_{R/4}(x_0)} |\Zt (x) |^2 
\lesssim 
{1 \over \eps^3 R^3} \int_{B_R (x_0)} |\Zt|^2 (x) \, dx 
+ \eps R^4 \big(\max_{B_R (x_0)} | \Delta \Zt (x) | \big)^2 
+ \eps R^6 \big(\max_{B_R (x_0)} \big| \Delta (\nabla \Zt (x) ) \big| \big)^2.
\ee


Therefore, as long as $\eps$ is sufficiently small, we obtain from \eqref{eq:theo44.a}-\eqref{eq:theo44.c}
\bel{UZFV}
\aligned
& \| \ut\|_{C^{4,\alpha}_0(B_{3d(x)/4}(x))} + 
\| \Zt\|_{C_1 ^{3,\alpha}(B_{3d(x)/4}(x))}
\\
&
\lesssim 
r(x)^{-3/2} \| \ut\|_{L^2(B_{3d(x)/4}(x))} +  r(x)^{-1/2} \| \Zt\|_{L^2(B_{3d(x)/4}(x))} 
   +  \|f\|_{C^{0,\alpha}_{\oldp+1}(B_{3d(x)/4}(x))} + \|V\|_{C^{1,\alpha}_{\oldq}(B_{3d(x)/4}(x))}. 
\endaligned
\ee 
Since
$$
\aligned
& r(x)^{-3} \| \ut\|^2_{L^2(B_{3d(x)/4}(x))}
 \lesssim
d(x)^{-3}r(x)^{-2\oldp} \int_{B_{d(x)/2}(x)} |u(y)|^2\, dy
\\
& \lesssim
r(x)^{-2} \int_{B_{d(x)/2}(x)} |u(y)|^2r^{-3 + 2(1 - p)} \, dy
 \lesssim
r(x)^{-2} \|u\|^2_{H^2_{1-\oldp}}
\lesssim
r(x)^{- 2} \|(u,Z)\|^2_{H^2_{1-\oldp} \times H_{1,2-\oldq}},
\endaligned
$$
we have from Propositions~\ref{prop.key} and \ref{pro_integral}
$$
\aligned
r^{-3 / 2} \| \ut\|_{L^2(B_{3d(x)/4}(x))} 
&
\lesssim
r(x)^{-1} \big\|d\Gcal^*_{(g_1,h_1)}(u,Z)\big\|_{L^2_{3 - \oldp} (N)\times L^2_{3 - \oldq}(N)}
   \lesssim
r^{-1} \|(f,V)\|_{L^2_{\oldp + 2}(N) \times L^2_{\oldq + 1}(N)}.
\endaligned
$$
Similarly, we also have
$
 r(x)^{-1/2} \| \Zt\|_{L^2(B_{3d(x)/4}(x))} 
\lesssim
r^{-1} \|(f,V)\|_{L^2_{\oldp + 2}(N) \times L^2_{\oldq + 1}(N)}.
$
Taking this into account in \eqref{UZFV}, we establish that 
$$
\| \ut\|_{C^{4,\alpha}_0(B_{3d(x)/4}(x))} + 
\| \Zt\|_{C_1 ^{3,\alpha}(B_{3d(x)/4}(x))} \lesssim r^{-1} \|(f,V)\|_{L^2_{\oldp + 2}(N) \times L^2_{\oldq + 1}(N)} +  \|f\|_{C^{0,\alpha}_{\oldp+1}(B_{3d(x)/4}(x))} + \|V\|_{C^{1,\alpha}_{\oldq}(B_{3d(x)/4}(x))}
$$ 
and by the definition of $(g_2, h_2)$ in \eqref{def: gh} we arrive at 
$$
\aligned
\|g_2\|_{C^{2,\alpha}_{\oldp}(\Mbf)}  + \|h_2\|_{C^{2,\alpha}_{\oldq}(\Mbf)}
& 
\lesssim  
\|f\|_{L^2C^{0,\alpha}_{\oldp+2}(\Mbf)} + \| V \|_{L^2C^{1,\alpha}_{\oldq+1}(\Mbf)}. 
\endaligned
$$

\vskip.16cm 

2. In order to improve the previous estimates and cope with the dependency in $\eeG$, it suffices to re-apply the same arguments to each of the two sets of equations in \eqref{system1}. 
\end{proof}

\ese


\section{Technical arguments on the linearized  
Einstein constraints 
 (Section~\ref{section--6})}
\label{append-sec4}

\begin{proof}[Proof of Proposition \ref{proposition behavior at infinity}]
 The stated regularity was already explained in Proposition~\ref{propo.laplace - holder}. For the first statement in \eqref{decayPoisson}, recall that  
$$
w(x)  
= {1\over 4 \pi} \int_{\RR^3} {E(y) \over |x  - y|} \, dy
$$
and select an arbitrarily small $\eps \in (0,1)$. Since $E\in L^1(\RR^3)$, we can find a radius $R_\eps$ so large that 
$$
\int_{\RR^3 \setminus B_{R_\eps}(0)} |E| \, dy \leq \eps.
$$
\begin{subequations}
\label{eq.rrp-tout}
Writing 
\bel{eq.rrp}
\aligned
|x|w(x)  
= & {1 \over 4 \pi} \int_{B_{R_\eps}(0)} \frac{|x|E(y)}{|x - y|} \, dy + {1 \over 4 \pi} \int_{\RR^3\setminus B_{R_\eps}(0)} \frac{|x|E(y)}{|x - y|} \, dy
=: I_1^\eps + I_2^\eps
\endaligned
\ee 
and using $\lim_{|x| \to + \infty}\big(\sup_{y\in B_{R_\eps}(0)} |x| / |x - y| \big) = 1$, 
we obtain   
\bel{eq.r5}
\lim_{|x| \to +\infty} I_1^\eps
= {1 \over 4 \pi}  \int_{B_{R_\eps}(0)} E(y) \, dy.
\ee
On the other hand, we have
\bel{esI2}
\aligned
| I_2^\eps |
& \leq 
\int_{\RR^3\setminus B_{R_\eps}(0) \atop|x - y| \geq \sqrt{\eps} |x|} \frac{|x||E(y)|}{|x - y|} \, dy 
 + 
\int_{\RR^3\setminus B_{R_\eps}(0) \atop|x - y| \leq \sqrt{\eps} |x|} \frac{|x||E(y)|}{|x - y|} \, dy
\\
&
\lesssim 
\eps^{- 1/2} \int_{\RR^3\setminus B_{R_\eps}(0)} |E(y)| \, dy 
 +  
\int_{\RR^3\setminus B_{R_\eps}(0) \atop|x - y| \leq \sqrt{\eps} |x|} \frac{|x||y|^{- 3}}{|x - y|} \, dy
\\
& \lesssim
\eps^{- 1/2} \int_{\RR^3\setminus B_{R_\eps}(0)} |E(y)| \, dy 
 +  
\int_{\RR^3\setminus B_{R_\eps}(0) \atop|x - y| \leq \sqrt{\eps} |x|} \frac{|x|^{- 2}}{|x - y|} \, dy
\quad 
\lesssim 
\sqrt{\eps} + \eps, 
\endaligned
\ee 
\end{subequations}
 provided $\eps \in (0, 1/4)$, say.
Since $\eps$ can be chosen to be arbitrarily small, from \eqref{eq.rrp-tout} we obtain the first statement in \eqref{decayPoisson}. 
Finally, we write 
$$
|x|^2\nabla w (x) =  - {1 \over 4 \pi} \int_{\RR^3} |x|^2 (\nabla |x  - y| ) E(y) |x  - y|^{-2} \, dy
$$
and by observing that $|\nabla |x  -  y| | =1$, a similar analysis as above leads us to the second statement in \eqref{decayPoisson}. 
\end{proof}

\begin{proof}[Proof of Proposition \ref{proposition behavior at infinity k}] 
\bse
When $\theta = 1$ the desired result follows immediately from the proof of Proposition~\ref{proposition behavior at infinity}, so we treat here the interval $\theta \in (1, 2)$.  Since $\int_{\RR^3} E \, dy = 0$, we have 
$$
\aligned
|x|^\theta w(x) 
& = {1 \over 4 \pi} \int_{\RR^3} |x|^\theta \, {E(y) \over |x - y|} \, dy 
 = {1 \over 4 \pi} \int_{\RR^3} |x|^{\theta  - 1} \big( |x|  -  |x  -  y| \big) \, {E(y) \over |x - y|} \, dy, 
\endaligned
$$
and therefore given any radius $d>0$
$$
\aligned
4\pi |x|^\theta w(x) 
& = \int_{B_d(0)} \frac{|x|^{\theta  - 1}(|x|  -  |x  -  y|)E(y)}{|x - y|} \, dy
+ \int_{\RR^3 \setminus B_d(0) \atop |x|/2 \leq |x  -  y| \leq 3|x|/2} \frac{|x|^{\theta  - 1}(|x|  -  |x  -  y|) E(y)}{|x - y|} \, dy 
\\
& \quad + \int_{\RR^3 \setminus B_d(0) \atop|x  -  y| \geq 3|x|/2} \frac{|x|^{\theta  - 1}(|x|  -  |x  -  y|)E(y)}{|x - y|} \, dy
+ \int_{\RR^3 \setminus B_d(0) \atop|x  -  y| \leq |x|/2} \frac{|x|^{\theta  - 1}(|x|  -  |x  -  y|)E(y)}{|x - y|} \, dy
\\
&
   =: I_1(x) + I_2(x) + I_3(x) + I_4(x).
\endaligned
$$
Since
$
\frac{|x|^{\theta  - 1} | |x - y| - |x| |}{|x - y|} 
\leq \frac{|x|^{\theta  - 1}  |y|}{|x - y|},
$
we have for any $\eps \in  (0,1)$
$$
|I_1|(x) \lesssim  (1 + \sgn(|x| - 2d))  |x|^{\theta - 2}d \| E \|_{L^1(B_d(0))} + (1 - \sgn(|x| - 2d)) |x|^{\theta - \eps}  \| E \|_{C^{0,\alpha}_{2 + \eps}}, 
$$
where our convention is that $\sgn(0)= 0$.  
For the term $I_2$ we write 
$$ 
\aligned
|I_2(x)| 
& \lesssim |x|^{\theta - 2} \int_{\RR^3 \setminus B_d(0) \atop |x|/2 \leq |x  -  y| \leq 3|x|/2} |y||E(y)| \, dy
\\
&
 \lesssim |x|^{\theta - 2} \| r^{\theta + 2} E \|_{C^0(\RR^3 \setminus B_d(0))} \int_{\RR^3 \setminus B_d(0) \atop |y| \leq 2 |x|} |y|^{-\theta - 1} \, dy 
 \lesssim {1 \over 2 - \theta} \| r^{\theta + 2} E \|_{C^0(\RR^3 \setminus B_d(0))}, 
\endaligned
$$
since $\theta \in (1,2)$, 
while for the term  $I_3$
$$
\aligned
|I_3(x)|
& \lesssim |x|^{\theta  - 1} \int_{\RR^3 \setminus B_d(0) \atop|x  -  y| \geq 3|x|/2} |E(y)| \, dy
\lesssim |x|^{\theta  - 1} \| r^{\theta + 2} E \|_{C^0(\RR^3 \setminus B_{\max(d,|x|/2)}(0))} \int_{\RR^3 \setminus B_d(0) \atop |y| \geq 2|x|}  |y|^{-\theta - 2} \, dy
\\
&
\lesssim {1 \over \theta - 1} \| r^{\theta + 2} E \|_{C^0(\RR^3 \setminus B_{\max(d,|x|/2)}(0))}.
\endaligned
$$

Finally, we estimate $I_4$ by writing 
$$
\aligned
|I_4(x)|
& \lesssim 
|x|^{\theta  - 1} \int_{\RR^3 \setminus B_d(0) \atop|x  -  y| \leq |x|/2} {|y| |E(y)| \over |x - y|} \, dy
\lesssim
|x|^{\theta  - 1}\| r^{\theta + 2} E \|_{C^0(\RR^3 \setminus B_{\max(d, 3|x|/2)}(0))} \int_{\RR^3 \setminus B_d(0) \atop|x  -  y| \leq |x|/2} {|y|^{-\theta - 1} \over |x - y|} \, dy 
\\
& \lesssim
|x|^{\theta  - 1} \| r^{\theta + 2} E \|_{C^0(\RR^3 \setminus B_{\max(d, 3|x|/2)}(0))} \int_{\RR^3 \setminus B_d(0) \atop |x - y| \leq |x|/2} {(|x| - |x - y|)^{-\theta - 1} \over |x - y|}  \, dy
\\
&
 \lesssim \| r^{\theta + 2} E \|_{C^0(\RR^3 \setminus B_{\max(d, 3|x|/2)}(0))}.
\endaligned 
$$
Taking the above inequalities into account, we conclude that
\bel{eq:65et}
\aligned
|x|^\theta |w| 
& \lesssim (1 + \sgn(|x| - 2d))  |x|^{\theta - 2}d \| E \|_{L^1({B_d(0)})} + (1 - \sgn(|x| - 2d)) |x|^{\theta - \eps}  \| E \|_{C^{0,\alpha}_{2 + \eps}}
\\
& \quad + {1 \over 2 - \theta}  \| r^{\theta + 2} E \|_{C^0(\RR^3 \setminus B_d(0))} 
+ \theta \| r^{\theta + 2} E \|_{C^0(\RR^3 \setminus B_{\max(d,|x|/2)}(0))}. 
\endaligned
\ee
Combining this with the property
$$ 
\lim_{d \to 0^*}
\Big((1 + \sgn(|x| - 2d)) \,  |x|^{\theta - 2}d \| E \|_{L^1(B_d(0))} + (1 - \sgn(|x| - 2d)) |x|^{\theta - \eps}  \| E \|_{C^{0,\alpha}_{2 + \eps}} \Big) = 0,
$$
we find
$$
|x|^\theta |w| \lesssim (2 - \theta) \| r^{\theta + 2} E \|_{C^0(\RR^3)} + \theta \| r^{\theta + 2} E \|_{C^0(\RR^3 \setminus B_{|x|/2(0)})}.
$$
Hence, by Proposition~\ref{propo.laplace - holder} the inequality \eqref{eq:EEhat} holds. 

From \eqref{eq:65et}, one has 
$$
\limsup_{|x| \to +\infty} |x|^\theta |w| \lesssim \|r^{\theta + 2} E\|_{C^0(\RR^3 \setminus B_d(0))}
$$
for all $d > 0$.
Therefore, thanks to \eqref{eq:65cond} and by letting $d \to +\infty$, we obtain
$$
\lim_{|x| \to +\infty} |x|^\theta |w| = 0,
$$
as claimed. 
  
Now for any $d > 0$, let $\xi_d$ be a cut-off function which equals $0$ for all $|x| \leq d$ and equals $1$ for all $|x| \geq 2d$, and
$|\del^i \xi_d | \lesssim d^{-i}$ ($i = 1, 2, 3, \ldots$). 
By setting
$\wt_d := \xi_d w$, 
we find 
$$
- \Delta \wt_d = \xi_d E - w \Delta \xi_d + 2 \nabla w. \nabla \xi_d := \Et_d.
$$
Thanks to Proposition~\ref{propo.laplace - holder}, we have
$$
\| \wt_d \|_{C^{k + 2, \alpha}_\theta (\RR^3)} \lesssim \| \wt_d \|_{C^{0, \alpha}_\theta (\RR^3)} + \|\Et_d\|_{C^{k, \alpha}_{\theta + 2} (\RR^3)}.
$$
Observing also that
$$
\|\Et_d\|_{C^{k, \alpha}_{\theta + 2} (\RR^3)} \lesssim \|r^{\theta + 2 + i} \del^i E \|_{C^0(\RR^3 \setminus B_d(0))}, 
$$
 we obtain
$$
\| \wt_d \|_{C^{k + 2, \alpha}_\theta (\RR^3)} \lesssim \| \wt_d \|_{C^{0, \alpha}_\theta (\RR^3)} + \|r^{\theta + 2 + i} \del^i E \|_{C^0(\RR^3 \setminus B_d(0))}.
$$
Therefore, given $\eps \in (0,1)$, in view of \eqref{eq:65cond} and our previous conclusion $\lim_{|x| \to +\infty} |x|^\theta \, |w| = 0$, we choose $d$ to be sufficiently large so that
$
\| \wt_d \|_{C^{k + 2, \alpha}_\theta (\RR^3)} \lesssim \eps,
$
which gives us \eqref{eq:65decay}.
\ese
\end{proof}


\begin{proof}[Proof of Lemma~\ref{lemma-477}]
\begin{subequations}
Since $\LMcal$ has constant coefficients, without loss of generality we can assume that $y = 0$ and, after setting 
$\Mbb_{ij}(x) := \Mbb_{ij}(x,0)$, we only need to check
\be
I_{ij}(\phi) \coloneqq \int_{\RR^3} \Big( \Mbb_{ij}(x) \Delta \phi + \sum_{k=1,2,3} \Mbb_{jk}(x) \del_{ik}\phi \Big) \, dx 
= - \phi(0) \, \delta_{ij}
\ee
for all smooth and compactly supported functions $\phi: \RR^3 \to \RR$. By observing that $\Mbb \in {L^1_\loc(\RR^3 \setminus \{0 \})}$, it suffices to check that
$$
- \phi(0) \, \delta_{ik} = \lim_{a \to 0^+} I_{ik}(a),
\qquad
I_{ik}(a) \coloneqq \int_{|x| \geq a} \Big( \Mbb_{ki}(x) \Delta \phi 
+ \sum_{j=1,2,3} \Mbb_{jk}(x) \del_{ij}\phi \Big) \, dx.
$$
Using that $\phi$ has compact support and integrating by parts, we obtain 
$$
\aligned
I_{ik}(a) 
& = - \int_{|x| \geq a} \Big( \del_j \Mbb_{ki} \del_j \phi + \del_j \Mbb_{jk} \del_i \phi \Big) \, dx 
   - \int_{|x| =a} \Big( \Mbb_{ki} \del_j \phi {x_j \over |x|} + \Mbb_{jk} \del_i \phi {x_i \over |x|} \Big) \, d\omega
\\
& = \int_{|x| \geq a} \big( \Delta \Mbb_{ki} + \del_{ij} \Mbb_{jk} \big)\phi \, dx  + \int_{|x| = a} \Big( \del_j \Mbb_{ki} {x_j \over |x|} + \del_j \Mbb_{jk} {x_i \over |x|} \Big) \phi \, d\omega 
\\
&\quad
 -  \int_{|x| =a} \Big( \Mbb_{ki}{x_j \over |x|}  \del_j \phi  + \Mbb_{jk}  {x_i \over |x|}\del_i \phi  \Big) \, d\omega.
\endaligned
$$

Since $(x_j/|x|) \del_j \phi$ is bounded
 and since $|\Mbb(x)| \lesssim 1/|x|$, the last integral converges to zero when $a \to 0$. On the other hand, a straightforward calculation away from the singularity at $x=0$ gives us 
\be
\aligned
& \Delta \Mbb_{ki} + \del_{ij} \Mbb_{jk} = 0, 
\qquad
 {x_j \over |x|} \del_j \Mbb_{ki} + {x_i \over |x|} \del_j \Mbb_{jk} 
= - {3 \over 16\pi \,  |x|^2} \, \Bigg( \delta_{ik}  + {x_i x_k \over |x|^2} \Bigg).
\endaligned
\ee
Taking these identities into account, we then find
\begin{align*}
\lim_{a \to 0^+} I_{ik} & = - {3 \over 16\pi } \lim_{a \to 0^+} \int_{|x| = a} 
{1 \over |x|^2} 
\Bigg( \delta_{ik}+ {x_i x_k \over |x|^2} \Bigg) \phi(x) \, d\omega 
\\
& = - {3 \over 16\pi } \phi(0) \lim_{a \to 0^+} \int_{|x| = a} 
{1 \over |x|^2} 
\Bigg( \delta_{ik} + {x_i x_k \over |x|^2} \Bigg) \, d\omega 
- {3 \over 16\pi }  \lim_{a \to 0^+} \int_{|x| = a} 
{1 \over |x|^2} 
\Bigg( \delta_{ik} + {x_i x_k \over |x|^2} \Bigg) \big( \phi(x) - \phi(0) \big) \, d\omega 
\\
& = - {3 \, \phi(0) \over 16\pi } \lim_{a \to 0^+} \int_{|x| = a} 
{1 \over |x|^2} 
\Bigg( \delta_{ik} + {x_i x_k \over |x|^2} \Bigg) \, d\omega
 = - \phi(0) \, \delta_{ik}. 
  \qedhere
\end{align*}
\end{subequations}
\end{proof}

\end{document}